%% file: main.tex
\documentclass[11pt,oneside,a4paper]{article} 

\input{preamble}

\author{Thibault Langlais\textsuperscript{$*$} and Viktor F. Majewski\textsuperscript{$\dagger$}}
\date{}

\title{Degenerations of exotic Calabi--Yau metrics through Atiyah's flop}

\begin{document}

\maketitle

\footnotetext[1]{Humboldt-Universität, Berlin, Germany. Email: \texttt{thibault.langlais[]hu-berlin.de}.}
\footnotetext[2]{University of Waterloo, ON, Canada. Email: \texttt{viktor.majewski[]uwaterloo.ca}.}

\begin{abstract}
   We construct new families of complete Calabi--Yau metrics with maximal volume growth on the small resolutions of the conifold $\mathcal{Z} = \{z_1^2 + z_2^2 + z_3^2 + z_4^2 = 0\} \subset \C^4$. These metrics have tangent cone $\C \times (\C^2 / \Z_2)$ at infinity and are parametrised by their Kähler class. As the Kähler class degenerates, the metrics converge in the pointed Gromov--Hausdorff sense to a Calabi--Yau metric on $\mZ$ with an isolated conical singularity modelled on the Stenzel metric at the ordinary double point and tangent cone at infinity $\C \times (\C^2/\Z_2)$, thereby providing a new metric realisation of the Atiyah flop.
\end{abstract}

\tableofcontents

\setcounter{footnote}{0}
\renewcommand{\thefootnote}{\arabic{footnote}}

\input{section1_introduction}
\input{section2_srfk}
\input{section3_asympsolutions}
\input{section4_newfamilies}
\input{section5_singularlimit}
\input{section6_discussion}

\appendix
\input{appendixA_moseriteration}

\addcontentsline{toc}{section}{References}

\small 

\bibliographystyle{amsplain}
\bibliography{literature}

\end{document}

%% file: preamble.tex
\usepackage[english]{babel}	
\usepackage[T1]{fontenc}		
\usepackage[utf8]{inputenc}	
\usepackage{amsmath}
\usepackage{amsthm}	
\usepackage{amssymb}
\usepackage{mathrsfs}
\usepackage{amssymb}
\usepackage{mathabx}
\usepackage{paralist}
\usepackage{stmaryrd}
\usepackage{mathtools}
\usepackage{geometry}			
\usepackage[colorlinks = true,
            linktocpage = true,
            linkcolor = blue,
            urlcolor  = blue,
            citecolor = blue,
            anchorcolor = black]{hyperref}
\usepackage[hang,flushmargin,symbol]{footmisc}
\usepackage{appendix}

\numberwithin{equation}{section}            % Numbering of equations within Sections
\newtheoremstyle{bfnote}            %% definition of Theorem-like environment
	{}{}
	{\itshape}{}
	{\bfseries}{.}
	{ }
	{\thmname{#1}\thmnumber{ #2}\thmnote{ (#3)}}

\newtheoremstyle{bfremark}          %% definition of Remark environment
	{}{}
	{}{}
	{\bfseries}{.}
	{ }
	{\thmname{#1}\thmnumber{ #2}\thmnote{ (#3)}}

\newtheorem{manualtheoreminner}{Theorem}    % maualtheorem
\newenvironment{manualtheorem}[1]{%
  \renewcommand\themanualtheoreminner{#1}%
  \manualtheoreminner
}{\endmanualtheoreminner}
	
\theoremstyle{bfnote}
\newtheorem{thm}{Theorem}[section]          % Theorem environment
\newtheorem{prop}[thm]{Proposition}         % Proposition environment
\newtheorem{lem}[thm]{Lemma}                % Lemma environment
\newtheorem{cor}[thm]{Corollary}            % Corollary environment

\theoremstyle{definition}
\newtheorem{Def}[thm]{Definition}           % Definition environment

\theoremstyle{bfremark}
\newtheorem{rem}[thm]{Remark}               % Remark environment

\AtEndEnvironment{rem}{\hfill$\lozenge$}    % puts a rombus at the end of the remark 

\newcommand{\N}{\mathbb{N}}             % natural numbers
\newcommand{\Z}{\mathbb{Z}}             % integers
\newcommand{\R}{\mathbb{R}}             % real numbers
\newcommand{\C}{\mathbb{C}}             % complex numbers
\renewcommand{\H}{\mathbb{H}}           % quaternions

\renewcommand{\P}{\mathcal{P}}          % U(1)-bundle P over M or Z
\newcommand{\M}{\mathcal{M}}            % universal EH space
\newcommand{\mZ}{\mathcal{Z}}           % space Z
\newcommand{\mC}{\mathcal{C}}           % exceptional curve

\newcommand{\iq}{\mathrm{i}}        % quaternionic and complex i
\newcommand{\jq}{\mathrm{j}}        % quaternionic j
\newcommand{\kq}{\mathrm{k}}        % quaternionic k

\newcommand{\mr}{\boldsymbol{\mathrm{r}}}           % radius function coming from hk reduction
\newcommand{\mt}{\boldsymbol{\mathrm{t}}}           % variable t used in asymp solutions
\newcommand{\mA}{\boldsymbol{\mathrm{A}}}           % operator A used in asump solutions
\newcommand{\mB}{\boldsymbol{\mathrm{B}}}           % operator B used in asump solutions

\newcommand{\dist}{\mathrm{dist}}                   % distance
\newcommand{\diff}{\mathrm{d}}                      % exterior differential

\newcommand{\mo}{\mathrm{O}}                        % O for comparisons
\newcommand{\ohat}{\hat{\mathrm{O}}}                % O with a hat for uniform comparisons                  

\DeclareMathOperator{\vol}{vol}         % volume form
\DeclareMathOperator{\Vol}{Vol}         % volume of a region

\DeclareMathOperator{\Un}{U}            % Unitary group 
\DeclareMathOperator{\SO}{SO}           % Special orthogonal group
\DeclareMathOperator{\SU}{SU}           % Special unitary group

\DeclareMathOperator{\Real}{Re}             % prep for re-defining Re for real part
\renewcommand{\Re}{\Real}                   % re-defining \Re for real part
\DeclareMathOperator{\Imaginary}{Im}        % prep for re-defining Im for imaginary part
\renewcommand{\Im}{\Imaginary}              % re-defining \Im for imaginary part

\newcommand{\m}[1]{{\mathrm{#1}}}         % shortcut for mathrm  

%% file: section1_introduction.tex
%%%%%%%%%%%%%%%%%%%%%%%%%%%%%%%%%%
%%%%% Beginning of Section 1 %%%%%

    \section{Introduction}

As a result of Yau’s proof of the Calabi conjecture \cite{yau1978ricci}, Ricci-flat Kähler metrics on a compact complex manifold with trivial canonical bundle are known to be in one-to-one correspondence with its Kähler classes. In the non-compact setting however, the classification of Calabi--Yau metrics is more subtle, since asymptotic conditions need to be taken into account. A case of particular interest concerns Calabi--Yau metrics with maximal (Euclidean) volume growth. The best understood examples of such metrics are Asymptotically Conical (AC) Calabi--Yau metrics, which are asymptotic to a Calabi--Yau cone with smooth link: such metrics have been classified by Conlon and Hein when the rate of convergence to the tangent cone at infinity is polynomial \cite{conlon2024classification}. In general however, the tangent cone at infinity may have singular link. Examples include Joyce's construction of Quasi-Asymptotically Locally Euclidean (QALE) Calabi--Yau manifolds (which are asymptotic to $\C^m / \Gamma$ for some finite subgroup $\Gamma \subset \SU(m)$) via crepant resolutions, and the generalisation thereof to Quasi-Asymptotically Conical (QAC) metrics by Conlon--Degeratu--Rochon \cite{conlon2019quasi}.

Over the past decade, there has been considerable interest in Calabi--Yau metrics with maximal volume growth and singular tangent cone at infinity exhibiting a new type of asymptotics. Independent work of Li \cite{li2019new}, Conlon--Rochon \cite{conlon2021new} and Székelyhidi \cite{szekelyhidi2019degenerations} showed that for any $m \geq 3$, $\C^m$ admits infinitely many exotic Calabi--Yau metrics with maximal volume growth and singular tangent cone at infinity whose volume form coincides with the Euclidean one, thereby disproving a conjecture attributed to Tian. These metrics are related to certain holomorphic fibrations of $\C^m$ and are asymptotically `semi Ricci-flat’; that is, they are asymptotic to a metric whose fibres are themselves Calabi--Yau. The interest in such metrics is notably motivated by \emph{adiabatic problems} in the compact setting: they provide local models for Calabi--Yau manifolds collapsing along a holomorphic fibration, after an appropriate rescaling \cite{li2018gluing}. These works have since been extended by Conlon--Rochon \cite{conlon2023degenerations} and Yan \cite{yan2024gluing} to more general classes of affine varieties, providing a great number of so-called `warped-QAC' Calabi--Yau metrics, including some singular ones. 

In the present paper, we extend the construction of exotic Calabi--Yau metrics with maximal volume growth in a different direction to the non-affine setting, and exhibit new relations with the degeneration of Kähler classes and the algebraic notion of flop. We shall focus on the specific example of the small resolutions of the ordinary double point in dimension $3$, where we can make use of symmetries in order to give a very precise description of the new Calabi--Yau metrics using relatively elementary techniques. 

Our construction is also motivated by an adiabatic problem originating from resolution methods for special-holonomy orbifolds \cite{joyce2017new,majewski2025spin7orbifoldresolutions}. For sufficiently regular gluing data, these methods notably yield families of Calabi--Yau threefolds degenerating towards an orbifold along a path which jointly deforms the Kähler class and the underlying complex structure. On the other hand, for certain choices of singular gluing data these constructions are expected to produce degenerating families of \emph{conically singular} Calabi--Yau manifolds, and the families of metrics constructed in the present paper aim to be a local adiabatic model for such a singular limit. We shall discuss this interpretation and the role of our metrics in singular gluing problems in more detail in \S\ref{subsec: motivating problem}.

\paragraph{Main results.} In this article, we shall consider the local model of the ordinary double point singularity in dimension $3$, that is, the hypersurface
\begin{equation*}
    \mZ = \{ x_1^2 + x_2^2 + x_3^2 + x_4^2 = 0 \} \subset \C^4 .
\end{equation*}
It is well-known that it admits two distinct small resolutions. From the algebraic perspective, these resolutions can be constructed by blowing up $\mZ$ at the ordinary double point $o$, which yields a resolution $\m{Bl}_o(\mZ) \to \mZ$ whose exceptional divisor is isomorphic to $\mathbb{P}^1 \times \mathbb{P}^1$, and contracting either of the two rulings. The resulting small resolutions $\pi_\pm \colon  \mZ_\pm \to \mZ$ have exceptional curves $\mC_\pm \coloneqq \pi_\pm^{-1}(o) \cong \mathbb{P}^1$, and the birational map $\mZ_- \dashrightarrow \mZ_+$ induced by the resolution maps is the so-called `Atiyah flop'. 

Even though the resolutions are distinct, both $\mZ_-$ and $\mZ_+$ are biholomorphic to the total space of $\mathcal{O}_{\mathbb{P}^1}(-1) \oplus \mathcal{O}_{\mathbb{P}^1}(-1)$. Let us introduce the line bundles $\mathcal{L}_+ = p_+^* \mathcal{O}_{\mathbb{P}^1}(1) \to \mZ_+$ and $\mathcal{L}_- = p_-^* \mathcal{O}_{\mathbb{P}^1}(-1) \to \mZ_-$, where $p_\pm \colon \mZ_\pm \to \mathbb{P}^1$ are the natural bundle projections. Moreover, we let $\vol \in \Omega^6(\mZ^\m{reg})$ be the volume form defined as $\vol = c \Omega \wedge \overline{\Omega}$, where $\Omega \coloneqq x_1^{-1} \diff x_2 \wedge \diff x_3 \wedge \diff x_4$ is the canonical holomorphic volume form of $\mZ$ and $c \in \C^*$ is a constant (see \S\ref{subsec: Relation to the flop} for our specific choice of normalisation). This volume form can be lifted to smooth volume forms $\vol_\pm \coloneqq \pi_\pm^* \vol \in \Omega^6(\mZ_\pm)$ on the small resolutions.

With these notations, the first main result of our paper may be stated as follows:

\begin{manualtheorem}A
    \label{thm: A}
    There exist families of complete Calabi--Yau metrics, $\{\omega_{\m{CY},s}\}_{s > 0}$ on $\mZ_+$ and $\{\omega_{\m{CY},s}\}_{s < 0}$ on $\mZ_-$, such that the following hold for any $s > 0$:
    \begin{enumerate}[(i)]
        \item $[\omega_{\m{CY},s}] = 2\pi s \cdot c_1(\mathcal{L}_+) \in \m{H}^2(\mZ_+)$ and $[\omega_{\m{CY},-s}] = - 2\pi s \cdot  c_1(\mathcal{L}_-) \in \m{H}^2(\mZ_-)$.
        \item $\omega_{\m{CY},\pm s}^3 = \vol_\pm$.
        \item $\omega_{\m{CY},\pm s}$ have maximal volume growth and tangent cone at infinity $\C \times (\C^2 / \Z_2)$.
    \end{enumerate}    
\end{manualtheorem}

The above theorem states the existence of a second family of complete Calabi--Yau metrics with maximal volume growth on $\mZ_\pm$, in addition to the metrics constructed by Candelas and de la Ossa \cite{candelas1990comments}. We will see during the course of the article that the new metrics are invariant under a cohomogeneity-two action of $\Un(1) \times \SO(3)$ and are asymptotic to a semi Ricci-flat metric with Eguchi--Hanson fibres at infinity. In particular, the metrics $\omega_{\m{CY},s}$ are warped-QAC in the terminology of Conlon--Rochon. The high degree of symmetry will allow us to give a very detailed description of the asymptotic expansion of $\omega_{\m{CY},s}$, analogously to the exotic Calabi--Yau metric on $\C^3$ constructed by Li \cite{li2019new}.

The second main result of our paper concerns the behaviour of the metrics $\omega_{\m{CY},s}$ when the Kähler class degenerates:

\begin{manualtheorem}B
\label{thm: B}
    Let us fix base-points $o_\pm \in \mC_\pm \subset \mZ_\pm$ and let $o \in \mZ$ be the ordinary double point. Then as $s \to 0$, $\omega_{\m{CY},s}$ converges locally smoothly to an incomplete Calabi--Yau metric $\omega_{\m{CY},0}$ on $\mZ^{\m{reg}} \cong \mZ_+ \setminus \mC_+ \cong \mZ_- \setminus \mC_-$, where $\mC_\pm \cong \mathbb{P}^1$ are the exceptional curves. Moreover, 
    \begin{enumerate}[(i)]
        \item  $\omega_{\m{CY},0}$ is conically singular in the following sense: there exist neighbourhoods $U,V$ of $o$ in $\mZ$, $\nu > 0$ and a biholomorphism $\Phi \colon U \to V$ such that
        \begin{equation*}
            |\nabla^k_{\omega_{\m{Stz}}}(\Phi^* \omega_{\m{CY},0} - \omega_{\m{Stz}})| = \mo(\varrho^{\nu-k})
        \end{equation*}
        as $\varrho \to 0$ for all $k \geq 0$, where $\omega_{\m{Stz}}$ is the Stenzel cone metric on $\mZ$ and $\varrho$ the associated radius function.

        \item $(\mZ_\pm,\omega_{\m{CY},\pm s},o_\pm)$ converges to a complete length space $(\mZ,d_0,o)$ isometric to the completion of $(\mZ^{\m{reg}},\omega_{\m{CY},0})$ in the pointed Gromov--Hausdorff sense as $s \to 0$.

        \item The tangent cone at infinity of $\omega_{\m{CY},0}$ is $\C \times (\C^2 / \Z_2)$.
    \end{enumerate}
\end{manualtheorem}

The limiting metric $\omega_{\m{CY},0}$ consequently exhibits a transition between two different tangent cones: the local tangent cone at the ordinary double point $o$, and the singular tangent cone $\C \times (\C^2/\Z_2)$ at infinity. Extending the parameter $s$ through $0$ contracts $\mC_+$, passes through the singular Calabi--Yau metric on $\mZ$, and produces $\mC_-$ on the other small resolution, giving a metric realisation of Atiyah's flop. This picture is consistent with the description of the limits of Calabi--Yau metrics along degenerations of Kähler classes \cite{tosatti2009limits,collins2015kahler} and the Gromov--Hausdorff continuity through flops \cite{rong2011continuity,song2015conjecture} which are known in the compact setting.

\paragraph{Comparison with related constructions.} It is interesting to relate our metrics to Conlon--Rochon's warped-QAC metrics and Li's exotic Calabi--Yau metric on $\C^3$.

Conlon--Rochon previously constructed examples of conically singular warped-QAC Calabi--Yau metrics in \cite{conlon2023degenerations}, including a singular Calabi--Yau metric on $\mZ$ exhibiting a transition between the same two tangent cones (the Stenzel cone at $o$ and $\C \times (\C^2 / \Z_2)$ at infinity), which confirmed a prediction of Li \cite{li2018gluing}. From the geometrical perspective, they obtain the singular metric by considering a family of Calabi--Yau metrics degenerating along the smoothings of $\mZ$, instead of contracting the Kähler class of the small resolutions $\mZ_\pm$. To the authors' knowledge, the family of metrics constructed in the present paper is the first family of warped-QAC Calabi--Yau metrics realising a degeneration of the Kähler class, and the first construction of such metrics on non-affine varieties (although in \cite[Remarks 1.8 and 5.6]{conlon2023degenerations}, it is mentioned that their methods could also be used in the non-affine setting). On the other hand, from the analytical perspective, the construction of Conlon--Rochon uses very different machinery (w-QAC analysis) and does not allow for an asymptotic description of the metrics as precise as what we are able to achieve in the present paper, which is crucial for applications to singular gluing problems. While it seems natural to conjecture that the metric $\omega_{\m{CY},0}$ coincides with the singular Conlon--Rochon metric on $\mZ$, the uniqueness of a metric on $\mZ$ satisfying the conditions of the above theorem remains an open question. We will return to this problem in more detail in Section \ref{sec: Discussion and open questions}.

There is also an interesting parallel to be drawn between the family of metrics constructed in the present article and the exotic Calabi--Yau metric on $\mathbb{C}^3$ first constructed by Li \cite{li2019new} and subsequently used in his description of adiabatically Lefschetz fibrations \cite{li2018gluing}. Li considers the standard Lefschetz fibration $(z_1,z_2,z_3) \in \mathbb{C}^3 \mapsto z_1^2+z_2^2+z_3^2 \in \C$, whose non-singular fibres are smoothings of the $A_1$ surface singularity $\C^2 / \Z_2$ and carry Stenzel metrics (which coincide with Eguchi--Hanson metrics), while the central fibre is singular. His metric is the local model which appears near a nodal fibre after rescaling a family of collapsing Calabi--Yau manifolds along a Lefschetz K3 fibration. By comparison, our family is governed by a signed resolution parameter rather than a smoothing parameter: its vanishing contracts the exceptional curve, while its sign distinguishes the two small resolutions. The metrics constructed here are therefore intended as local models for adiabatic resolutions of codimension-four Calabi--Yau orbifold singularities when the resolution data degenerates; see \S\ref{subsec: motivating problem}.

\paragraph{Outline of the paper.} By means of closing this introduction, we discuss our strategy of proof and emphasize the main novelties of our approach on the technical level.

Our starting point is Kronheimer's quotient construction of the hyperkähler ALE instantons, which allows us to construct natural families of semi Ricci-flat Kähler metrics $\omega^{\m{srf}}_s$ on the small resolutions of $\mZ$ in Section \ref{sec: sRfK metrics}, where $s \in \R$ essentially represents the (signed) area of the exceptional curves. Although the semi Ricci-flat metrics are asymptotically Calabi--Yau, this ansatz has two coupled defects. The first is that the Ricci potential decays too slowly at infinity for the non-compact Monge--Amp\`ere theory to apply directly. While this is a common phenomenon which also occurs in \cite{conlon2021new,li2019new,szekelyhidi2019degenerations}, there is in our case an additional tension between two requirements: firstly, the need to correct the semi Ricci-flat ansatz at infinity to a metric with fast decay of the Ricci potential; and secondly, the necessity to bound the correction terms uniformly with respect to the parameter $s$. The derivation of appropriate corrections terms fixing this defect will be carried out in Section \ref{sec: Asymptotic solutions of the MA equation}, where we exhibit asymptotic solutions whose decay properties are uniformly under control. This allows us to prove an effective version of Theorem \ref{thm: A} (Theorem \ref{thm: the new complete metrics}).

The second major defect of the semi Ricci-flat ansatz is that the Ricci curvature of $\omega^{\m{srf}}_s$ blows up near the exceptional curve as $s \to 0$, and in particular $\omega^{\m{srf}}_s$ fails to be approximately Calabi--Yau near the ordinary double point. This second defect is addressed in Section \ref{sec: New families of CY metrics}, where we interpolate the corrected semi Ricci-flat ansatz with an appropriate family of Candelas--de la Ossa metrics near the exceptional curves. We then establish a uniform control on the geometry of the resulting family of metrics, and in particular show that it satisfies a uniform version of Hein's $\mathsf{SOB}(6)$-property, which crucially allows us to control the relevant Sobolev constants in the limit where $s \to 0$. Hein's method for solving the Monge--Ampère equation on noncompact manifolds \cite{hein2010gravitational} can consequently be applied to the whole family of metrics with constants independent of $s$. The key point of working with a well-designed family of metrics is precisely this uniformity: the relevant constants in the Moser iteration and the $C^2$-estimates are tracked quantitatively rather than being allowed to depend on the individual member of the family (see also Appendix \ref{app: quantitative moser iteration}).

In Section \ref{sec: The singular limit}, these uniform estimates are used to control the family of exotic Calabi--Yau metrics as the Kähler class degenerates. We establish the smooth convergence on the regular locus (Theorem \ref{thm: convergence to omegaCY0}), pointed Gromov--Hausdorff convergence (Theorem \ref{thm: metric completion}) and prove polynomial convergence of the limiting metric to the Stenzel cone at the ordinary double point using the recent work of Zhang \cite{zhang2024polynomial} concerning the singular limits of non-polarised sequences of Calabi--Yau manifolds. 

Finally, Section \ref{sec: Discussion and open questions} discusses in more detail the adiabatic problem motivating the present construction and the role that the metrics $\omega_{\m{CY},s}$ are expected to play in the corresponding singular gluing problem. We also mention a number of related open questions: the relation with the Conlon--Rochon metrics, possible extensions of our construction, and the importance of such extensions for more general singular gluing problems.

\paragraph{Acknowledgements.} The authors would like to thank Jason Lotay and Spiro Karigiannis for valuable discussion, and Thomas Walpuski for providing helpful comments on an earlier draft of this paper. T.L. is supported by a fellowship from the Alexander-von-Humboldt Foundation.

%%%%% End of Section 1 %%%%%
%%%%%%%%%%%%%%%%%%%%%%%%%%%%

%% file: section2_srfk.tex
%%%%%%%%%%%%%%%%%%%%%%%%%%%%%%%%%%%%
%%%%% Beginning of section 2 %%%%%%%

    \section{Semi Ricci-flat Kähler metrics via hyperkähler reduction}   \label{sec: sRfK metrics}

In this section, we use Kronheimer's construction of the hyperkähler ALE instantons via hyperkähler reduction \cite{kronheimer1989construction} in order to construct a natural family of semi Ricci-flat Kähler metrics on $\mZ$ and its small resolutions $\mZ_\pm$, which are asymptotically Calabi--Yau. 

In \S\ref{subsec: HK reduction}, we describe the construction of a `universal Eguchi--Hanson space' $\M$ via hyperkähler reduction. In \S\ref{subsec: Complex point of view}, we use Kähler reduction to construct a $1$-parameter family of complex threefolds $\{\mZ_s\}_{s \in \R}$ foliating $\M$, and endow them with a natural family of semi Ricci-flat Kähler metrics $\{\omega_s^{\m{srf}}\}_{s \in \R}$ in \S\ref{subsec: A family of sRfK metrics}. We then study the equivariant aspects of these constructions in \S\ref{subsec: Equivariant aspects}, where we prove that the whole data is invariant under a natural action of the Lie group $\Un(2)$. Finally, in \S\ref{subsec: Relation to the flop} we show that the threefolds $\mZ_s$ can be identified with either of the small resolutions $\mZ_\pm$ depending on the sign of $s$ (when $s \neq 0$), while $\mZ_0$ is naturally biholomorphic to the regular part $\mZ^{\m{reg}}$ of $\mZ$.

    \subsection{The universal Eguchi--Hanson space}    \label{subsec: HK reduction}

Let $1,\iq,\jq,\kq$ be the standard basis for the space of quaternions $\H$, and let $\H^2$ be endowed with the hyperkähler complex structures $I,J,K \in \m{End}(\H^2)$ corresponding to the multiplication by $\iq,\jq,\kq$ \emph{on the left}. The standard flat metric on $\H^2$ may be written as
\begin{equation}
    g^{\m{st}} = \Re(\diff q_1 \diff \overline{q}_1 + \diff q_2 \diff \overline{q}_2) .
\end{equation}
The associated triple of standard Kähler forms defines an element $\underline{\omega}^{\m{st}} \in \bigwedge^2 (\H^2)^* \otimes \Im(\mathbb{H})^*$ defined as
\begin{equation}    \label{eq: omegazeta2}
    \langle \underline{\omega}^{\m{st}}, \lambda \rangle = \omega^{\m{st}}_\lambda = g^{\m{st}}(\lambda \cdot, \cdot) =  \Re(\lambda \diff q_1 \diff \overline{q}_1) + \Re(\lambda \diff q_2 \diff \overline{q}_2), ~~~\forall \lambda \in \Im(\H) .
\end{equation}
For simplicity, we will also denote by $\omega^{\m{st}}_{\iq} = \omega^{\m{st}}_1$, $\omega^{\m{st}}_{\jq} = \omega^{\m{st}}_2$ and $\omega^{\m{st}}_{\kq} = \omega^{\m{st}}_3$ the Kähler forms associated with $I$, $J$, $K$.

The Lie group $\Un(1)$ acts (on the right) on $\H^2$ by
\begin{equation}    \label{eq: Action of U(1)}
    (q_1,q_2) \cdot e^{\iq \theta}= (q_1 e^{\iq \theta}, q_2 e^{\iq \theta}).
\end{equation}
This action is generated by the vector field $\xi$ defined by
\begin{equation}        \label{eq: genvfxi}
    \xi(q_1,q_2) \coloneqq \partial_\theta =  (q_1 \iq,q_2\iq) .
\end{equation}
This action is free except at $(q_1,q_2) = (0,0)$. Hence we can see $\P = \H^2 \setminus \{ (0,0) \}$ as a circle bundle over $\M = \P / \Un(1)$, which is diffeomorphic to $(0,\infty) \times \mathbb{P}^3$. Moreover it is clear that the action of $\Un(1)$ preserves $g^{\m{st}}$ and $\underline{\omega}^{\m{st}}$. This action is even trihamiltonian:

\begin{lem}     \label{lem: Moment map mu}
    The action of $\Un(1)$ on $\H^2$ is trihamiltonian, with moment map given by
    \begin{equation*}
        \mu^{\m{st}}(q_1,q_2) \coloneqq - \frac{1}{2} \Im(q_1 \iq \overline{q}_1 + q_2 \iq \overline{q}_2) .
    \end{equation*}
    That is, if we write $\mu^{\m{st}}_\lambda = \langle \mu^{\m{st}}, \lambda \rangle$ for $\lambda \in \Im(\H)$, we have
    \begin{equation*}
        \diff\mu^{\m{st}}_\lambda = \omega^{\m{st}}_\lambda(\xi,\cdot) = g^{\m{st}}(\lambda \xi, \cdot) .
    \end{equation*}
\end{lem}

\begin{proof}
    The map $\mu^{\m{st}}$ is evidently invariant under the action of $\Un(1)$. Moreover, we can easily calculate that
    \begin{equation*}
        \diff \mu^{\m{st}}_\lambda = - \langle \Im(\diff q_1 \iq \overline{q}_1 + \diff q_2 \iq \overline{q}_2), \lambda \rangle = \Re(\lambda q_1 \iq \diff \overline{q}_1 + \lambda q_2 \iq \diff \overline{q}_2) .
    \end{equation*}
    Comparing with \eqref{eq: omegazeta2} and \eqref{eq: genvfxi} this yields the desired result.
\end{proof}

\begin{lem}     \label{lem: dmuhor}
    If $(q_1,q_2) \neq (0,0)$ then $(\diff \mu^{\m{st}})_{(q_1,q_2)}$ is surjective, and the kernel satisfies:
    \begin{equation*}
        \ker ((\diff \mu^{\m{st}})_{(q_1,q_2)}) = \{ I\xi, J\xi, K\xi \}^\perp .
    \end{equation*}
    Moreover, for any $\lambda \in \Im(\H)$ we have 
    \begin{equation*}
        \diff \mu^{\m{st}} (\lambda \xi) = \mr^2 \lambda
    \end{equation*}
    where $\mr^2 = g^{\m{st}}(\xi,\xi) = |q_1|^2 + |q_2|^2$.
\end{lem}

\begin{proof}
    For $\lambda^\prime \in \Im(\H)$, the dual vector of $\diff \mu^{st}_{\lambda^\prime}$ with respect to the metric $g^{\m{st}}$ is $\lambda^\prime \xi$ and $I\xi$, $J\xi$ and $K\xi$ are independent whenever $\xi \neq 0$, which proves that $\mu^{\m{st}}$ is a submersion if $(q_1,q_2) \neq (0,0)$. In addition, for $\lambda \in \Im(\H)$,
    \begin{align*}
        \langle \diff \mu^{\m{st}}(\lambda \xi) , \lambda^\prime \rangle & = g^{\m{st}} (\lambda^\prime \xi, \lambda \xi) = g^{\m{st}} (\overline{\lambda} \lambda^\prime \xi, \xi)  = \Re(\overline{\lambda} \lambda^\prime (|\xi_1|^2 + |\xi_2|^2) = \mr^2 \langle \lambda, \lambda^\prime \rangle
    \end{align*}
    which yields the second statement of the lemma.
\end{proof}

We shall denote by $\mathcal{H} = \{ \lambda \xi \mid \lambda \in \Im(\H)\}$ the horizontal distribution on $\mu^{\m{st}}\colon \P\rightarrow \m{Im}(\mathbb{H})$. Its orthogonal space is the kernel of $\diff\mu^{\m{st}}$ and it has a further orthogonal splitting $\ker(\diff\mu^{\m{st}}) = \mathcal{V} \oplus \mathcal{E}$, where $\mathcal{V}$ is spanned by $\xi$ and $\mathcal{E}$ is its $g^{\m{st}}$-orthogonal complement. Thus we obtain a splitting
\begin{equation*}
    T \P = \mathcal{V} \oplus \mathcal{E} \oplus \mathcal{H} .
\end{equation*}
Note that the quotient $\M$ inherits a metric induced by $g^{\m{st}}$, which makes the projection map $\P \rightarrow \M$ a Riemannian fibration. Since the above decomposition of $T\P$ is equivariant under the action of $\Un(1)$, there is an induced splitting $T \M = \mathcal{E} \oplus \mathcal{H}$. Moreover, as the map $\mu^\m{st}$ is invariant under the action of $\Un(1)$, it descends to a smooth fibration $\mu : \M \to \Im(\H)$ such that $\mathcal{E}$ is the tangent space of the fibres of $\mu$ and $\mathcal{H}$ is its orthogonal complement with respect to the quotient metric on $\M$. 

By hyperkähler reduction, the fibres $M_\lambda = \mu^{-1}(\lambda)$ of $\mu$ are hyperkähler $4$-manifolds, equipped with the metric $\hat{g}^\lambda$ and the triple of $2$-forms $\underline{\hat{\omega}}^\lambda$ obtained by restriction of $g^\m{st}$ and $\underline{\omega}^\m{st}$ to the distribution $\mathcal{E}$. When $\lambda \neq 0$, $M_\lambda$ is diffeomorphic to the total space of $T^*\mathbb{S}^2$ and the hyperkähler metric $\hat{g}^\lambda$ is the well-known Eguchi--Hanson metric, which is a complete ALE metric asymptotic to $\H/\Z_2$. Loosely speaking, the modulus $|\lambda|$ represents the size of the zero section $\mathbb{S}^2 \subset M_\lambda$. When $\lambda = 0$, $M_0 \cong (\H \setminus \{0\}) / \Z_2$ and the metric $\hat{g}^0$ is flat. The fibration $\mu : \M \to \Im(\H)$ can be thought of as a `universal Eguchi--Hanson space', e.g. the moduli space of ALE hyperkähler structures on the Eguchi--Hanson space \cite{kronheimer1989construction}.

We shall denote by $\hat{\underline{\omega}} \in \Gamma(\bigwedge^2 \mathcal{E}^* \otimes \Im(\H)^*)$ the section obtained by restriction of $\underline{\omega}^\m{st}$ to the distribution $\mathcal{E}$. We can consider $\hat{\omega}$ as a $\Un(1)$-invariant $2$-form on $\P$, or equivalently a $2$-form on the quotient $\M$, by requiring that $\hat{\omega}(v,\cdot) = 0$ for any section $v \in \Gamma(\mathcal{V} \oplus \mathcal{H})$. In particular, the restriction of $\hat{\omega}$ to the fibres of $M_\lambda$ coincides by definition with the hyperkähler triple $\hat{\omega}^\lambda$. We shall also use the notation $\hat{\omega}_1 = \hat{\omega}_{\iq}$, $\hat{\omega}_2 = \hat{\omega}_{\jq}$ and $\hat{\omega}_3 = \hat{\omega}_{\kq}$ for simplicity. For later purposes, we record the following result:

\begin{lem}     \label{lem: Decomposition of Kahler forms}
    If $(ijk)$ is a direct permutation of $(123)$, we have
    \begin{equation*}
        \omega_i^{\m{st}} = \hat{\omega}_i - \iq \mathfrak{a} \wedge \diff\mu^{\m{st}}_i + \frac{1}{\mr^2} \diff\mu^{\m{st}}_j \wedge \diff\mu^{\m{st}}_k
    \end{equation*}
    where $\mathfrak{a} = \mr^{-2} g^{\m{st}}(\xi, \cdot) \iq \in \Omega^1(\P,\iq \R) \cong \Omega^1(\P,\mathfrak{u}(1))$ is the connection $1$-form on $\P$, seen as an $\Un(1)$-bundle over $\M$, associated with the metric $g^{\m{st}}$.
\end{lem}

\begin{proof}
    Since $\mathcal{E}$ and $\mathcal{V} \oplus \mathcal{H}$ are orthogonal and invariant under the hyperkähler structure, and the restriction of $\omega_i^{\m{st}}$ to $\mathcal{E}$ coincides with $\hat \omega_i$ by definition, we only need to determine the restriction of $\omega_i^{\m{st}}$ to $\mathcal{V} \oplus \mathcal{H}$. The vector space $\mathcal{V} \oplus \mathcal{H}$ has a global orthonormal frame $\xi_0 = \tfrac{1}{\mr}\xi$, $\xi_1 = \tfrac{1}{\mr} I\xi$, $\xi_2 = \tfrac{1}{\mr} J\xi$ and $\xi_3 = \tfrac{1}{\mr} K\xi$, and if $\xi^0,\xi^1,\xi^2,\xi^3$ is the dual co-frame (seen as $1$-forms on $\P$ by specifying $\left. \xi^\ell \right|_{\mathcal{E}} = 0$), we must have
    \begin{equation*}
        \left. \omega^{\m{st}}_i \right|_{\mathcal{V} \oplus \mathcal{H}} = \xi^0 \wedge \xi^i + \xi^j \wedge \xi^k .
    \end{equation*}
    On the other hand, Lemma \ref{lem: dmuhor} shows that $\xi^0 = \frac{1}{\mr}g^{\m{st}}(\xi,\cdot) = -\mr \iq \mathfrak{a}$ and $\xi^\ell = \tfrac{1}{\mr} \diff \mu^{\m{st}}_\ell$ for all $\ell = 1,2,3$, which yields the result.
\end{proof}

    \subsection{Complex point of view and Kähler reduction}     \label{subsec: Complex point of view}

There is an explicit identification $\C^2 \cong \mathbb{H}$ given by $(z_1,z_2) \mapsto z_1 + z_2 \jq$. This gives an isomorphism $\C^4 \cong \C^2 \times \C^2 \cong \mathbb{H}^2$. Using complex coordinates $(z_1,z_2,z_3,z_4)$ on $\C^4$ with $q_1 = z_1 + z_2 \jq$ and $q_2 = z_3 + z_4 \jq$, the action \eqref{eq: Action of U(1)} of $\Un(1)$ can be written as
\begin{equation*}
    (z_1,z_2,z_3,z_4) \cdot e^{\iq \theta} = (e^{\iq \theta} z_1, e^{-\iq \theta} z_2, e^{\iq \theta} z_3, e^{-\iq \theta} z_4) .
\end{equation*}
The trihamiltonian moment map $\mu^{\m{st}} \colon \H^2 \rightarrow \Im(\mathbb{H})$ splits into two parts, $h^{\m{st}} \colon \C^4 \rightarrow \R$ and $\zeta^{\m{st}} \colon \C^4 \rightarrow \C$. Indeed,
\begin{align*}
    \mu^{\m{st}}(z_1,z_2,z_3,z_4) & = - \frac{1}{2} \Im((z_1 + z_2 \jq)\iq(\overline{z}_1- \jq \overline{z}_2) + (z_3 + z_4 \jq)\iq(\overline{z}_3- \jq \overline{z}_4)) \\
        & = \frac{1}{2} (|z_2|^2 + |z_4|^2 - |z_1|^2 - |z_3|^2) \iq + \iq( z_1 z_2 + z_3 z_4) \jq
\end{align*}
and therefore $\mu^{\m{st}} = h^{\m{st}} \iq + \zeta^{\m{st}} \jq$ where
\begin{align}
    h^{\m{st}}(z_1,z_2,z_3,z_4) & \coloneqq \frac{1}{2} (|z_2|^2 + |z_4|^2 - |z_1|^2 - |z_3|^2) \in \R, \\
    \zeta^{\m{st}}(z_1,z_2,z_3,z_4) & \coloneqq \iq (z_1z_2 + z_3 z_4) \in \C = \R \oplus \iq \R  .
\end{align}
The maps $h^\m{st} \colon \C^4 \to \R$ and $\zeta^\m{st} \colon \C^4 \to \C$ are $\Un(1)$-invariant, and therefore induce smooth maps $h \colon \M \rightarrow \R$ and $\zeta \colon \M \rightarrow \C$ on the quotient. For any $s \in \R$, we shall denote by $\mZ_s \subset \M$ the level set $h^{-1}(s)$, and by $\zeta_s \colon \mZ_s \to \C$ the map induced by restriction of $\zeta$. Note that the map $h^{\m{st}}$ is a moment map for the action of $\Un(1)$ on $\C^4$ for the standard complex structure $I$ (which corresponds to left multiplication by $\iq$ under our identification $\C^4 \cong \H^2$) and the associated Kähler form
\begin{equation*}
    \omega^{\m{st}} \coloneqq \omega_{\iq}^{\m{st}} = \frac{\iq}{2} (\diff z_1 \wedge \diff \overline{z}_1 + \diff z_2 \wedge \diff \overline{z}_2 + \diff z_3 \wedge \diff \overline{z}_3 + \diff z_4 \wedge \diff \overline{z}_4) .
\end{equation*}
Using Kähler reduction, we can give $\mZ_s$ the structure of a Kähler manifold, endowed with the complex structure $I$ and the Kähler form $\omega_s^\m{red}$ induced by $(I,\omega^\m{st})$. Moreover, the horizontal distribution $\mathcal{H}$ can be decomposed as $\mathcal{H} = \R \cdot I \xi \oplus \widetilde{\mathcal{H}}$ where $\widetilde{\mathcal{H}} = \mathrm{span}\{J\xi,K\xi\} =\ker (\diff h)$, and this induces a splitting $T\mZ_s = \widetilde{\mathcal{H}}_s \oplus \mathcal{E}_s$.

\begin{prop}     \label{prop: Properties of reduction}
    For any $s \in \R$, let us denote by $(g^{\m{red}}_s,I_s,\omega^{\m{red}}_s)$ the Kähler structure on $\mZ_s$ obtained by Kähler reduction of $(g^{\m{st}},I,\omega^{\m{st}})$. Then the following properties hold:
    \begin{enumerate}[(i)]
        \item The map $\zeta_s \colon \mZ_s \rightarrow \C$ induced by $\zeta^{\m{st}}$ is a holomorphic fibration.
        
        \item With respect to the splitting $T\mZ_s = \widetilde{\mathcal{H}}_s \oplus \mathcal{E}_s$ into horizontal and vertical components, the Kähler form $\omega^\m{red}_s$ decomposes as
        \begin{equation*}
            \omega^{\m{red}}_s = \hat\omega_{\iq} + \frac{\iq}{2\mr_s^2} \diff \zeta_s \wedge \diff \overline{\zeta}_s
        \end{equation*}
        where $\hat\omega_{\iq}$ is the vertical Kähler form on the Eguchi--Hanson fibres associated with the complex structure induced by $I$, and $\mr^2_s$ the smooth function on $\mZ_s$ induced by the $\Un(1)$-invariant function $\mr^2 = |q_1|^2 + |q_2|^2$.

        \item For any $s \in \R$, 
        \begin{equation*}
            \omega_s^\m{red} =  s \iq \Theta_s + \frac{\iq}{2} \partial \overline{\partial}(\mr_s^2)
        \end{equation*}
        where $\Theta_s \coloneqq \diff \mathfrak{a}_s \in \Omega^{1,1}(\mZ_s,\iq \R)$ is the curvature of the connection $\mathfrak{a}_s \in \Omega^1(\P_s,\iq\R)$ on $\P_s$ obtained by restriction of $\mathfrak{a}$.

        \item If $s \neq 0$, $g^\m{red}_s$ is complete.
    \end{enumerate}
\end{prop}

\begin{proof}
    The first point is immediate, since $\zeta^{\m{st}}$ is holomorphic on $\H^2 \cong \C^4$ and the map $\mu = h \times \zeta \colon \M \rightarrow \R \times \C$ is a fibration, so that $\zeta$ is a fibration on the level sets of $h$. The second point follows from Lemma \ref{lem: Decomposition of Kahler forms}, since
    \begin{align*}
        \omega^{\m{st}}_{\iq} & = \hat \omega_{\iq} - \iq \mathfrak{a} \wedge \diff \mu^{\m{st}}_1 + \frac{1}{\mr^2} \diff \mu^{\m{st}}_2 \wedge \diff \mu^{\m{st}}_{3} \\
            & = \hat \omega_{\iq} - \iq  \mathfrak{a} \wedge \diff h^{\m{st}} + \frac{1}{\mr^2} \diff \Re(\zeta^{\m{st}}) \wedge \diff \Im(\zeta^{\m{st}}) \\
            & = \hat \omega_{\iq} - \iq  \mathfrak{a} \wedge \diff h^{\m{st}} + \frac{\iq}{2 \mr^2} \diff \zeta^\m{st} \wedge \diff \overline{\zeta^\m{st}}
    \end{align*}
    and $\diff h^{\m{st}} = 0$ on the level sets of $h^{\m{st}}$. 

    To prove point (iii), recall that $\tfrac{\iq}{2} \partial \overline{\partial} (\mr^2_s) = \tfrac{1}{4} \diff \diff^c (\mr^2_s)$ where $\diff^c F = \diff F \circ I^{-1} = - \diff F \circ I$ for a function $F$. The horizontal lift of $\frac{1}{4} \diff^c(\mr^2_s)$ on $\P_s$ is
    \begin{equation*}
        \frac{1}{4} \diff^c(\mr^2) + \frac{1}{4} \diff^c(\mr^2)(\xi) \cdot \iq \mathfrak{a}_s
    \end{equation*}
    and
    \begin{equation*}
        \frac{1}{4} \diff^c(\mr^2) (\xi) = - \frac{1}{4} \diff(\mr^2)(I\xi) = -\frac{1}{2}\Re(q_1 \iq \overline{q}_1\iq + q_2 \iq \overline{q}_2 \iq) = - s .
    \end{equation*}
    Thus it follows that the horizontal lift of $\tfrac{\iq}{2} \partial \overline{\partial}(\mr_s^2) = \frac{1}{4} \diff \diff^c(\mr^2_s)$ on $\P_s$ is
    \begin{equation*}
        \diff (\frac{1}{4} \mathrm{d}^c(\mr^2) - \iq s \mathfrak{a}_s ) = \frac{1}{4} \diff\diff^c(\mr^2) - \iq s \diff \mathfrak{a}_s = \omega^{\m{st}} - \iq s \diff \mathfrak{a}_s .
    \end{equation*}
    This equality implies $\frac{\iq}{2} \partial \overline{\partial}(\mr^2_s) = \omega^\m{red}_s - s \iq \Theta_s$ on $\mZ_s$.
    
    Finally, if $s \neq 0$, $\P_s \subset \H^2$ is a closed submanifold and thus $\P_s$ is complete with respect to the standard metric $g^\m{st}$, which implies that $\mZ_s = \P_s / \Un(1)$ is complete, equipped with the quotient metric $g^\m{red}_s$.
\end{proof}

\begin{rem}
    In terms of holomorphic line bundles, the connection $\mathfrak{a}_s$ can be seen as the Chern connection of the holomorphic line bundle $\mathcal{L}_s = \P_s \times_{\Un(1)} \C$ for the unique hermitian metric $H_s$ on $\mathcal{L}_s$ such that $\P_s \subset \mathcal{L}_s$ is the sub-bundle of unit vectors.
\end{rem}

    \subsection{A family of semi Ricci-flat Kähler metrics}     \label{subsec: A family of sRfK metrics}

Henceforth, we shall denote by $\hat \omega = \hat \omega_{\iq}$ the family of vertical Kähler forms on the Eguchi--Hanson fibres, and by $\hat \Omega = \hat \omega_{\jq} + \iq \hat \omega_{\kq}$ the family of vertical holomorphic volume forms (with respect to the complex structure $I$). Using the Ehresmann connection $T\M = \mathcal{H} \oplus \mathcal{E}$, we can regard $\hat \omega$ and $\hat \Omega$ as $2$-forms on $\M$. Moreover, since the decomposition $T\mZ_s = \widetilde{\mathcal{H}}_s \oplus \mathcal{E}_s$ is orthogonal and invariant by $I$, the restriction $\hat \omega_s$ of $\hat \omega$ to $\mZ_s$ defines a $(1,1)$-form, and the restriction $\hat \Omega_s$ of $\hat \Omega$ defines a $(2,0)$-form on $\mZ_s$. In particular, for any $s \in \R$, we can define the $(3,0)$-form $\Omega_s \in \Omega^{3,0}(\mZ_s)$ as
\begin{equation}
    \Omega_s \coloneqq \diff \zeta_s \wedge \hat \Omega_s .
\end{equation}

\begin{lem}
    For any $s \in \R$, $\Omega_s$ is a holomorphic volume form on $\mZ_s$. 
\end{lem}

\begin{proof}
    Since $\zeta_s$ is a holomorphic fibration, $\diff \zeta_s$ does not vanish and hence $\Omega_s$ is a nowhere vanishing $(3,0)$-form, so it remains to see that $\Omega_s$ is holomorphic, or equivalently, that it is closed. From Lemma \ref{lem: Decomposition of Kahler forms}, we can see that the horizontal lift of $\hat \Omega_s$ onto $\P_s$ is given by $\omega^{\m{st}}_2 + \iq \omega^{\m{st}}_3 + \iq \mathfrak{a} \wedge \diff \zeta^{\m{st}}$, so that $\diff \hat \Omega_s = \iq \diff \zeta_s \wedge \Theta_s$. This implies the closedness of $\Omega_s$.
\end{proof}

Let us now define
\begin{equation}    \label{eq: SKRF ansatz}
    \omega^{\m{srf}}_s \coloneqq \omega^{\m{red}}_s + \frac{\iq}{2} \diff \zeta_s \wedge \diff \overline{\zeta}_s = \hat \omega_s + (1+\mr_s^{-2}) \frac{\iq}{2}\diff \zeta_s \wedge \diff \overline{\zeta}_s .
\end{equation}
Since $\omega^{\m{red}}_s$ is a Kähler form on $\mZ_s$ and $\diff \zeta_s \wedge \diff \overline{\zeta}_s$ is closed (even exact) and semi-positive, the $(1,1)$-form $\omega^{\m{srf}}_s$ is a Kähler form on $\mZ_s$.

\begin{prop}        \label{prop: SKRF ansatz}
    The Kähler form $\omega^{\m{srf}}_s$ satisfies the equation
    \begin{equation*}
        \frac{1}{6} (\omega^{\m{srf}}_s)^3 = (1+ \mr_s^{-2}) \cdot \frac{\iq}{8} \cdot  \Omega_s \wedge \overline{\Omega}_s .
    \end{equation*}
\end{prop}

\begin{proof}
    A straightforward calculation yields
    \begin{equation*}
        \frac{1}{6} (\omega^{\m{srf}}_s)^3 = \frac{\iq}{4} (1+ \mr_s^{-2}) \diff \zeta_s \wedge \diff \overline{\zeta}_s \wedge \hat{\omega}^2_s .
    \end{equation*}
    On the other hand, using quaternionic relations we have
    \begin{equation*}
        \hat{\omega}_s^2 = \frac{1}{2} \hat \Omega_s \wedge \overline{\hat \Omega}_s
    \end{equation*}
    whence the result follows.
\end{proof}

In particular, when $s \neq 0$, the function $\mr_s$ is bounded below away from $0$ on $\mZ_s$, and $\omega^{\m{srf}}_s$ is almost Calabi--Yau when $\mr_s \to \infty$. By construction, $\omega^{\m{srf}}_s$ is `semi Ricci-flat', in the sense that its restriction to the fibres of $\zeta_s$ are Calabi--Yau (in fact hyperkähler) metrics. This is analogous to the ansatz used by Li in \cite{li2019new}, except for the fact that he uses an ansatz which already decays faster than the naive semi Ricci-flat Kähler ansatz. In our case, the term $\mr_s^{-2}$ decays quadratically along the fibres of $\zeta_s$ but only like the inverse of the distance in the horizontal directions (cf. \S\ref{subsec: The distance function}), which is slower than the ansatz of Li. Improving our ansatz will require a substantial amount of work, to be carried out in Section \ref{sec: Asymptotic solutions of the MA equation}; the expressions that we will obtain for the correction of $\omega^{\m{srf}}_s$ to an approximate solution of the Monge--Amp\`ere equation with error term decaying faster than quadratically should make it clear that there is no simple way to improve the semi Ricci-flat ansatz at this stage.

When $s$ becomes very small, there is an additional difficulty since $\mr_s^2$ is only bounded below by $2s$, and therefore the factor $\mr_s^{-2}$ blows up in a compact region of $\mZ_s$ when $s \to 0$. Thus we will also need to modify our ansatz near the locus where $\mr_s^{-2}$ is very large in order to understand the limit $s \to 0$. This issue will be addressed in \S\ref{subsec: Interpolation asymptotic regimes}.

    \subsection{Equivariant aspects}    \label{subsec: Equivariant aspects}

Let us identify $\C^4$ with the space of $2 \times 2$ matrices $\mathfrak{gl}(2,\C)$ by defining for any tuple $\underline{z} = (z_1,z_2,z_3,z_4)$ the matrix
\begin{equation}
    Z(\underline{z}) = \begin{pmatrix} z_1 & -z_4 \\ z_3 & z_2 \end{pmatrix} .
\end{equation}
Under this identification, the maps $h^{\m{st}}$ and $\zeta^{\m{st}}$ are given by
\begin{equation*}
    h^{\m{st}}(\underline{z}) = \frac{1}{2}(|Z_2(\underline{z})|^2 - |Z_1(\underline{z})|^2), ~~ \zeta^{\m{st}}(\underline{z}) = \iq \det(Z(\underline{z}))
\end{equation*}
where we denote by $Z_1(\underline{z}), Z_2(\underline{z})$ the columns of the matrix $Z(\underline{z})$. The matrix representation also allows us to define an action of $\Un(2)$ on $\C^4$ by
\begin{equation}
    Z(A \cdot \underline{z}) = A Z(\underline{z}), ~~~~ \forall A \in \Un(2), \forall \underline{z} \in \C^4 .
\end{equation}
It satisfies the following properties:

\begin{lem}
    The action $\Un(2)$ on $\C^4$ satisfies the following properties:
    \begin{enumerate}[(i)]
        \item The action of $\Un(2)$ commutes with the action of $\Un(1)$.
        \item $h^{\m{st}} \circ A = h^{\m{st}}$, for any $A  \in \Un(2)$.
        \item $\zeta^{\m{st}} \circ A = \det(A) \cdot \zeta^{\m{st}}$, for any $A  \in \Un(2)$. 
        \item The action of $\Un(2)$ preserves the complex structure $I$ and the metric $g^{\m{st}}$.
        \item The action of $\SU(2) \subset \Un(2)$ preserves the full quaternionic structure of $\C^4 \cong \H^2$.
    \end{enumerate}
\end{lem}

\begin{proof}
    The first four properties are obvious, so only point (v) requires a justification since the action of $\SU(2) \subset \Un(2)$ on $\H^2$ is not equivalent to the diagonal action under the identification $\SU(2) \cong \mathrm{Sp}(1)$. Since the action of $\Un(2)$ leaves $I$ invariant, it is enough to check that the action of $\SU(2)$ leaves $J$ invariant. The complex structure $J$ acts on $\mathfrak{gl}(2,\C)$ via the real-linear transformation
    \begin{equation*}
        J \begin{pmatrix} z_1 & - z_4 \\ z_3 & z_2\end{pmatrix} = \begin{pmatrix} - \overline{z}_2 & - \overline{z}_3  \\ -\overline{z}_4 & \overline{z}_1\end{pmatrix} \cdot
    \end{equation*}
    Thus if $\alpha, \beta \in \C$ and $|\alpha|^2+|\beta|^2 = 1$ we have
    \begin{align*}
        \begin{pmatrix}\alpha & \beta \\ - \overline{\beta} & \overline{\alpha} \end{pmatrix} J \begin{pmatrix} z_1 & - z_4 \\ z_3 & z_2 \end{pmatrix} & = \begin{pmatrix} -\alpha \overline{z}_2 - \beta \overline{z}_4 & -\alpha \overline{z}_3 + \beta \overline{z}_1 \\ \overline{\beta} \overline{z}_2 - \overline{\alpha} \overline{z}_4 &  \overline{\beta} \overline{z}_3 + \overline{\alpha} \overline{z}_1 \end{pmatrix} , \\
        J \begin{pmatrix}\alpha & \beta \\ - \overline{\beta} & \overline{\alpha} \end{pmatrix} \begin{pmatrix} z_1 & -z_4 \\ z_3 & z_2 \end{pmatrix} & =  J \begin{pmatrix} \alpha z_1 + \beta z_3 & -\alpha z_4 + \beta z_2 \\ -\overline{\beta}z_1 + \overline{\alpha} z_3 & \overline{\beta}z_4 + \overline{\alpha}z_2\end{pmatrix} \\
            & = \begin{pmatrix} -\alpha \overline{z}_2 - \beta \overline{z}_4 & -\alpha \overline{z}_3 + \beta \overline{z}_1 \\ \overline{\beta} \overline{z}_2 - \overline{\alpha} \overline{z}_4 &  \overline{\beta} \overline{z}_3 + \overline{\alpha} \overline{z}_1 \end{pmatrix}
    \end{align*}
    which proves point (v).
\end{proof}

From the commutation of the actions of $\Un(1)$ and $\Un(2)$, it follows that there is an induced action of $\Un(2)$ on $\M$, which makes $\mu = h \times \zeta \colon \M \rightarrow \R \times \C$ an equivariant fibration for the trivial action of $\Un(2)$ on $\R$ and the action of $\Un(2)$ on $\C$ given by the determinant. By point (iv) in the previous lemma, it induces by restriction an action of $\Un(2)$ on $\mZ_s$ by biholomorphisms for any $s \in \R$, with respect to which the holomorphic fibration $\zeta_s$ is equivariant. Moreover, the splitting $T\M = \mathcal{E} \oplus \mathcal{H}$ is equivariant under the action of $\Un(2)$. We list a few additional properties of this action:

\begin{lem}     \label{lem: Equivariance relations for rhos}
    For any $s \in \R$ and $A \in \Un(2)$,
    \begin{enumerate}[(i)]
        \item $\zeta_s \circ A = \det(A) \cdot \zeta_s$,
        \item $A^* \omega^{\m{srf}}_s = \omega^{\m{srf}}_s$,
        \item $A^* \Omega_s = \det(A)^2 \cdot \Omega_s$.
    \end{enumerate}
\end{lem}

\begin{proof}
    The first two properties are a direct consequence of the previous lemma. For the last point, since $\SU(2)$ preserves $\zeta_s$ and the whole hyperkähler structure of $\H^2$, it is enough to check that for $\theta \in \R$, $(e^{\iq \theta} \mathrm{Id})^* \hat \Omega_s = e^{2 \iq \theta} \hat \Omega_s$. But this is clear since on $\H^2$ we can write $\omega_{\jq} + \iq \omega_{\kq} = \diff z_1 \wedge \diff z_2 + \diff z_3 \wedge \diff z_4$ and thus $(e^{\iq \theta} \mathrm{Id})^*(\omega_{\jq} + \iq \omega_{\kq}) = e^{2\iq \theta} (\omega_{\jq} + \iq \omega_{\kq})$ for any $\theta \in \R$.
\end{proof}

\begin{rem}     \label{rem: group action}
    The action of $\Un(2)$ on $\M$ is not faithful; it is easy to check that its kernel is $\{\pm \mathrm{Id}\}$, so that there is an induced action $\Un(1) \times \SO(3) \cong \Un(2) / \{\pm \mathrm{Id}\}$.
\end{rem}

    \subsection{Relation to Atiyah's flop}      \label{subsec: Relation to the flop}

Let $\mZ \subset \C^4$ be the subvariety defined as
\begin{equation}
    \mZ \coloneqq \{(w_{11},w_{12},w_{21},w_{22}) \in \C^4 \mid w_{11} w_{22} - w_{21} w_{12} = 0 \} .
\end{equation}
It is well-known that under a linear transformation of $\C^4$, it can be identified with the nodal cone $\{w_1^2 + w_2^2 + w_3^2 + w_4^2 = 0\}$. An interesting feature of $\mZ$ is that it admits two distinct small crepant resolutions $\mZ_+$ and $\mZ_-$. In order to describe them, we identify $\C^4$ with $\mathfrak{gl}(2,\C)$ using the matrices
\begin{equation*}
    W(\underline{w}) = \begin{pmatrix} w_{11} & w_{12} \\ w_{21} & w_{22} \end{pmatrix} \cdot
\end{equation*}
Then we can regard $\mZ$ as the variety $\mZ = \{ W \in \mathfrak{gl}(2;\C) \mid \det(W) = 0\}$, whose smooth locus $\mZ^{\mathrm{reg}} = \mZ \setminus \{0\}$ is the set of rank-one matrices. Let us define the following two subvarieties of $\mathfrak{gl}(2,\C) \times \mathbb{P}^1$
\begin{equation}        \label{eq: def of small resolutions}
    \begin{aligned}
        \mZ_+ & \coloneqq \{ (W,\Pi) \in \mathfrak{gl}(2,\C) \times \mathbb{P}^1 \mid \det(W) = 0, ~\Pi \subset \ker(W) \} , \\
        \mZ_- & \coloneqq \{ (W,\Pi) \in \mathfrak{gl}(2,\C) \times \mathbb{P}^1 \mid \det(W) = 0, ~\mathrm{im}(W) \subset \Pi \} .
    \end{aligned}
\end{equation}
The subvarieties $\mZ_\pm \subset \mathfrak{gl}(2,\C) \times \mathbb{P}^1$ are in fact smooth and biholomorphic to the total space of the holomorphic vector bundle $\mathcal{O}_{\mathbb{P}^1}(-1) \oplus \mathcal{O}_{\mathbb{P}^1}(-1)$ over $\mathbb{P}^1$, where the bundle projections $p_\pm \colon \mZ_\pm \to \mathbb{P}^1$ are given by the projections onto the second component. Moreover, the projections onto $\mathfrak{gl}(2,\C)$ define small crepant resolutions $\pi_\pm \colon \mZ_\pm \to \mZ$; in particular $\pi_\pm^{-1}(0) = \mC_\pm$ are embedded rational curves. Note in addition that the action of $\Un(2)$ on $\mZ$ by $A \cdot W = AWA^T$ lifts uniquely to $\mZ_+$ and $\mZ_-$ in a way that makes the resolution maps $\pi_+$ and $\pi_-$ equivariant. However, $(\mZ_+,\pi_+)$ and $(\mZ_-,\pi_-)$ are distinct resolutions of $\mZ$; the induced birational map $\mZ_+ \dashrightarrow \mZ_-$ is a \emph{flop}, first discovered by Atiyah \cite{atiyah1958analytic}. 

Consider now the map $p \colon \C^4 \to \mZ$ defined as
\begin{equation*}
    p(\underline{z}) \coloneqq Z_1(\underline{z}) Z_2(\underline{z})^T = \begin{pmatrix} -z_1z_4 & z_1 z_2 \\ -z_3 z_4 &  z_3 z_2 \end{pmatrix} \cdot
\end{equation*}
This map is manifestly holomorphic and invariant under the action of $\Un(1)$, and therefore it descends to a smooth map $\pi \colon \M \to \mZ$ whose restrictions $\pi_s \colon \mZ_s \to \mZ$ to the fibres of $h$ are holomorphic for all $s \in \R$. It is moreover clear that the map $\pi$ is equivariant under the action of $\Un(2)$. 

\begin{prop}
\label{prop:identification-hyperkahler-quotients}
    The holomorphic maps $\pi_s \colon \mZ_s \to \mZ$ satisfy:
    \begin{enumerate}[(i)]
        \item $\pi_0 \colon \mZ_0 \to \mZ$ defines a biholomorphism of $\mZ_0$ onto the smooth locus $\mZ^{\mathrm{reg}} = \mZ \setminus \{0\}$.

        \item If $s > 0$, $\pi_s \colon \mZ_s \to \mZ$ lifts uniquely to a biholomorphism $\pi^+_s \colon \mZ_s \to \mZ_+$. Moreover, the line bundle $\mathcal{L}_s = \P_s \times_{\Un(1)} \C \cong (\pi^+_s)^* \mathcal{L}_+$ where $\mathcal{L}_+ = p_+^*\mathcal{O}_{\mathbb{P}^1}(1)$, and under this identification $\mathfrak{a}_s$ is the Chern connection of the metric $H_s = (s+\sqrt{R^2+s^2}) H_{\mathrm{can}}$ where $H_{\mathrm{can}} = H_{\mathrm{taut}}^{\vee}$ is the canonical metric of the line bundle $\mathcal{O}_{\mathbb{P}^1}(1) \cong \mathcal{O}_{\mathbb{P}^1}(-1)^\vee$ over $\mathbb{P}^1$ and $R^2 = |W|^2$ for $W \in \mathfrak{gl}(2,\C)$.

        \item If $s < 0$, $\pi_s \colon \mZ_s \to \mZ$ lifts uniquely to a biholomorphism $\pi^-_s \colon \mZ_s \to \mZ_-$. Moreover, the line bundle $\mathcal{L}_s = \P_s \times_{\Un(1)} \C \cong (\pi^-_s)^* \mathcal{L}_-$ where $\mathcal{L}_- = p_-^* \mathcal{O}_{\mathbb{P}^1}(-1)$, and under this identification $\mathfrak{a}_s$ is the Chern connection of the metric $H_s = (-s+\sqrt{R^2+s^2})^{-1}H_{\m{taut}}$.
    \end{enumerate}
\end{prop}

\begin{proof}
    For point (i): if $W \in \mZ \setminus \{0\}$ is a rank-one matrix, there exist two column vectors $Z_1,Z_2$ such that $W = Z_1 Z_2^T$; moreover $Z_1$ and $Z_2$ are unique up to multiplying $Z_1$ by $\Lambda$ and $Z_2$ by $\Lambda^{-1}$ for some $\Lambda \in \C^*$. If moreover we specify $|Z_1| = |Z_2|$, then the vectors are unique up to multiplying $Z_1$ by $e^{\iq \theta}$ and $Z_2$ by $e^{-\iq \theta}$ for some $\theta \in \R$. Since the real moment map is given by $h^{\m{st}}(\underline{z}) = \frac{1}{2}(|Z_2|^2-|Z_1|^2)$, this proves that $\pi_0$ defines a bijection of $\mZ_0$ onto $\mZ \setminus \{0\}$. Moreover, the inverse $\pi_0^{-1}$ is smooth since $\ker(W)$ and $\mathrm{im}(W)$ vary smoothly with $W \in \mZ \setminus \{0\}$.   

    For point (ii): since $h^{\m{st}}(\underline{z}) = \frac{1}{2}(|Z_2|^2-|Z_1|^2)$, the column vector $Z_2(\underline{z})$ does not vanish on $\P_s = (h^{\m{st}})^{-1}(s)$ when $s > 0$. Thus $\pi_s$ lifts to the holomorphic map $\pi^+_s \colon \mZ_s \to \mZ_+$ defined as
    \begin{equation*}
        \pi_s^+([\underline{z}]) \coloneqq (Z_1(\underline{z})Z_2(\underline{z})^T, [Z_2(\underline{z})^\perp]) \in \mZ \times \mathbb{P}^1, ~~ \text{where} ~ Z_2(\underline{z})^\perp = \begin{pmatrix} z_2 \\ z_4 \end{pmatrix} \in \ker(Z_1Z_2^T) .
    \end{equation*}
    As for point (i), it is clear that $\pi^+_s$ is bijective and has a smooth inverse, so that it is a biholomorphism. In order to prove the second part of the assertion, we define an augmentation of the previous map:
    \begin{equation*}
        \P_s \times \C \to \mathfrak{gl}(2,\C) \times \mathcal{O}_{\mathbb{P}^1}(-1), ~~ (\underline{z},\Lambda) \mapsto  (Z_1(\underline{z})Z_2(\underline{z})^T, [Z_2(\underline{z})^\perp], \Lambda Z_2(\underline{z})^\perp) .
    \end{equation*}
    This map is invariant under the right action of $\Un(1)$ on $\P_s \times \C$ defined as $(\underline{z},\Lambda) \cdot e^{\iq \theta} = (\underline{z} \cdot e^{-\iq\theta}, e^{- \iq \theta} \Lambda)$, and thus induces an identification $\mathcal{L}_s^\vee \cong (\pi_s^+)^* p_+^* \mathcal{O}_{\mathbb{P}^1}(-1)$. Moreover, $-\mathfrak{a}_s$ (seen as a connection form on $\mathcal{L}_s^\vee$) is the Chern connection of the bundle metric $H_s$ such that $\mathcal{P}_s \subset \mathcal{L}_s^\vee$ is the bundle of unit vectors. Thus $H_s = |Z_2|^{-2} H_{\mathrm{taut}}$. To express $|Z_2|^2$ in terms of $R^2 = |W|^2 = |Z_1 Z_2^T|^2$, note that $R^2 = |Z_1|^2|Z_2|^2 = (|Z_2|^2-2s)|Z_2|^2$, and when $s > 0$ the only positive solution is $|Z_2|^2 = s + \sqrt{R^2+s^2}$. By duality, it follows that $\mathfrak{a}_s$ is the Chern connection of $H_s^\vee = (s+\sqrt{R^2+s^2})H_{\mathrm{can}}$.

    For point (iii): when $h^{\m{st}}(\underline{z}) = s < 0$, the vector $Z_1(\underline{z})$ does not vanish, and thus $\pi_s$ lifts to the smooth map $\pi_s^- \colon \mZ_s \to \mZ_-$ defined as
    \begin{equation*}
        \pi_s^-([\underline{z}]) \coloneqq (Z_1(\underline{z})Z_2(\underline{z})^T, [Z_1(\underline{z})]) \in \mZ \times \mathbb{P}^1, ~~ \text{where} ~ Z_1 \in \mathrm{im}(Z_1Z_2^T) .
    \end{equation*}
    As before, this defines a biholomorphism. To prove the assertion on the line bundle $\mathcal{L}_s$ we proceed as for the previous point: we define a map
    \begin{equation*}
        \P_s \times \C \to \mathfrak{gl}(2,\C) \times \mathcal{O}_{\mathbb{P}^1}(-1), ~~ (\underline{z},\Lambda) \mapsto  (Z_1(\underline{z})Z_2(\underline{z}), [Z_1(\underline{z})], \Lambda Z_1(\underline{z})) .
    \end{equation*}
    This map is invariant under the right action of $\Un(1)$ on $\P_s \times \C$ defined as $(\underline{z},\Lambda) \cdot e^{\iq \theta} = (\underline{z}\cdot e^{\iq \theta},e^{-\iq \theta} \Lambda)$ and thus induces an identification $\mathcal{L}_s \cong (\pi^-_+)^* p_-^*\mathcal{O}_{\mathbb{P}^1}(-1)$. Moreover, $\mathfrak{a}_s$ is identified with the Chern connection of the metric $H_s = |Z_1|^{-2} H_{\mathrm{taut}}$, and $R^2 = |Z_1|^2 |Z_2|^2 = |Z_1|^2(|Z_2|^2+2s)$ yields $|Z_1|^2 = -s + \sqrt{R^2+s^2}$.
\end{proof}

\begin{rem}
    From these identifications, we obtain embedded rational curves $\mC_s \subset \mZ_s$ for any $s \neq 0$, defined as $\mC_s = (\pi_s^{+})^{-1}(\mC_+)$ if $s > 0$ and $\mC_s = (\pi_s^-)^{-1}(\mC_-)$ if $s < 0$.
\end{rem}

It is well-known that $\mZ$ admits a holomorphic volume form, which can be defined on the open subset $\{w_{11} \neq 0\} \subset \mZ$ as
\begin{equation*}
    \Omega \coloneqq w_{11}^{-1} \diff w_{11} \wedge \diff w_{21} \wedge \diff w_{12} .
\end{equation*}
It is not difficult to check that $A^* \Omega = \det(A)^2 \Omega$, the same equivariance relation as $\Omega_0$ on $\mZ_0$. In particular, this means that under the identification $\pi_0 \colon \mZ_0 \to \mZ^{\mathrm{reg}}$, we have $\Omega_0 = \psi \Omega$ for some nowhere vanishing holomorphic function $\psi$ on $\mZ^{\mathrm{reg}}$, and Lemma \ref{lem: Equivariance relations for rhos} implies that $\psi$ must be invariant under the action of $\Un(2)$. It is not difficult to see that the $\Un(2)$-invariance forces $\psi = F(\mr_0^2,|\zeta_0|^2)$ to be a function of $\mr_0^2$ and $|\zeta_0|^2$ only, but since the $(0,1)$-forms $\overline{\partial}(\mr_0^2)$ and $\zeta_0 \diff \overline{\zeta}_0$ are independent on an open subset of $\mZ_0$ the equation $\overline{\partial}\psi = 0$ forces $\psi$ to be a constant. We define the volume form $\vol = \vol_0 = \frac{3}{4}\Omega_0 \wedge \overline{\Omega}_0 = \frac{3|\psi|^2}{4} \Omega \wedge \overline{\Omega}$, which agrees with $\Omega_0 \wedge \overline{\Omega}_0$ up to a positive multiplicative constant. Since $\Omega$ lifts to holomorphic volume forms $\Omega_{\pm}$ on $\mZ_\pm$, we can lift $\vol$ to volume forms $\vol_{\pm}$.

\begin{prop}       \label{prop: Explicit form of srf ansatz}
    Under the identifications $\pi_0 \colon \mZ_0 \to \mZ^{\mathrm{reg}}$ and $\pi_s \colon \mZ_s \to \mZ_{\pm}$,
    \begin{align*}
        \omega^{\m{srf}}_0 & = \iq \partial \overline{\partial} \left( R + \frac{|w_{12}-w_{21}|^2}{2} \right) , \\
        \omega^{\m{srf}}_s & = s \widetilde{\omega}^+_{\m{FS}} + \iq \partial \overline{\partial} \left(\sqrt{R^2+s^2} - s \log(s+ \sqrt{R^2+s^2}) + \frac{|w_{12}-w_{21}|^2}{2} \right), ~~ (s > 0),\\
        \omega^{\m{srf}}_{s} & = - s \widetilde{\omega}^-_{\m{FS}} + \iq \partial \overline{\partial} \left(\sqrt{R^2+s^2} + s \log(-s+ \sqrt{R^2+s^2}) + \frac{|w_{12}-w_{21}|^2}{2} \right), ~~ (s < 0).
    \end{align*}
    where $\widetilde{\omega}^\pm_{\m{FS}} = p_\pm^*\omega_{\m{FS}}$ is the pull-back of the Fubini--Study metric on $\mathbb{P}^1$, normalised so that $[\omega_{\m{FS}}] = 2\pi c_1(\mathcal{O}_{\mathbb{P}^1}(1)) \in \m{H}^2(\mathbb{P}^1)$, that is, $\omega_{\m{FS}}$ is the curvature of the metric $H_{\m{can}}$ on $\mathcal{O}_{\mathbb{P}^1}(1)$. Moreover, we can identify the volume forms:
    \begin{equation*}
        \vol_s = \vol_+ ~~ \text{(when $s > 0$)}, ~~ \text{and} ~~ \vol_s = \vol_- ~~ \text{(when $s < 0$)} .
    \end{equation*}
\end{prop}

\begin{proof}
    The expressions for $\omega^{\m{srf}}_s$ follow directly from the previous proposition since $\omega_{\m{FS}}$ is the curvature of $H_{\mathrm{can}}$ on $\mathcal{O}_{\mathbb{P}^1}(1)$ and $r_s^2 = |Z_1|^2+|Z_2|^2 = 2 \sqrt{R^2+s^2}$ on $\mZ_s$, by the calculations of $|Z_1|^2$ and $|Z_2|^2$ made in the previous proof. 
    
    For the volume forms, we can argue as for the case $s = 0$: $\Omega_s = \psi_s \Omega_\pm$ for some nowhere vanishing holomorphic function $\psi_s$ on $\mZ_s$. By Lemma \ref{lem: Equivariance relations for rhos}, $\psi_s$ must be invariant under the action of $\Un(2)$, and as before we deduce that $\psi_s \in \C^*$ is a constant. In particular if $s > 0$, $\vol_s = c_s \cdot \vol$ for some constant $c_s > 0$. To prove that $c_s = 1$, we can use the expression of $\omega_s$ and identify $\mathcal{Z_+} \setminus \mC_+$ with $\mZ^{\mathrm{reg}}$: indeed as $R \to \infty$ we have
    \begin{equation*}
        |\partial \overline{\partial} (\sqrt{R^2+s^2} - \log(s+\sqrt{R^2+s^2}))-\partial \overline{\partial}(R)|_{\omega^{\m{srf}}_0} = \mo(R^{-1}) .
    \end{equation*}
    This follows from the bound $|\partial \overline{\partial}R|_{\omega_0^{\m{srf}}} = \mo(1)$, which is obvious, and $|\diff R|_{\omega_0^{\m{srf}}} = \mo(R^{1/2})$, which we will see in the proof of Proposition \ref{prop: interpolation metric}. Similarly, because $\widetilde{\omega}_{\m{FS}}$ has local potentials of order $\log(R)$, we have $|\widetilde{\omega}_{\m{FS}}|_{\omega_0^{\m{srf}}} = \mo(R^{-1})$ and thus $|\omega^{\m{srf}}_s - \omega^{\m{srf}}_0|_{\omega_0} = \mo(R^{-1})$. By Proposition \ref{prop: SKRF ansatz} this implies that $c_s = 1$.
\end{proof}

The above proposition has an interesting consequence. Let $\iota \colon \mZ \to \mZ$ be the involution induced by the transposition map $W \mapsto W^T$ in $\mathfrak{gl}(2,\C)$. It is clear that $\iota$ leaves invariant the functions $R$, $|w_{12}-w_{21}|^2$ and the volume form $\vol$ on $\mZ$. On the other hand, $\iota$ can be lifted to a $\Un(2)$-equivariant biholomorphism $\jmath \colon \mZ_- \to \mZ_+$ between the two small resolutions of $\mZ$, which under the embeddings $\mZ_\pm \subset \mathfrak{gl}(2,\C) \times \mathbb{P}^1$ of \eqref{eq: def of small resolutions} can be obtained by restriction of the map $(W,\Pi) \mapsto (W^T,\Pi^\perp)$ (where as before, $\Pi^\perp$ is the orthogonal complement of $\Pi \in \mathbb{P}^1$ for the canonical complex quadratic form on $\C^2$). From this description, it is immediate to see that $\perp\circ \,p_- = p_+ \circ \jmath$, which implies that for any $s > 0$, 
\begin{equation}
    \jmath^* \omega^{\m{srf}}_s = \omega^{\m{srf}}_{-s} .
\end{equation}
For this reason, we will usually choose to work on $\mZ_+$ with positive values of the moment map parameter $s$ in the remainder of this article, since we can pull everything back to $\mZ_-$ for negative values of $s$ using the equivariant biholomorphism $\jmath \colon \mZ_- \to \mZ_+$.

In the course of the previous proof, we have also proved the following result, which will play an important role later:

\begin{lem}     \label{lem: asymptotic equivalence of skrf metrics}
    After identifying $\mZ^{\m{reg}} \cong \mZ_+ \setminus \mC_+$, for any $s_0 > 0$ there exists $R_0, C_0 > 0$ such that for all $s \in [0,s_0]$,
    \begin{equation*}
        |\omega_s^{\m{srf}}-\omega_0^{\m{srf}}|_{\omega_0^{\m{srf}}} \leq C_0 R^{-1}, ~~~~ \text{if} ~ R \geq R_0.
    \end{equation*}
\end{lem}

\begin{rem}
    The bound $|\widetilde{\omega}_{\m{FS}}|_{\omega_0^{\m{srf}}} = \mo(R^{-1})$ could also be justified as follows. The Stenzel metric on $\mZ^{\m{reg}}$ is conical with potential proportional to $R^{4/3}$ and radius function $R^{2/3}$. Using the estimates on $|\diff R|_{\omega_0^{\m{srf}}}$, one can deduce that the Stenzel metric is bounded above by $R^{1/3} \omega_0^{\m{srf}}$; hence the norm of dilation-invariant vectors grows at least at order $R^{1/2}$ with respect to $\omega_0^{\m{srf}}$, and since $\widetilde{\omega}_{\m{FS}}$ is a dilation-invariant $2$-form it must decay at order at least $R^{-1}$ with respect to $\omega_0^{\m{srf}}$.
\end{rem}

\begin{rem}     \label{rem: algebraic aspects}
    In this section, we have adopted a very differential-geometric point of view on the flop, because it is well-adapted for the construction of the family of semi Ricci-flat metrics $\omega^{\m{srf}}_s$. From a more algebraic perspective, $\mathcal{Z}$ is the affine three-dimensional ordinary double point, or equivalently the affine cone over the Segre embedding of $\mathbb{P}^1\times\mathbb{P}^1$. Blowing up its singular point produces an exceptional divisor $\mathbb{P}^1\times\mathbb{P}^1$, and contracting either of its two rulings gives the two small resolutions. The resulting birational transformation between them is precisely the Atiyah flop. 
    
    This picture can equivalently be understood through variation of GIT. The sign of the moment-map parameter determines which resolution is obtained, while the zero level lies on the wall and corresponds to the singular space $\mathcal{Z}$. Away from the exceptional curves, the line bundles introduced above agree naturally, whereas the preferred ample directions on the two resolutions are opposite. Thus the passage from $s>0$ to $s<0$ in Proposition \ref{prop: Explicit form of srf ansatz} is a metric counterpart of crossing the wall of the Atiyah flop.
\end{rem}
   
%%%%% End of section 2 %%%%%
%%%%%%%%%%%%%%%%%%%%%%%%%%%%

%% file: section3_asympsolutions.tex
%%%%%%%%%%%%%%%%%%%%%%%%%%%%%%%%%%
%%%%% Beginning of Section 3 %%%%%

\section{Asymptotic solutions of the Monge--Amp\`ere equation}      \label{sec: Asymptotic solutions of the MA equation}

In this section, we seek to improve the semi Ricci-flat ansatz $\omega^{\m{srf}}_s$ to an asymptotic solution of the Monge--Amp\`ere equation on $\mZ_\pm$ with error term decaying faster than quadratically. Given the symmetry between $\omega^{\m{srf}}_{s}$ and $\omega^{\m{srf}}_{-s}$, we may decide to work on $\mZ_+$ with $s \geq 0$ (where we shall implicitly identify $\mZ^{\m{reg}}$ with the locus $\mZ_+ \setminus \mC_+$ when $s = 0$ without further comment). As it turns out, the explicit form of $\omega_s^{\m{srf}}$ derived in Proposition \ref{prop: Explicit form of srf ansatz} is not easy to work with, so we shall instead work on $\mZ_s \subset \M$ and exploit the hyperkähler reduction picture in order to build a convenient frame and write the Monge--Amp\`ere equation in a convenient way. The global $\Un(2)$-symmetry and the moment map relations will allow us to choose a set of variables where the Monge--Amp\`ere equation looks relatively tractable, and eventually to construct approximate solutions whose asymptotic properties are uniformly under control with respect to the parameter $s$. 

Let us briefly outline the structure of this section. In \S\ref{subsec: A convenient frame}, we construct an $s$-dependent frame  which is well-adapted to $\omega_s^{\m{srf}}$ on a dense open subset of $\mZ_+$, identify two independent $\Un(2)$-invariant functions that will serve as variables for the Monge--Amp\`ere equation, and derive a few fundamental identities that will play a key role throughout this article. In \S\ref{subsec: The distance function}, we derive an approximate expression for the distance function of $\omega_s^{\m{srf}}$ at infinity and study the asymptotic behaviour of various useful functions. Based on this preparatory work, we derive in \S\ref{subsec: The MA equation} an explicit expression for the Monge--Amp\`ere equation restricted to an admissible class of $\Un(2)$-invariant potentials, in terms of our chosen set of variables. The following parts, \S\ref{subsec: Improvement I}, \S\ref{subsec: Improvement II} and \S\ref{subsec: Interpolation III}, are dedicated to the construction of an admissible family of potentials solving the Monge--Amp\`ere equation up to an error term with faster than quadratic decay at infinity, uniformly with respect to $s$. Finally, \S\ref{subsec: Higher-order estimates} establishes higher-order estimates on the approximate solutions, both in terms of the $\Un(2)$-invariant variables and intrinsically with respect to the Levi--Civita connection. 

For the convenience of the reader, the main outcome of the present section is summarised in the following proposition:

\begin{prop}
    There exists a family of $\Un(2)$-symmetric potentials $\{\phi_s\}_{s \in [0,\infty)}$, defined in the region $\{R \geq 1\} \subset \mZ_+$, such that the following properties hold:
    \begin{enumerate}[(i)]
        \item For any $k \geq 1$ and $s_0 \in (0,\infty)$, there exists a constant $C_{k,s_0} > 0$ such that for any $s \in [0,s_0]$,
        \begin{equation*}
            |\nabla_{\omega_s^{\m{srf}}}^k \phi_s| \leq C_{k,s_0} \mr_s^{-k} .
        \end{equation*}

        \item $(\omega_s^{\m{srf}}+\frac{\iq}{2} \partial\overline{\partial}\phi_s)^3 = (1+\psi_s)\vol_+$, where for any $k \geq 0$, $s \in (0,\infty)$ and $\varepsilon \in (0,1)$, there exists a constant $C^\prime_{k,s_0,\varepsilon} > 0$ such that for any $s \in [0,s_0]$,
        \begin{equation*}
            |\nabla^k_{\omega_s^{\m{srf}}} \psi_s| \leq C^\prime_{k,s_0,\varepsilon} \mr_s^{-k} \dist_{\omega_s^{\m{srf}}}(o_+,\cdot)^{-3+\varepsilon} .
        \end{equation*}
    \end{enumerate}
\end{prop}

    \subsection{A convenient frame}     \label{subsec: A convenient frame}

Our first task is to fix a convenient frame on $\M$. In fact, since $T\M$ is not globally trivial, we can only fix a frame on an open dense subset of $\M$, which is enough for our purpose. Equivalently, we may seek an $\Un(1)$-invariant frame on $T \P = \mathcal{V} \oplus \mathcal{E} \oplus \mathcal{H}$. Using the vector field $\xi$ generating the $\Un(1)$-action, we have already constructed an $\Un(1)$-invariant (and $g^\m{st}$-orthonormal) frame of $\mathcal{V} \oplus \mathcal{H}$, $\xi_0 = \tfrac{1}{\mr} \xi$, $\xi_1 = I\xi_0$, $\xi_2 = J \xi_0$, $\xi_3 = K \xi_0$, so we only need to fix a frame of $\mathcal{E}$. We will use the following observation:

\begin{lem}
    The functions $|\mu^{\m{st}}|$ and $r$ satisfy the inequality
    \begin{equation*}
        |\mu^{\m{st}}| \leq \frac{1}{2} \mr^2 .
    \end{equation*}
    Moreover, on the open dense set $\{ |\mu^\m{st}| < \tfrac{1}{2} r^2\} \subset \H^2$, the vector field $\partial_{\mr} = \tfrac{1}{\mr} (q_1,q_2)$ is linearly independent of $\xi_0,\xi_1,\xi_2,\xi_3$. 
\end{lem}

\begin{proof}
    The inequality is straightforward since
    \begin{equation*}
        |\mu^{\m{st}}| = \frac{1}{2}|\Im(q_1\iq \overline{q}_1 + q_2 \iq \overline{q}_2)| \leq \frac{1}{2}(|q_1|^2+|q_2|^2).
    \end{equation*}
    Moreover, it is clear that $\partial_{\mr}$ is $g^\m{st}$-orthogonal to $\mathcal{V}$, and the $g^\m{st}$-orthogonal projection of $\partial_{\mr}$ onto $\mathcal{H}$ reads
    \begin{align*}
        \partial_{\mr}^\mathcal{H} & = \tfrac{1}{\mr^2} (g^\m{st}(I\xi,\partial_{\mr})I\xi +  g^\m{st}(J\xi,\partial_{\mr})J\xi +  g^\m{st}(K\xi,\partial_{\mr})K\xi) \\ 
        & = \tfrac{1}{\mr^2} (\omega_1^\m{st}(\xi,\partial_{\mr}) I\xi + \omega_2^\m{st}(\xi,\partial_{\mr}) J\xi + \omega_3^\m{st}(\xi,\partial_{\mr}) K\xi) \\
        & = \tfrac{1}{\mr^2} (\diff \mu^\m{st}_1(\partial_{\mr}) I\xi + \diff \mu^\m{st}_2(\partial_{\mr}) J\xi + \diff \mu^\m{st}_3(\partial_{\mr}) K\xi)
    \end{align*}
    using the fact that $\mu^\m{st}$ is a hyperkähler moment map. On the other hand, the computations of Lemma \ref{lem: Moment map mu} show that for any $\lambda \in \Im(\H)$,
    \begin{equation*}
        \diff \mu^\m{st}_\lambda(\partial_{\mr}) = - \frac{1}{\mr}\langle \Im( q_1 \iq \overline{q}_1 + q_2 \iq  \overline{q}_2), \lambda \rangle = \frac{2\mu^\m{st}_\lambda}{\mr} \cdot
    \end{equation*}
    Thus the $g^\m{st}$-orthogonal projection of $\partial_{\mr}$ onto $\mathcal{H}$ is
    \begin{equation}        \label{eq: Projection of eta onto H}
        \partial_{\mr}^\mathcal{H} = \frac{2\mu_1}{\mr^3} I\xi + \frac{2\mu_2}{\mr^3} J\xi + \frac{2\mu_3}{\mr^3} K\xi
    \end{equation}
    which satisfies $|\partial_{\mr}^\mathcal{H}|^2 = \tfrac{4|\mu^\m{st}|^2}{\mr^4}$, whilst $|\partial_{\mr}|^2 = 1$. The lemma follows.
\end{proof}

Let us define the function 
\begin{equation}
    \upsilon \coloneqq \tfrac{1}{4}\mr^4 - |\mu^\m{st}|^2 .
\end{equation}
By the previous lemma, the $g^\m{st}$-orthogonal projection $\partial_{\mr}^\mathcal{E}$ of $\partial_{\mr}$ onto $\mathcal{E}$ has $|\partial_{\mr}^\mathcal{E}|^2 = \tfrac{4\upsilon}{\mr^4}$, and on the open domain $\{ \upsilon > 0\}$ we obtain a global $g^\m{st}$-orthonormal frame of $\mathcal{E}$ defined as $\sigma_0 = \tfrac{\mr^2}{2\sqrt{\upsilon}}\partial_{\mr}^\mathcal{E}$, $\sigma_1 = I \sigma_0$, $\sigma_2 = J\sigma_0$, $\sigma_3 = K \sigma_0$. If $\sigma^0 = g^\m{st}(\sigma_0,\cdot)$, $\sigma^1 = g^\m{st}(\sigma_1,\cdot)$, $\sigma^2 = g^\m{st}(\sigma_2,\cdot)$, $\sigma^3= g^\m{st}(\sigma_3,\cdot)$ is the dual co-frame, then on $\{\upsilon > 0\}$ we can write
\begin{align*}
    \omega^\m{st}_i & = \sigma^0 \wedge \sigma^i + \sigma^j \wedge \sigma^k + \xi^0 \wedge \xi^1 + \xi^2 \wedge \xi^3 \\
        & = \sigma^0 \wedge \sigma^i + \sigma^j \wedge \sigma^k - \iq \mathfrak{a} \wedge \diff \mu^{\m{st}}_i + \tfrac{1}{r^2} \diff \mu^\m{st}_j \wedge \diff \mu^\m{st}_k
\end{align*}
for any direct permutation $(ijk)$ of $(123)$. We want to identify the $1$-forms $\sigma^\ell$, $\ell = 0,1,2,3$. In order to do this, we will use the following result:

\begin{lem}
    The $1$-form $\diff \upsilon$ vanishes on $\mathcal{V} \oplus \mathcal{H}$; and on the open set $\{\upsilon > 0\}$ we have 
    \begin{equation*}
        \sigma^0 = \frac{1}{2\mr \sqrt{\upsilon}} \diff \upsilon .
    \end{equation*}
\end{lem}

\begin{proof}
    Since the function $\upsilon$ is invariant under the action of $\Un(1)$, $\diff \upsilon$ vanishes on $\mathcal{V}$. On the other hand, $\diff (\mr^4) = 2 \mr^2 \diff(\mr^2)$ and
    \begin{equation*}
        \diff (\mr^2) (\lambda \xi) = 2 g^\m{st}(\mr\partial_{\mr},\lambda \xi) = 2 \mr\diff \mu^\m{st}_\lambda (\partial_{\mr}) = 4 \mu^\m{st}_\lambda = \tfrac{4 \mu^\m{st}_\lambda}{\mr^2} \diff\mu^\m{st}_\lambda(\lambda \xi)
    \end{equation*}
    for any $\lambda \in \Im(\H)$. This shows that $4\diff \upsilon = \diff(\mr^4) - 8 \mu^\m{st}_1 \diff \mu^\m{st}_1 - 8 \mu^\m{st}_2 \diff \mu^\m{st}_2 - 8 \mu^\m{st}_3 \diff \mu^\m{st}_3$ vanishes on the horizontal distribution $\mathcal{H}$. To prove the second assertion, we calculate using the fact that $\diff \upsilon$ coincides with $\tfrac{1}{4}\diff (\mr^4)$ along the distribution $\mathcal{E}$:
    \begin{align*}
        \diff \upsilon (\sigma_0) = \frac{1}{4|\partial_{\mr}^\mathcal{E}|_\m{st}} \cdot 2 \mr^2 \diff(\mr^2)(\partial_{\mr}^\mathcal{E}) = \frac{\mr^3}{|\partial_{\mr}^\mathcal{E}|_\m{st}} g^\m{st}(\partial_{\mr},\partial_{\mr}^\mathcal{E}) = \mr^3 |\partial_{\mr}^\mathcal{E}|_\m{st} = 2 \mr \sqrt{\upsilon}
    \end{align*}
    and if $\ell \in \{1,2,3\}$, we similarly obtain
    \begin{equation*}
        \diff \upsilon (\sigma_\ell) = \frac{\mr^3}{|\partial_{\mr}^\mathcal{E}|} g^\m{st}(\partial_{\mr},\sigma_\ell) = \frac{\mr^3}{|\partial_{\mr}^\mathcal{E}|} g^\m{st}(\partial_{\mr}^\mathcal{E},\sigma_\ell) = 0
    \end{equation*}
    since $\sigma_\ell \perp \partial_{\mr}^\mathcal{E}$ for $\ell = 1,2,3$. 
\end{proof}

\begin{cor}
    On the open set $\{\upsilon > 0\}$, 
    \begin{equation*}
        \omega^\m{st}_i = \frac{1}{4\mr^2\upsilon} (\diff \upsilon \wedge \diff^c_i \upsilon  + \diff^c_j \upsilon \wedge \diff^c_k \upsilon) - \iq \mathfrak{a} \wedge \diff \mu^\m{st}_i + \frac{1}{\mr^2} \diff \mu^\m{st}_j \wedge \diff \mu^\m{st}_k
    \end{equation*}
    for any direct permutation $(ijk)$ of $(123)$, where we use the notations $\diff^c_1 \upsilon = \diff \upsilon \circ I^{-1}$, $\diff^c_2 \upsilon = \diff \upsilon \circ J^{-1}$, $\diff^c_3 \upsilon = \diff \upsilon \circ K^{-1}$. 
\end{cor}

We can also use this frame to conveniently write the curvature $\diff \mathfrak{a}$ of the connection form $\mathfrak{a} = \xi^0\iq  = \tfrac{1}{\mr^2} g^\m{st}(\xi,\cdot) \iq \in \Omega^1(\P,\iq \R)$.

\begin{prop}
    On the open set $\{\upsilon > 0\} \subset \M$, the curvature $\diff \mathfrak{a}$ satisfies
    \begin{align*}
        - \iq \diff \mathfrak{a} = \diff \xi^0 = ~ & ~ \frac{4}{\mr^6}(\mu_1 \diff \mu_2 \wedge \diff \mu_3  + \mu_2 \diff \mu_3  \wedge \diff \mu_1  + \mu_3 \diff \mu_1 \wedge \diff \mu_2) \\
            & - \frac{2}{\mr^6} (\diff \mu_1 \wedge \diff^c_1 \upsilon + \diff \mu_2 \wedge \diff^c_2 \upsilon + \diff \mu_3 \wedge \diff^c_3 \upsilon) \\ 
            & + \frac{1}{\upsilon \mr^6} \diff \upsilon \wedge (\mu_1 \diff^c_1 \upsilon + \mu_2 \diff^c_2 \upsilon + \mu_3 \diff^c_3 \upsilon ) \\
            & - \frac{1}{\upsilon \mr^6} (\mu_1 \diff^c_2 \upsilon \wedge \diff^c_3 \upsilon + \mu_2 \diff^c_3 \upsilon \wedge \diff^c_1 \upsilon + \mu_3 \diff^c_1 \upsilon \wedge \diff^c_2 \upsilon) .
    \end{align*} 
\end{prop}

\begin{proof}
    On $\P$, $\xi^0 = \tfrac{1}{\mr^2}(\Re(q_1 \iq d\overline{q}_1 + q_2 \iq \diff \overline{q}_2))$ and thus 
    \begin{equation*}
        \diff \xi^0 = - \frac{1}{\mr^2} \diff (\mr^2) \wedge \xi^0 + \tfrac{1}{\mr^2} \Re(\diff q_1 \iq \wedge \diff \overline{q}_1 + \diff q_2 \iq \wedge \diff \overline{q}_2)
    \end{equation*}
    But since $\iota_\xi \diff \mathfrak{a} = 0 = \iota_\xi \diff \xi^0$, we deduce that
    \begin{equation*}
        \diff \xi^0 = \frac{1}{\mr^2} \left. (\Re(\diff q_1 \iq \wedge \diff \overline{q}_1 + \diff q_2 \iq \wedge \diff \overline{q}_2) \right|_{\mathcal{H} \oplus \mathcal{E}} .
    \end{equation*}
    Notice that if $\dot q = (\dot q_1, \dot q_2)$ we have 
    \begin{equation*}
        \iota_{\dot q}  \Re(\diff q_1 \iq \wedge \diff \overline{q}_1 + \diff q_2 \iq \wedge \diff \overline{q}_2) = 2 \Re(\dot{q}_1 \iq \diff \overline{q}_1 + \dot{q}_2 \iq \diff \overline{q}_2).
    \end{equation*}
    From this observation and \eqref{eq: genvfxi}, we deduce that if $\lambda \in \Im(\H)$, 
    \begin{align*}
        \iota_{\lambda \xi} \diff \xi^0 & = - \frac{2}{\mr^2} \left. \Re(\lambda q_1 \diff \overline{q}_1 + \lambda q_2 d\overline{q}_2) \right|_{\mathcal{E}\oplus \mathcal{H}} \\
            & = - \frac{2}{\mr} \left.  g^\m{st}(\lambda \partial_{\mr}, \cdot)\right|_{\mathcal{E}\oplus \mathcal{H}} \\
            & = - \frac{2}{\mr}(g^\m{st}(\lambda \partial_{\mr}^\mathcal{E}, \cdot) + g^\m{st}((\lambda \partial_{\mr})^\mathcal{H}, \cdot)) .
    \end{align*}
    On the one hand,
    \begin{equation*}
        g^\m{st}(\lambda \partial_{\mr}^\mathcal{E}, \cdot) = \frac{2\sqrt{\upsilon}}{\mr^2} \sigma^0 \circ \lambda^{-1} = \frac{1}{\mr^3} \diff^c_\lambda \upsilon,
    \end{equation*}
    and on the other hand, one sees from \eqref{eq: Projection of eta onto H} that
    \begin{align*}
        (I\partial_r)^\mathcal{H} & = \frac{2}{\mr^3}(\mu^\m{st}_2 K\xi - \mu_3^\m{st} J \xi) , \\
        (J\partial_r)^\mathcal{H} & = \frac{2}{\mr^3}(\mu^\m{st}_3 I\xi - \mu_1^\m{st} K \xi) , \\
        (K\partial_r)^\mathcal{H} & = \frac{2}{\mr^3}(\mu^\m{st}_1 J\xi - \mu_2^\m{st} I \xi) .
    \end{align*}
    These computations account for the first two lines of the expression of $\diff \xi^0$. For the third and fourth lines, remark that 
    \begin{equation*}
        \iota_{\partial_{\mr}} \diff \xi^0 = \frac{2}{\mr^3} \left. \Re(q_1 \iq \diff \overline{q}_1 + q_2 \iq \diff \overline{q}_2) \right|_{\mathcal{H} \oplus \mathcal{E}} = \frac{2}{\mr^3} \left . g^\m{st}(\xi,\cdot) \right|_{\mathcal{H} \oplus \mathcal{E}} = 0 
    \end{equation*}
    and therefore 
    \begin{align*}
        \left. \iota_{\partial_{\mr}^\mathcal{E}} \diff \xi^0 \right|_{\mathcal{E}} = - \left. \iota_{\partial_{\mr}^\mathcal{H}} \diff \xi^0 \right|_{\mathcal{E}} & = - \frac{2}{\mr^3} \left. (\mu_1 \diff \xi^0(I\xi,\cdot) + \mu_2 \diff \xi^0(J\xi,\cdot) + \mu_3 \diff \xi^0(K\xi,\cdot) ) \right|_{\mathcal{E}} \\
        & = \frac{4}{\mr^7} (\mu_1 \diff^c_1 \upsilon + \mu_2 \diff^c_2 \upsilon + \mu_3 \diff^c_3 \upsilon) .
    \end{align*}
    Since $\partial_{\mr}^\mathcal{E} = \tfrac{2\sqrt{\upsilon}}{\mr^2} \sigma_0$ and $\sigma^0 = \tfrac{1}{2\mr \sqrt{\upsilon}} \diff \upsilon$ we deduce the third line. Finally, the restriction of $\diff \xi^0 = - \iq \diff \mathfrak{a}$ to $\mathcal{E}$ is anti-self-dual\footnote{This is because the restriction of $\diff \mathfrak{a}$ to the Eguchi--Hanson fibres must be of type $(1,1)$ with respect to all the quaternionic complex structures, which amounts to saying that it is anti-self-dual for the associated metric.}, which yields the desired result.
\end{proof}

We will only need the restriction to $\mZ_s \cong \mZ_+$ of the previous expressions, where in particular $\mu_1 = h = s$ and $\diff \mu_1 = 0$ (where we assume $s \geq 0$ throughout). We have seen in the previous section that under the identification of $\mZ_s$ with $\mZ_+$, the fibration $\zeta_s$ can be identified with the fixed holomorphic fibration $\zeta \colon \mZ_+ \to \C$ which maps $W \in \mZ_+ \setminus \mC_+ \cong \mZ^{\m{reg}} \subset \mathfrak{gl}(2,\C)$ to $\zeta(W) = \iq (w_{12}-w_{21})$. Following Li's notations in \cite{li2019new}, we will also denote by $\eta$ the squared modulus of $\zeta$:
\begin{equation}
    \eta \coloneqq |\zeta|^2 .
\end{equation}
Conveniently, the restriction of the function $\upsilon$ to $\mZ_s$ turns out to also be independent of $s$ under the identification $\mZ_s \cong \mZ_+$. Indeed, we have seen in \S\ref{subsec: Relation to the flop} that $\mr_s^2 = 2 \sqrt{R^2+s^2}$, and therefore $4R^2 + 4s^2 = \mr_s^4 = 4(\upsilon+\eta+s^2)$, whence 
\begin{equation*}
    \upsilon = R^2-\eta
\end{equation*}
is also independent of $s$. In the remainder of this section, it will be convenient to keep denoting by $\mr_s$ the function satisfying $\mr^2_s =  2 \sqrt{\upsilon+\eta +s^2}$, which we now consider as a family of functions on $\mZ_+$ depending on the parameter $s$.

By the above computations, we have obtained on the open set $\{\upsilon > 0\} \subset \mZ_+$ an $s$-dependent family of co-frames, $\diff \upsilon, \diff^c \upsilon, \sigma^2_s, \sigma^3_s, \diff (\Re(\zeta)), \diff (\Im(\zeta))$ (where under the identification $\mZ_s \cong \mZ_+$, $\sigma_s^\ell$ corresponds to the restriction of $\sigma^\ell$ to $\mZ_s$), or after complexification, $\partial \upsilon, \overline{\partial} \upsilon, \sigma_s, \overline{\sigma}_s, \diff \zeta, \diff \overline{\zeta}$, where $\sigma_s = \sigma^2_s + \iq \sigma^3_s$. Let us emphasize that while $\partial \upsilon,  \overline{\partial}\upsilon, \diff \zeta, \diff \overline{\zeta}$ do not depend on $s$, the $1$-forms $\sigma_s, \overline{\sigma}_s$ genuinely do, because the hyperkähler complex structures $J$ and $K$ on the fibres of $\zeta$ depend on which Eguchi--Hanson space they are identified with. In any case, with these notations we can rewrite
\begin{equation*}
    \hat{\omega}_s = \sigma^0_s \wedge \sigma^1_s + \sigma^2_s \wedge \sigma^3_s = \frac{1}{4r^2_s \upsilon} \diff \upsilon \wedge \diff^c \upsilon + \sigma^2_s \wedge \sigma^3_s = \frac{\iq}{2r^2_s \upsilon} \partial \upsilon \wedge \overline{\partial} \upsilon + \frac{\iq}{2} \sigma_s \wedge \overline{\sigma}_s
\end{equation*}
since $\partial \upsilon = \tfrac{1}{2} (\diff \upsilon + \iq \diff^c \upsilon)$ and $\overline{\partial} \upsilon = \tfrac{1}{2}(\diff \upsilon - \iq \diff^c \upsilon)$. In particular,

\begin{cor}     \label{cor: Writing the forms conveniently}
	For any $s \geq 0$, the following identities hold on $\{\upsilon > 0\} \subset \mZ_+$:
    \begin{enumerate}[(i)]
        \item The Kähler form $\omega^{\m{srf}}_s$ satisfies
        \begin{equation*}
            \omega^{\m{srf}}_s = \frac{\iq}{2} \left( \frac{1}{\upsilon \mr_s^2} \partial \upsilon \wedge \overline{\partial} \upsilon + \sigma_s \wedge \overline{\sigma}_s + (1+\mr_s^{-2}) \diff \zeta \wedge \diff \overline{\zeta} \right) .
        \end{equation*}

        \item The curvature $\iq \Theta_s$ satisfies
        \begin{align*}
            \iq \Theta_s  & = - \frac{2\iq s}{\upsilon \mr^6_s} \partial \upsilon \wedge \overline{\partial} \upsilon + \frac{2 \iq s}{\mr^4_s} \sigma_s \wedge \overline{\sigma}_s - \frac{2 \iq s}{\mr^6_s} \diff \zeta \wedge \diff \overline{\zeta}  \\
    	    & ~~~~ + \frac{2\sqrt{\upsilon}}{\mr^5_s} (\diff \zeta \wedge \overline{\sigma}_s + \diff \overline{\zeta} \wedge \sigma_s) - \frac{2}{\sqrt{\upsilon}\mr^5_s} (\zeta \cdot  \partial \upsilon \wedge  \overline{\sigma}_s +  \overline{\zeta} \cdot \overline{\partial} \upsilon \wedge  \sigma_s).
        \end{align*}

        \item The function $\upsilon$ satisfies
        \begin{align*}
	       \iq \partial \overline{\partial} \upsilon & = \frac{\iq}{2\upsilon \mr_s^2} \left( \mr_s^2 + \frac{4\upsilon}{\mr_s^2} + \frac{4s^2}{\mr_s^2} \right) \partial \upsilon \wedge \overline{\partial} \upsilon + \frac{\iq}{2} \left( \mr_s^2 - \frac{4s^2}{\mr_s^2} \right) \sigma_s \wedge \overline{\sigma}_s \\
            & ~~~~ - \frac{2\upsilon \iq}{\mr_s^4} \diff \zeta \wedge \diff \overline{\zeta} + \frac{2 \iq}{\mr_s^4}(\zeta \partial \upsilon \wedge \diff \overline{\zeta} + \overline{\zeta} \diff \zeta \wedge \overline{\partial} \upsilon) \\
            & ~~~~ - \frac{2 s \sqrt{\upsilon}}{\mr_s^3} (\diff \zeta \wedge \overline{\sigma}_s + \diff \overline{\zeta} \wedge \sigma_s) + \frac{2s}{\sqrt{\upsilon}\mr_s^3} (\zeta \cdot  \partial \upsilon \wedge  \overline{\sigma}_s +  \overline{\zeta} \cdot \overline{\partial} \upsilon \wedge  \sigma_s) .
	\end{align*}
    \end{enumerate}
\end{cor}

\begin{proof}
	The first two points are just rewriting the previous expressions under the identification $\mZ_+ \cong \mZ_s$. For the proof of the third point, we have on $\mZ_s$
	\begin{equation*}
		\partial \overline{\partial} \upsilon = \partial \overline{\partial} (\tfrac{1}{4} \mr^4_s - |\zeta|^2 - s^2) = \tfrac{1}{4} \partial \overline{\partial} (\mr^4_s) - \diff \zeta \wedge \diff \overline{\zeta} .
	\end{equation*}
	Moreover,
	\begin{equation*}
		\partial \overline{\partial} (\mr^4_s) = 2 \partial(\mr^2_s) \wedge \overline{\partial}(\mr^2_s) + 2 \mr^2_s \partial \overline{\partial} (\mr^2_s) ,
	\end{equation*}
	and since $\mr^2_s = 2 \sqrt{\upsilon + |\zeta|^2 + s^2}$ we obtain
	\begin{equation*}
		\partial(\mr^2_s) = \frac{2}{\mr^2_s} (\partial \upsilon + \overline{\zeta} \diff \zeta), ~~ \overline{\partial}(\mr^2_s) = \frac{2}{\mr^2_s} (\overline{\partial} \upsilon + \zeta \diff \overline{\zeta})
	\end{equation*}
	From this we deduce that
	\begin{align*}
		\iq \partial \overline{\partial} \upsilon & = \mr_s^2 \frac{\iq}{2} \partial \overline{\partial}(\mr^2_s) + \frac{2\iq}{\mr_s^4} (\partial \upsilon + \overline{\zeta} \diff \zeta) \wedge (\overline{\partial} \upsilon + \zeta \diff \overline{\zeta}) - \iq \diff \zeta \wedge \diff \overline{\zeta} .
	\end{align*}
	Now by Proposition \ref{prop: Properties of reduction}, $\frac{\iq}{2}\partial \overline{\partial} (\mr^2_s) = -s\iq \Theta_s + \omega^\m{red}_s$, and thus the proposition follows after adding up all the terms.
\end{proof}

One may deduce several useful identities from this proposition. First, the volume form $\vol_+ = (1+\mr_s^{-2})^{-1} (\omega^{\m{srf}}_s)^3$ is given by
\begin{equation}    \label{eq: Volume form in invariant frame}
    \vol_+ = - \frac{3\iq}{4 \upsilon \mr_s^2}\partial \upsilon \wedge \overline{\partial} \upsilon \wedge \diff \zeta \wedge \diff \overline{\zeta} \wedge \sigma_s \wedge \overline{\sigma}_s .
\end{equation}
In addition, the square of the Kähler form reads
\begin{multline}        \label{eq: Square of Kahler form}
    \omega^{\m{srf}}_s\wedge \omega^{\m{srf}}_s = -\frac{1}{2} \Big( \frac{1}{\upsilon \mr_s^2} \partial \upsilon \wedge \overline{\partial} \upsilon \wedge \sigma_s \wedge \overline{\sigma}_s + \left( \frac{1}{\upsilon \mr_s^2} + \frac{1}{\upsilon \mr^4_s}\right) \partial \upsilon \wedge \overline{\partial} \upsilon \wedge \diff \zeta \wedge \diff \overline{\zeta}  \\  + \left(1 + \frac{1}{\mr_s^2}\right) \diff \zeta \wedge \diff \overline{\zeta} \wedge \sigma_s \wedge \overline{\sigma}_s  \Big) .
\end{multline}
Examining the expression of $\iq \partial \overline{\partial}\upsilon$, we see that the wedge product of the last three terms with $(\omega^{\m{srf}}_s)^2$ vanish. Thus only the first three terms contribute to $(\omega^{\m{srf}}_s)^2 \wedge \iq \partial \overline{\partial}\upsilon$, and we see that
\begin{multline}        \label{eq: Product of omega squared and iddbar upsilon}
    \omega_s^{\m{srf}} \wedge \omega_s^{\m{srf}}  \wedge \iq \partial \overline{\partial} \upsilon =  \frac{1}{3} \left(-\frac{4\upsilon}{\mr^4_s} + \left(1 + \frac{1}{\mr_s^2} \right) \left( \mr^2_s + \frac{4\upsilon}{\mr^2_s} + \frac{4s^2}{\mr^2_s} + \mr^2_s - \frac{4s^2}{\mr_s^2}\right) \right) \vol_+ \\
        = \frac{2}{3} \left( \mr_s^2 + 1 + \frac{2\upsilon}{\mr^2_s} \right) \vol_+.
\end{multline}
Another straightforward computation yields
\begin{align*}
    (\iq \partial \overline{\partial}\upsilon)^2 = & - \frac{4\upsilon+2\eta}{\upsilon \mr^2_s} \partial \upsilon \wedge \overline{\partial} \upsilon \wedge \sigma_s \wedge \overline{\sigma}_s + \frac{4}{\mr^4_s} \partial \upsilon \wedge \overline{\partial} \upsilon \wedge \diff \zeta \wedge \diff \overline{\zeta} \\
        & + \frac{2\upsilon}{\mr^2_s} \sigma_s \wedge \overline{\sigma}_s \wedge \diff \zeta \wedge \diff \overline{\zeta} - \frac{4s \iq }{\sqrt{\upsilon}\mr^3_s} \partial \upsilon \wedge \overline{\partial} \upsilon \wedge (\diff \zeta \wedge \overline{\sigma}_s + \diff \overline{\zeta} \wedge \sigma_s ) \\
        & - \frac{2}{\mr^2_s} \sigma_s \wedge \overline{\sigma}_s \wedge (\zeta \partial \upsilon \wedge \diff \overline{\zeta} + \overline{\zeta} \diff \zeta \wedge \overline{\partial} \upsilon ) .
\end{align*}
and we can deduce that
\begin{equation}        \label{eq: Wedge prod of omega and ddbar upsilon squared}
    \omega^{\m{srf}}_s \wedge (\iq \partial \overline{\partial}\upsilon)^2  = \frac{4}{3} \left((2\upsilon+\eta) + \frac{\eta-\upsilon}{\mr^2_s} \right) \vol_+ .
\end{equation}

    \subsection{The distance function and uniform comparisons}      \label{subsec: The distance function}

Let us pick a base-point $o_+ \in \mC_+ \subset \mZ_+$.

\begin{prop}    \label{prop: distance estimate}
    If $s > 0$, $\omega^{\m{srf}}_s$ is a complete Kähler metric on $\mZ_+$. Moreover, there exist constants $\kappa > 1$ and $D > 0$ independent of $s$ such that, on the domain $\{\upsilon + \eta \geq 1\}$, 
    \begin{equation*}
        \kappa^{-1} (\eta^{1/2} + \upsilon^{1/4}) - D s^{1/2} \leq \mathrm{dist}_{\omega^{\m{srf}}_s}(o_+, \cdot) \leq \kappa (\eta^{1/2} + \upsilon^{1/4}) + D s^{1/2}.
    \end{equation*}
\end{prop}

\begin{proof}
    Let us identify $\mZ_+$ with $\mZ_s$ and denote by $o_s$ the corresponding base point in $\mC_s$. We can lift $o_s$ to a base-point $\tilde{o}_s$ of the $\Un(1)$-bundle $\P_s$ over $\mZ_s$. Completeness is clear since $\omega^{\m{srf}}_s \geq \omega_s^\m{red}$ and $\omega^\m{red}_s$ is complete. In addition, we obtain the inequality 
    \begin{equation*}
        \mathrm{dist}_{\omega^{\m{srf}}_s}(o_s,\cdot) \geq \dist_{\omega_s^\m{red}}(o_s, \cdot) .
    \end{equation*}
    On the other hand, if $x \in \mZ_s$ is another point and $\tilde{x} \in \P_s \subset \H^2$ is a lift of $x$, that is, $x = [\tilde{x}]$, then it is clear that\footnote{This follows from the general fact that, in Riemannian quotients, horizontal lifts of geodesics remain geodesics, cf. \cite{oneil1967submersions}.}
    \begin{equation*}
        \dist_{\omega_s^\m{red}}(o_s,x) \geq \inf_{e^{\iq \theta} \in \Un(1)} \dist_{\P_s}(\tilde{o}_s,\tilde{x} \cdot e^{\iq \theta}) \geq \inf_{e^{\iq \theta} \in \Un(1)} \dist_{\H^2}(\tilde{o}_s,\tilde{x} \cdot e^{\iq \theta}) .
    \end{equation*}
    On the other hand, for any $e^{\iq \theta} \in \Un(1)$, 
    \begin{equation*}
        \dist_{\H^2}(\tilde{o}_s,\tilde{x} \cdot e^{\iq \theta}) \geq |\tilde{x} \cdot e^{\iq \theta}| - |\tilde{o}_s| = \mr_s(x) - \mr_s(o_s) .
    \end{equation*}
    Thus we obtain the inequality
    \begin{equation*}
        \dist_{\omega^{\m{srf}}_s}(o_s,\cdot) \geq \mr_s - \mr_s(o_s) = \sqrt{2}((\upsilon+\eta+s^2)^{1/4} - \sqrt{s}) \geq \sqrt{2} (\upsilon^{1/4} - s^{1/2}) .
    \end{equation*}
    On the other hand, we also have $\omega^{\m{srf}}_s \geq \frac{\iq}{2} \diff \zeta \wedge \diff \overline{\zeta}$, so that 
    \begin{equation*}
        \dist_{\omega^{\m{srf}}_s}(o_s,\cdot) \geq |\zeta| - |\zeta(o_s)| = \eta^{1/2} .
    \end{equation*}
    Thus we obtain the first inequality, $\dist_{\omega^{\m{srf}}_s}(o_s,\cdot) \geq \frac{1}{2}(\sqrt{2} \upsilon^{1/4} + \eta^{1/2} - \sqrt{2s})$.

    To obtain the other inequality, we take a specific path from $x$ to $o_s$ in $\mZ_s$. First, we descend from $x$ to $y_1$ along the radial geodesic inside the fibre $\mu^{-1}(\zeta(x)+\iq s)$. This path has length
    \begin{equation*}
        \frac{1}{2\sqrt{2}}\int_0^\upsilon \frac{\diff \tau}{\sqrt{\tau} (\tau+\eta+s^2)^{1/4}} \leq \frac{1}{2\sqrt{2}} \int_0^\upsilon \frac{\diff \tau}{\tau^{3/4}} \leq \sqrt{2} \upsilon^{1/4} .
    \end{equation*}
    Then, we take from $y_1$ a horizontal lift of the path $(1-t)\zeta(x)$ from $\zeta(x)$ to $0$ in $\C$ inside the locus $\{\mr_s^2 = 2\sqrt{\eta+s^2}\}$, e.g. the locus formed by the spheres, with end-point $y_2 \in \mC_s$. The restriction of the metric associated with $\omega^{\m{srf}}_s$ to this locus is a Riemannian fibration over the metric on $\C$ which can be written in radial coordinates as
    \begin{equation*}
        \left( 1+ \frac{1}{2\sqrt{\eta+s^2}} \right) (\diff \eta^2 + \diff \theta^2)
    \end{equation*}
    Therefore, the velocity along the horizontal lift of the path $(1-t)\zeta(x)$ is given by
    \begin{align*}
    \left(1+\frac{1}{2\sqrt{(1-t)^2|\zeta(x)|^2+s^2}}\right)^{1/2} |\zeta(x)| .
    \end{align*}
    Therefore the length of this piece of the path is given by
    \begin{align*}
        |\zeta(x)| \int_0^1 \left(1+\frac{1}{2\sqrt{t^2|\zeta(x)|^2+s^2}}\right)^{1/2} \diff t & =\int_0^{|\zeta(x)|}\left(1+\frac{1}{2\sqrt{\tau^2+s^2}}\right)^{1/2}\diff \tau \\
            & \leq \int_0^{|\zeta(x)|} \left( 1 + \frac{1}{\sqrt{2} (\tau^2 + s^2)^{1/4}}\right) \diff \tau  \\
            & \leq |\zeta(x)| + \frac{1}{\sqrt{2}} \int_0^{|\zeta(x)|} \frac{\diff \tau}{\tau^{1/2}}. \\
            & \leq |\zeta(x)| + \sqrt{2}|\zeta(x)|^{1/2} = \eta^{1/2} + \sqrt{2} \eta^{1/4} .
    \end{align*}
    Finally, we can take a geodesic from $y_2$ to $o_s$ inside the rational curve $\mC_s$, whose length is uniformly bounded by $\m{diam}(\mC_s) = D s^{1/2}$ for some constant $D$ independent of $s$. Thus we obtain
    \begin{equation*}
        \dist_{\omega^{\m{srf}}_s}(o_s,\cdot) \leq \eta^{1/2} + \sqrt{2}(\upsilon^{1/4} + \eta^{1/4}) + D s^{1/2} .
    \end{equation*}
    Now in the region $\{\upsilon + \eta \geq 1\}$, there exists a constant $\kappa > 1$ such that
    \begin{equation*}
        \eta^{1/2} + \sqrt{2}(\upsilon^{1/4} +\eta^{1/4}) \leq \sqrt{2} + (1+\sqrt{2}) \eta^{1/2} + \sqrt{2} \upsilon^{1/4} \leq \kappa (\eta^{1/2} + \upsilon^{1/4})
    \end{equation*}
    which finishes the proof.
\end{proof}

\begin{rem}
    The distance estimate and its proof of course extend to the case $s = 0$ by considering the distance to the vertex $o \in \mZ$, i.e. $o = 0$ when we regard $\mZ \subset \mathfrak{gl}(2,\C)$. 
\end{rem}

Let us introduce the function
\begin{equation}
    \rho \coloneqq (\upsilon+\eta^2)^{\frac{1}{4}} .
\end{equation}
on $\mZ_+$. Then in the region $\{\upsilon + \eta \geq 1\} \subset \mZ_+$ we have $\rho \geq \frac{3}{4} > 0$, and $\rho$ is uniformly comparable to the distance function $\dist_{\omega^{\m{srf}}_s}(o_+,\cdot)$, where the constant of comparison is by the previous proposition independent of $s$. Henceforth, it will also be convenient to denote the function $\mr_0$ merely by $\mr$ to lighten notations; since from now on we will always be working on $\mZ_+$ and not think of it as embedded into $\M$, this should not create any confusions. In particular:
\begin{equation}
    \mr^4 \coloneqq \mr^4_0 = 4(\upsilon+\eta) = 4 R^2.
\end{equation}
In the remainder of this section, it will be important to make asymptotic estimates on quantities related to $\omega^{\m{srf}}_s$ in an $s$-independent way, at least in the limit where $s \to 0$. For this purpose we shall introduce the following notation:

\begin{Def}
    Let $F,H \colon [0,+\infty)^3 \to \R$ be two functions such that $H(s,\upsilon,\eta) \geq 0$ in the region $\{\upsilon+\eta \geq 1\}$. 
    We will use the notation
    \begin{equation*}
        F(s,\upsilon,\eta) = \ohat(H(s,\upsilon,\eta))
    \end{equation*}
    (or simply $F = \ohat(H)$) if, for any $s_0 > 0$, there exists a constant $C > 0$ such that
    \begin{equation*}
        \forall s \in [0,s_0], \forall \upsilon,\eta \geq 0, ~~ \upsilon + \eta \geq 1 \implies |F(s,\upsilon,\eta)| \leq C H(s,\upsilon,\eta) .
    \end{equation*}
    Moreover if $U \subseteq [0,\infty)^2$, we will say that $|F| = \ohat(H)$ in the region $U$ if the above condition is satisfied with the additional restriction that $(\upsilon,\eta) \in U$.
\end{Def}

\begin{rem}
    We will naturally use the notation $|F| = \ohat(H)$ when $F(s,\cdot)$, $H(s,\cdot)$ are $1$-parameter families of functions on $\mZ_+$ which are invariant under the action of $\Un(2)$, since they can be written as functions of $\upsilon$ and $\eta$ only.
\end{rem}

\begin{rem}
    We will reserve the notation $F = \mo(H)$ for the cases when $F,H$ are functions which depend on $\upsilon,\eta$ only but not on $s$.
\end{rem}

\begin{lem}     \label{lem: uniform comparisons}
    The following uniform comparisons hold:
    \begin{equation*}
        \mr_s = \ohat(\mr), ~~ \mr = \ohat(\mr_s), ~~ \rho = \mo(\mr^2), ~~ \rho^2 = \mo((1+\eta)^{1/2}\mr^2), ~~ \mr = \mo(\rho).
    \end{equation*}
    Moreover,
    \begin{equation*}
        |\sigma_s|_{\omega_s^{\m{srf}}} = \ohat(1), ~~ |\diff \zeta|_{\omega_s^{\m{srf}}} = \ohat(1), ~~ |\diff \upsilon|_{\omega_s^{\m{srf}}} = \ohat(\sqrt{\upsilon}\mr_s), ~~ |\partial \overline{\partial}\upsilon|_{\omega_s^{\m{srf}}} = \ohat(\mr_s^2),
    \end{equation*}
    and
    \begin{equation*}
        |\omega_s^{\m{srf}} \wedge \partial \overline{\partial} \upsilon \wedge (\zeta \partial \upsilon \wedge \diff \overline{\zeta} + \overline{\zeta} \diff \zeta \wedge \overline{\partial} \upsilon)|_{\omega_s^{\m{srf}}} = \ohat(\eta \mr^2) .
    \end{equation*}
\end{lem}

\begin{proof}
    The comparisons $\mr_s = \ohat(\mr)$ and $\mr = \ohat(\mr_s)$ are clear since they concern the region $\{\upsilon + \eta \geq 1\}$. The comparisons $\rho = \mo(\mr^2)$ and $\mr = \mo(\rho)$ are also obvious from the definitions. Finally $(1+\eta)^{1/2} \mr^2 \gtrsim \mr^2 + \eta^{1/2} \mr^2 \gtrsim \upsilon^{1/2} + \eta \gtrsim \rho^2$, so that $\rho^2 = \mo((1+\eta)^{1/2} \mr^2)$.

    For the inequality on the norms, the comparisons $|\sigma_s|_{\omega_s^{\m{srf}}} = \ohat(1)$, $|\diff \zeta|_{\omega_s^{\m{srf}}} = \ohat(1)$, and $|\diff \upsilon|_{\omega_s^{\m{srf}}} = \ohat(\sqrt{\upsilon}\mr_s)$ are immediate from the expression of $\omega_s^{\m{srf}}$ given in Corollary \ref{cor: Writing the forms conveniently}. Using these comparisons and point (iii) in the same corollary, and observing that $\upsilon + \eta = \ohat(\mr_s^4)$, $\zeta = \mo(\eta^{1/2})$, it is not difficult to see that $|\partial \overline{\partial} \upsilon|_{\omega_s^{\m{srf}}} = \ohat(\mr_s^2)$. Finally, an explicit computation shows that there exists a constant $C$ independent of $s$ such that
    \begin{equation*}
        \omega_s^{\m{srf}} \wedge \partial \overline{\partial} \upsilon \wedge (\zeta \partial \upsilon \wedge \diff \overline{\zeta} + \overline{\zeta} \diff \zeta \wedge \overline{\partial} \upsilon) = \frac{C\eta\upsilon}{\mr_s^2} \vol_+ = \ohat(\eta \mr^2) \vol_+
    \end{equation*}
    and since $(\omega^{\m{srf}}_s)^3 = (1+\mr_s^{-2})\vol_+$ this yields the desired result.
\end{proof}

    \subsection{The Monge--Amp\`ere equation for symmetric potentials}      \label{subsec: The MA equation}

In order to solve the Monge--Amp\`ere equation on $\mZ_+$ (or find approximate solutions thereof), we may seek potentials invariant under the action of $\Un(2)$. Using the $\Un(2)$-invariant functions $\upsilon,\eta$ which are smooth on $\mZ_+$, we consider families of potentials of the form $\phi_s = F(s,\upsilon,\eta)$. Then it is straightforward to compute:
\begin{align}
    \overline{\partial}\phi_s & = F_\upsilon \overline{\partial} \upsilon + \zeta F_\eta \diff \overline{\zeta} , \\
    \partial \overline{\partial}\phi_s & = F_\upsilon \partial \overline{\partial} \upsilon + F_{\upsilon\upsilon} \partial \upsilon \wedge \overline{\partial} \upsilon + (F_\eta + \eta F_{\eta\eta}) \diff \zeta \wedge \diff \overline{\zeta} + F_{\upsilon\eta} (\zeta \partial \upsilon \wedge \diff \overline{\zeta} + \overline{\zeta} \diff \zeta \wedge \overline{\partial} \upsilon ) .
\end{align}
For our purpose, it will be important to construct a family of approximate solutions of the Monge--Amp\`ere equation $(\omega_s^{\m{srf}}+\iq \partial \overline{\partial}\phi_s)^3 = \vol_+$ which not only decay faster than quadratically, but do so in an $s$-uniform way (at least when $s$ remains bounded, and in particular in the limit $s \to 0$). In order to achieve this and simplify the computations, it will be useful to restrict the class of potentials to families of functions $F(s,\upsilon,\eta)$ whose derivatives have reasonable decay properties, uniformly in $s$. For this purpose, we shall adopt the following terminology:

\begin{Def}     \label{def: admissible potentials}
    A $1$-parameter family of symmetric potentials $\{\phi_s = F(s,\upsilon,\eta)\}_{s \geq 0}$ on $\mZ_+$ will be called \emph{admissible} if the derivatives of the function $F$ satisfy
    \begin{align*}
        |F_\upsilon| &= \ohat(\mr^{-4}), ~~ |F_{\upsilon\upsilon}| = \ohat(\mr^{-8}), ~~ |F_\eta| = \ohat(\rho^{-2}), \\
        & ~~~~ |F_{\upsilon\eta}| = \ohat(\mr^{-4}\rho^{-2}), ~~ |F_{\eta\eta}| = \ohat(\rho^{-4}) .
    \end{align*}
\end{Def}

\begin{lem}     \label{lem: bound on ddbar for admissible potentials}
    Let $\{\phi_s\}_{s \geq 0}$ be an admissible family of symmetric potentials. Then 
    \begin{equation*}
        |\partial\overline{\partial}\phi_s|_{\omega^{\m{srf}}_s} = \ohat(\mr^{-2}) = \ohat(\rho^{-1}) .
    \end{equation*}
\end{lem}

\begin{proof}
    Let us estimate each term separately. Using Lemma \ref{lem: uniform comparisons} we have $|F_\upsilon \partial \overline{\partial}\upsilon| = \ohat(\mr^{-2}) = \ohat(\rho^{-1})$ and $|F_{\upsilon\upsilon} \partial \upsilon \wedge \overline{\partial}\upsilon| = \ohat(\mr_s^{-8} |\partial \upsilon \wedge \overline{\partial}\upsilon|) = \ohat(\upsilon \mr_s^{-6}) = \ohat(\mr^{-2}) = \ohat(\rho^{-1})$ since $\upsilon = \ohat(\mr_s^4)$. In addition, $|F_\eta+\eta F_{\eta\eta}| = \ohat(\rho^{-2})$ since $\eta = \ohat(\rho^2)$, and $|\diff \zeta \wedge \diff \overline{\zeta}| = \ohat(1)$, and therefore $|(F_\eta+\eta F_{\eta\eta}) \diff \zeta \wedge \diff \overline{\zeta}| = \ohat(\rho^{-2}) = \ohat(\mr^{-2}) = \ohat(\rho^{-1})$. Finally,
    \begin{equation*}
        |F_{\upsilon\eta} (\zeta \partial \upsilon \wedge \diff \overline{\zeta} + \overline{\zeta} \diff \zeta \wedge \overline{\partial} \upsilon )| = \ohat(|F_{\upsilon\eta}| \cdot |\zeta| \cdot \sqrt{\upsilon}\mr_s) = \ohat(\mr^{-1}_+ \rho^{-1}) = \ohat(\mr^{-2}) = \ohat(\rho^{-1})
    \end{equation*}
    since $|\zeta| = \mo(\sqrt{\eta}) = \mo(\rho)$, $\sqrt{\upsilon} = \mo(\mr^2)$ and $\rho^{-1} = \mo(\mr^{-1})$. The lemma follows.
\end{proof}

\begin{prop}        \label{prop: MA equation}
    Let $\{\phi_s = F(s,\upsilon,\eta)\}_{s \geq 0}$ be an admissible family of symmetric potentials on $\mZ_+$. Then
    \begin{equation*}
        (\omega^{\m{srf}}_s + \tfrac{\iq}{2} \partial \overline{\partial}\phi_s)^3 = (1 + \mr_s^{-2} + \mA_s(F) + \mB_s(F) + \ohat(\rho^{-3})) \cdot \vol_+,
    \end{equation*}
    where
    \begin{align*}
        \mA_s(F) & = \eta F_{\eta\eta} + F_\eta + (\mr_s^2 + 1)\upsilon F_{\upsilon\upsilon} + \left( \mr^2_s + 1 + \frac{2\upsilon}{\mr_s^2} \right) F_\upsilon , \\
        \mB_s(F) & = ( 2\upsilon + \eta) F_\upsilon^2 + \frac{\upsilon \mr^4_s}{2} F_\upsilon F_{\upsilon\upsilon} .
    \end{align*}
\end{prop}

\begin{proof}
    Using binomial expansion, we have
    \begin{align*}
        (\omega^{\m{srf}}_s + \tfrac{\iq}{2} \partial \overline{\partial}\phi_s)^3 & = (\omega^{\m{srf}}_s)^3 + \tfrac{3}{2} (\omega^{\m{srf}}_s)^2 \wedge \iq \partial \overline{\partial} \phi_s +  \tfrac{3}{4} \omega^{\m{srf}}_s \wedge (\iq \partial \overline{\partial}\phi_s)^2 + \tfrac{1}{8}(\iq\partial \overline{\partial}\phi_s)^3 .
    \end{align*}
    By the previous lemma $|\partial\overline{\partial}\phi_s|_{\omega_s^{\m{srf}}} = \ohat(\rho^{-1})$, and since $(\omega_s^{\m{srf}})^3 / \vol_+ = 1+\mr_s^{-2} = \ohat(1)$ we deduce that $(\iq\partial\overline{\partial}\phi_s)^3/\vol_+ = \ohat(\rho^{-3})$ so the last term already decays fast enough. 
    
    Using \eqref{eq: Product of omega squared and iddbar upsilon} and the fact that $(\omega^{\m{srf}}_s)^2 \wedge \partial \upsilon \wedge \diff \overline{\zeta} = (\omega^{\m{srf}}_s)^2 \wedge \overline{\partial} \upsilon \wedge \diff \zeta = 0$, we see that
    \begin{align*}
        \frac{3}{2}(\omega^{\m{srf}}_s)^2 \wedge \iq \partial \overline{\partial} \phi_s & =  \frac{3\iq}{2} F_\upsilon (\omega^{\m{srf}}_s)^2 \wedge \partial \overline{\partial} \upsilon + \frac{3\iq}{2} F_{\upsilon\upsilon} (\omega^{\m{srf}}_s)^2 \wedge \partial \upsilon \wedge \overline{\partial}\upsilon \\ & ~~~~ + \frac{3\iq}{2}(F_\eta + \eta F_{\eta\eta}) (\omega^{\m{srf}}_s)^2 \wedge \diff \zeta \wedge \diff \overline{\zeta} \\
            & = \left( \mr_s^2 + 1 + \frac{2\upsilon}{\mr^2_s} \right) F_\upsilon \vol_+ + \frac{3\iq}{2} F_{\upsilon\upsilon} (\omega^{\m{srf}}_s)^2 \wedge \partial \upsilon \wedge \overline{\partial}\upsilon \\
            & ~~~~ + \frac{3\iq}{2}(F_\eta + \eta F_{\eta\eta}) (\omega^{\m{srf}}_s)^2 \wedge \diff \zeta \wedge \diff \overline{\zeta} .
    \end{align*}
    On the other hand, \eqref{eq: Volume form in invariant frame} and \eqref{eq: Square of Kahler form} imply
    \begin{equation*}
        \frac{3\iq}{2} (\omega^{\m{srf}}_s)^2 \wedge \partial \upsilon \wedge \overline{\partial} \upsilon = \upsilon \mr^2_s(1 + \mr_s^{-2}) \vol_+ = (\mr_s^2 + 1) \upsilon \vol_+ 
    \end{equation*}
    and
    \begin{equation*}
        \frac{3\iq}{2} (\omega^{\m{srf}}_s)^2 \wedge \diff \zeta \wedge \diff \overline{\zeta} = \vol_+
    \end{equation*}
    whence we deduce the expression for $\mA_s(F) = \tfrac{3\iq}{2} (\omega^{\m{srf}}_s)^2 \wedge \partial \overline{\partial}\phi_s / \vol_+$.
    
    For $\mB_s(F)$, we have
    \begin{align*}
    	(\iq \partial \overline{\partial} \phi_s)^2 = ~ & F_\upsilon^2 (\iq \partial \overline{\partial} \upsilon)^2 - 2F_\upsilon F_{\upsilon\upsilon} \partial \overline{\partial} \upsilon \wedge  \partial \upsilon \wedge \overline{\partial} \upsilon - 2 F_{\upsilon}(F_\eta + \eta F_{\eta\eta}) \partial \overline{\partial} \upsilon \wedge \diff \zeta \wedge \diff \overline{\zeta} \\
		& - 2 F_\upsilon F_{\upsilon\eta} \partial \overline{\partial} \upsilon \wedge (\zeta \partial \upsilon \wedge \diff \overline{\zeta} + \overline{\zeta} \diff \zeta \wedge \overline{\partial} \upsilon) - F_{\upsilon\eta}^2 (\zeta \partial \upsilon \wedge \diff \overline{\zeta} + \overline{\zeta} \diff \zeta \wedge \overline{\partial} \upsilon)^2 \\
		& - 2 F_{\upsilon\upsilon}(F_\eta + \eta F_{\eta\eta}) \partial \upsilon \wedge \overline{\partial} \upsilon \wedge \diff \zeta \wedge \diff \overline{\zeta} \\
		= ~ & F_\upsilon^2 (\iq \partial \overline{\partial} \upsilon)^2 - 2 F_\upsilon F_{\upsilon\upsilon} \partial \overline{\partial} \upsilon \wedge  \partial \upsilon \wedge \overline{\partial} \upsilon - 2 F_{\upsilon}(F_\eta + \eta F_{\eta\eta}) \partial \overline{\partial} \upsilon \wedge \diff \zeta \wedge \diff \overline{\zeta} \\
		& - 2 F_\upsilon F_{\upsilon\eta} \partial \overline{\partial} \upsilon \wedge (\zeta \partial \upsilon \wedge \diff \overline{\zeta} + \overline{\zeta} \diff \zeta \wedge \overline{\partial} \upsilon) \\
		& - 2 (\eta(F_{\upsilon\upsilon} F_{\eta\eta} - F_{\upsilon\eta}^2) + F_\eta F_{\upsilon\upsilon}) \partial \upsilon \wedge \overline{\partial} \upsilon \wedge \diff \zeta \wedge \diff \overline{\zeta} .
    \end{align*}
    Since $\phi$ is an admissible potential, we have bounds
    \begin{align*}
        \eta(F_{\upsilon\upsilon} F_{\eta\eta} - F_{\upsilon\eta}^2) = \ohat(\eta \mr^{-8} \rho^{-4}) , ~~F_\eta F_{\upsilon\upsilon} = \ohat(\mr^{-8}\rho^{-2})
    \end{align*}
    and since $|\partial \upsilon \wedge \overline{\partial} \upsilon \wedge \diff \zeta \wedge \diff \overline{\zeta}|_{\omega_s^{\m{srf}}} = \ohat(\upsilon \mr_s^2) = \ohat(\mr^6)$ (cf. Lemma \ref{lem: uniform comparisons}), it follows that the last term of the expression of $(\iq \partial \overline{\partial}\phi_s)^2$ is bounded by $\ohat(\eta \mr^{-2} \rho^{-4}) + \ohat(\mr^{-2}\rho^{-2})$. But since $\mr^{-2} = \ohat(\rho^{-1})$ and $\eta \mr^{-2} = \ohat(\eta^{1/2}) = \ohat(\rho)$ (by the aforementioned lemma), this means that we have a bound in $\ohat(\rho^{-3})$. Similarly, 
    \begin{equation*}
        F_\upsilon(F_\eta+ \eta F_{\eta\eta}) = \ohat(\mr^{-4}\rho^{-2}), ~~ |\partial \overline{\partial} \upsilon \wedge \diff \zeta \wedge \diff \overline{\zeta}|_{\omega_s^{\m{srf}}} = \ohat(\mr^2)
    \end{equation*}
    and hence the third term is bounded by $\ohat(\mr^{-2}\rho^{-2}) = \ohat(\rho^{-3})$. Finally,
    \begin{align*}
        F_\upsilon F_{\upsilon\eta} = \ohat(\mr^{-8} \rho^{-2}), ~~
        |\omega_s^{\m{srf}} \wedge \partial \overline{\partial} \upsilon \wedge (\zeta \partial \upsilon \wedge \diff \overline{\zeta} + \overline{\zeta} \diff \zeta \wedge \overline{\partial} \upsilon)|_{\omega_s^{\m{srf}}} = \ohat(\eta \mr^2)
    \end{align*}
    and thus the fourth term is bounded by $\ohat(\eta \mr^{-6} \rho^{-2}) = \ohat(\mr^{-2} \rho^{-2}) =  \ohat(\rho^{-3})$. In the end,
    \begin{equation*}
        \omega_s^{\m{srf}} \wedge (\iq \partial \overline{\partial} \phi_s)^2 =  F_\upsilon^2 \omega_s^{\m{srf}}\wedge (\iq \partial \overline{\partial} \upsilon)^2 - 2 F_\upsilon F_{\upsilon\upsilon} \omega_s^{\m{srf}} \wedge \partial \overline{\partial} \upsilon \wedge  \partial \upsilon \wedge \overline{\partial} \upsilon + \ohat(\rho^{-3}) .
    \end{equation*}
    Now using \eqref{eq: Wedge prod of omega and ddbar upsilon squared}, we see that
    \begin{align*}
        \tfrac{3}{4} F_\upsilon^2 \omega^{\m{srf}}_s \wedge (\iq \partial \overline{\partial} \upsilon)^2/\vol_+ = \left( 2\upsilon+\eta + \frac{\eta-\upsilon}{\mr_s^2} \right) F_\upsilon^2 & = (2\upsilon+\eta) F_\upsilon^2 + \ohat(\mr^{-6}) \\
            & = (2\upsilon+\eta)F_\upsilon^2 + \ohat(\rho^{-3}).
    \end{align*}
    On the other hand,
    \begin{equation*}
        \omega^{\m{srf}}_s \wedge \partial \upsilon \wedge \overline{\partial} \upsilon = \frac{\iq}{2} ( \partial \upsilon \wedge \overline{\partial} \upsilon \wedge \sigma_s \wedge \overline{\sigma}_s + (1+\mr_s^{-2} ) \partial \upsilon \wedge \overline{\partial} \upsilon \wedge \diff \zeta \wedge \diff \overline{\zeta})
    \end{equation*}
    and hence
    \begin{align*}
        \omega^{\m{srf}}_s \wedge \partial \overline{\partial} \upsilon \wedge \partial \upsilon \wedge \overline{\partial} \upsilon & = \frac{\iq}{4} \left( \left(\mr_s^2 - \frac{4s^2}{\mr_s^2} \right) \left(1 + \mr_s^{-2}\right) - \frac{4\upsilon}{\mr_s^4} \right) \partial \upsilon \wedge \overline{\partial}\upsilon \wedge \sigma_s \wedge \overline{\sigma}_s \wedge \diff \zeta \wedge \diff \overline{\zeta} \\
            & = \left(- \frac{\upsilon \mr_s^4}{3} + \ohat(\mr_s^6)\right) \vol_+
    \end{align*}
    and as $F_\upsilon F_{\upsilon\upsilon} = O(\mr^{-12})$ we obtain
    \begin{equation*}
        - \frac{3}{2} F_\upsilon F_{\upsilon\upsilon} \omega_s \wedge \partial \overline{\partial} \upsilon \wedge \partial \upsilon \wedge \overline{\partial} \upsilon = \frac{\upsilon \mr_s^4}{2} F_\upsilon F_{\upsilon\upsilon} \vol_+ + \ohat(\mr^{-6}) = \frac{\upsilon \mr_s^4}{2} F_\upsilon F_{\upsilon\upsilon} \vol_+ + \ohat(\rho^{-3})
    \end{equation*}
    which finishes the proof.
\end{proof}

    \subsection{Improvement of the ansatz in region I}      \label{subsec: Improvement I}

Thanks to Proposition \ref{prop: MA equation}, we have reduced the problem of finding approximate solutions of the Monge--Amp\`ere equation $(\omega_s^{\m{srf}}+ \frac{\iq}{2} \partial \overline{\partial}\phi_s)^3 = \vol_+$ to the study of the PDE
\begin{equation}    \label{eq: Simplified MA equation}
    \mA_s(F) + \mB_s(F) = - \frac{1}{\mr_s^2}
\end{equation}
where $\{\phi_s = F(s,\upsilon,\eta)\}_{s \geq 0}$ is an admissible potential, and we aim to solve the above equation up to an error term which has faster than quadratic decay, uniformly in $s$ as long as $s$ remains bounded above, say in $\ohat(\rho^{-2-\epsilon})$ for some $\epsilon > 0$. In order to find such a family of asymptotic solutions, we shall separate the domain $\{\upsilon + \eta \geq 1\}$ into two different regions and analyse them separately as a first step, before explaining how to interpolate between them. This is inspired by the construction of approximate solutions in \cite{szekelyhidi2019degenerations}, and the reason why this separation is needed in our case should become clear soon. 

We define the following regions:
\begin{itemize}
    \item Region I: $\{ \upsilon \geq \frac{1}{2} \eta^2, \upsilon + \eta \geq 1 \}$.
    \item Region II: $ \{ \upsilon \leq 2 \eta^2, \upsilon + \eta \geq 1 \}$.
\end{itemize}
Since the semi Ricci-flat ansatz $\omega^{\m{srf}}_s$ already decays quadratically along the fibres of $\zeta$, we can expect that in region I, the non-linear terms $\mB_s(F)$ will be negligible, whilst in region II both linear and non-linear terms will play a role. We first begin with region I, where the analysis is substantially easier, before taking on region II in the next section.

In region I, the following uniform comparisons hold:

\begin{lem}     \label{lem: Uniform comparisons in region I}
    In region I, the functions $\upsilon^{1/4}$, $\mr$ and $\rho$ are uniformly equivalent, which we denote by $\upsilon^{1/4} \asymp \mr \asymp \rho$. In particular, $\mr_s^{-1} = \ohat(\rho^{-1})$.
\end{lem}

\begin{proof}
    $\upsilon^{1/4} \leq \rho = (\upsilon+\eta^2)^{1/4} \lesssim \upsilon^{/4}$. Thus $\rho$ is uniformly equivalent to $\upsilon^{1/4}$ in region I. Furthermore, $\upsilon \leq \mr^4 = 4(\upsilon+\eta) \lesssim \upsilon+\upsilon^{1/2} \lesssim \upsilon$ and thus $\mr$ is also uniformly equivalent to $\upsilon^{1/4}$. In particular $\mr_s^{-1} = \ohat(\mr^{-1}) = \ohat(\rho^{-1})$.
\end{proof}

Let us now consider the potential 
\begin{equation}
    F^{(0)}(s,\upsilon,\eta) \coloneqq - \frac{1}{2} \log(\upsilon+\eta+s^2).
\end{equation}
It is easy to compute
\begin{align*}
    F^{(0)}_\upsilon & = F^{(0)}_\eta = - \frac{1}{2(\upsilon+\eta+s^2)} = - \frac{2}{\mr_s^4}, \\
    F^{(0)}_{\upsilon\upsilon} & = F^{(0)}_{\eta\eta} = F^{(0)}_{\upsilon\eta} = \frac{1}{2(\upsilon+\eta+s^2)^2} = \frac{8}{\mr_s^8} \cdot
\end{align*}
It is immediate to see from these expressions that $\mB_s(F) = \ohat(\mr_s^{-4})$, and moreover
\begin{equation*}
    \mA_s(F) = - \frac{2}{\mr_s^2} + \frac{4\upsilon}{\mr_s^6} + \ohat(\mr_s^{-4}) = - \frac{1}{\mr_s^2} - \frac{4 \eta}{\mr_s^6} +\ohat(\mr_s^{-4}) .
\end{equation*}
But in region I, $\eta = \ohat(\upsilon^{1/2}) = \ohat(\mr_s^2)$ and $\mr \asymp \rho$ by the previous lemma, so that
\begin{equation*}
    \mA_s(F) + \mB_s(F) = -\frac{1}{\mr_s^2} + \ohat(\mr_s^{-4}) = -\frac{1}{\mr_s^2} + \ohat(\rho^{-4}) .
\end{equation*}
This allows us to prove:

\begin{lem}     \label{lem: Potential in region I}
    The family of potentials $F^{(0)}(s,\upsilon,\eta)$ is admissible in region I, and moreover in this region it satisfies
    \begin{equation*}
        \mA_s(F^{(0)}) + \mB_s(F^{(0)}) + \frac{1}{\mr_s^2} = \ohat(\rho^{-4}) .
    \end{equation*}
\end{lem}

\begin{proof}
    We already proved that $F^{(0)}$ satisfies the desired equation, and the fact that it is admissible in region in I is immediate since $\mr_s^{-1} = \ohat(\rho^{-1})$ therein.
\end{proof}

	\subsection{Improvement of the ansatz in region II}    \label{subsec: Improvement II}

In this section, we consider region II: $\{\upsilon \leq 2\eta^2, \upsilon+\eta \geq 1\}$. In this region, the relevant uniform comparisons are as follows:

\begin{lem}     \label{lem: Uniform comparisons in region II}
    In region II, $\eta^{1/2} \asymp \rho \asymp \mr^2$. In particular, 
    \begin{equation*}
        (\eta+s^2)^{-\frac{1}{2}} = \ohat(\rho^{-1}), ~~ \text{and} ~~ \mr_s^{-2} = \ohat(\rho^{-1}).
    \end{equation*}
\end{lem}

\begin{proof}
    In region II, $\eta^{1/2} \leq \rho = (\upsilon+\eta^2)^{1/4} \lesssim \eta^{1/2}$ so that $\rho$ is uniformly equivalent to $\eta^{1/2}$ in region II. On the other hand, $\mr_s^{-2} = \ohat(\rho^{-1})$ on the whole region $\{\upsilon + \eta \geq 1\}$ as we already noted in Lemma \ref{lem: uniform comparisons}.
\end{proof}

Using the previous lemma and the definition of admissible families of potentials, in is not difficult to see that in region II, the leading-order term in \eqref{eq: Simplified MA equation} reads
\begin{equation*}
    \upsilon \mr_s^2 F_{\upsilon\upsilon} + \left( \mr_s^2 + \frac{2\upsilon}{\mr_s^2} \right) F_\upsilon + \frac{1}{\mr_s^2} = 0,
\end{equation*}
which after multiplying by $\frac{1}{2}\mr_s^2 = \sqrt{\upsilon+\eta+s^2}$ can be written as
\begin{equation*}
    2\upsilon(\upsilon+\eta+s^2) F_{\upsilon\upsilon} + (3\upsilon + 2 \eta + 2 s^2) F_\upsilon + \frac{1}{2} = 0 .
\end{equation*}
If we introduce the function\footnote{Remark that $\eta$ does not vanish in region II, and therefore the function $\mt_s$ is smooth and well-defined in this region even when $s = 0$.}
\begin{equation}
    \mt_s \coloneqq \frac{\upsilon}{\eta + s^2}
\end{equation}
this is equivalent to
\begin{equation*}
    2 \mt_s(\mt_s+1)(\eta+s^2)^2 F_{\upsilon\upsilon} + (3\mt_s+2) (\eta + s^2) F_\upsilon + \frac{1}{2} = 0
\end{equation*}
and thus $F$ is a solution of this equation if and only if $F = \tilde{F} \circ \mt_s$ where the function $H(t) = \tilde{F}^\prime(t)$ satisfies the first-order ODE
\begin{equation}        \label{eq: ODE for H}
    2t(t+1) H^\prime(t) + (3t + 2) H(t) + \frac{1}{2} = 0 . 
\end{equation}

\begin{lem}
    The ODE \eqref{eq: ODE for H} has a unique smooth (even real-analytic) solution $H_1$ defined on $[0,+\infty)$,
    \begin{equation*}
        H_1(t) = - \left(1 - \frac{1}{\sqrt{1+t}} \right) \frac{1}{2t} = - \frac{t+1 - \sqrt{t+1}}{2 t (t+1)}\cdot
    \end{equation*}
\end{lem}

\begin{proof}
    On the interval $(0,+\infty)$, we can seek a solution of the form $H(t) = \frac{a(t)}{t\sqrt{1+t}}$. Since
    \begin{equation*}
        \frac{\diff}{\diff t} \left(\frac{a(t)}{t\sqrt{1+t}}\right) = \frac{a^\prime(t)}{t\sqrt{1+t}} - \frac{3t+2}{2t(t+1)} \cdot \frac{a(t)}{t\sqrt{t+1}}
    \end{equation*}
    the ODE becomes
    \begin{equation*}
        2a^\prime(t) \sqrt{1+t}  + \frac{1}{2} = 0 .
    \end{equation*}
    The solutions of this equation are of the form $a(t) = c - \tfrac{1}{2} \sqrt{1+t}$ for some constant $c$, and thus the solutions of the ODE \eqref{eq: ODE for H} on the interval $(0,+\infty)$ are of the form
    \begin{equation*}
        H(t) = - \left( \frac{1}{2} - \frac{c}{\sqrt{1+t}} \right) \frac{1}{t} \cdot
    \end{equation*}
    It is clear that the unique solution that extends smoothly, and even real-analytically, to the whole interval $[0,+\infty)$ corresponds to $c = \tfrac{1}{2}$.
\end{proof}

We can easily check that the function
\begin{equation}
    \widetilde{H}_1(t) \coloneqq - \frac{1}{2} \log \left(t \frac{\sqrt{1+t}+1}{\sqrt{1+t}-1} \right)
\end{equation}
is a smooth primitive of $H_1$ on the interval $[0,+\infty)$. Let us introduce the potential
\begin{equation}
    F^{(1)}(s,\upsilon,\eta) \coloneqq \widetilde{H}_1(\mt_s) - \frac{1}{2}\log(\eta + s^2) .
\end{equation}
In more explicit terms, one can check that
\begin{equation*}
    F^{(1)}(s,\upsilon,\eta) = - \frac{1}{2} \log \left( \upsilon \frac{\sqrt{\upsilon+\eta+s^2} + \sqrt{\eta+s^2}}{\sqrt{\upsilon+\eta+s^2}-\sqrt{\eta+s^2}} \right) \cdot
\end{equation*}
By construction,
\begin{equation*}
    \upsilon \mr_s^2 F^{(1)}_{\upsilon\upsilon} + \left( \mr_s^2 + \frac{2\upsilon}{\mr_s^2} \right) F^{(1)}_\upsilon = - \frac{1}{\mr_s^2} \cdot
\end{equation*}
Now we need to prove that $F^{(1)}(\upsilon,\eta)$ is an admissible potential. To simplify the computations, remark that we can write 
\begin{equation*}
    H_1(t) = \frac{a(t)}{t+1}, ~~ H^\prime_1(t) = \frac{b(t)}{(t+1)^2}
\end{equation*}
where $a$ and $b$ are \emph{bounded} real-analytic functions on $[0,+\infty)$, explicitly given by
\begin{equation*}
    a(t) = - \frac{t+1 - \sqrt{t+1}}{2t} ~~ \text{and} ~~ b(t) = \frac{2(t+1)^2-(3t+2)\sqrt{t+1}}{2t^2} \cdot
\end{equation*}
Also, the derivatives of the function $\mt_s$ are obviously given by
\begin{equation}
    \partial_\upsilon \mt_s = \frac{1}{\eta+s^2}, ~~ \partial_\eta \mt_s = - \frac{\upsilon}{(\eta+s^2)^2} = - \frac{\mt_s}{\eta+s^2} \cdot
\end{equation}
Using these observations, we compute:
\begin{align*}
    F^{(1)}_\upsilon & = \frac{1}{\eta+s^2} H_1(\mt_s) = \frac{a(\mt_s)}{(\eta+s^2)(1+\mt_s)} = \frac{4a(\mt_s)}{\mr_s^4} , \\
    F^{(1)}_{\upsilon\upsilon} & = \frac{1}{(\eta+s^2)^2} H_1^\prime(\mt_s) = \frac{b(\mt_s)}{(\eta+s^2)^2(1+\mt_s)^2} = \frac{16 b(\mt_s)}{\mr_s^8} \cdot
\end{align*}
Thus $|F^{(1)}_\upsilon| = \ohat(\mr_s^{-4})$ and $|F^{(1)}_{\upsilon\upsilon}| = \ohat(\mr_s^{-8})$. For the derivatives involving $\eta$, we have
\begin{align*}
    F^{(1)}_\eta & = - \frac{\mt_s H_1(\mt_s)}{(\eta+s^2)} - \frac{1}{2(\eta+s^2)} = - \frac{1}{2\sqrt{\mt_s +1} (\eta+s^2)} = - \frac{1}{\mr_s^2\sqrt{\eta+s^2}} , \\
    F^{(1)}_{\upsilon\eta} & = \frac{2}{\mr_s^6\sqrt{\eta+s^2}} \\
    F^{(1)}_{\eta\eta} & = \frac{1}{2\mr_s^2 \sqrt{\eta+s^2}^3} + \frac{2}{\mr_s^6 \sqrt{\eta+s^2}} \cdot
\end{align*}
Using Lemma \ref{lem: Uniform comparisons in region II}, these expressions immediately imply the bounds, which hold in region II, $|F^{(1)}_\eta| = \ohat(\rho^{-2})$, $|F^{(1)}_{\upsilon\eta}| = \ohat(\mr^{-4}\rho^{-2})$ and $F^{(1)}_{\eta\eta} = \ohat(\rho^{-4})$, which means that $F^{(1)}$ is admissible in region II. 

\begin{rem}
    At the intersection of regions I and II, e.g. $\{\frac{1}{2}\eta^2 \leq \upsilon \leq 2\eta^2\}$, we have $F^{(1)} \sim - \frac{1}{2}\log(\upsilon) \sim - \frac{1}{2}\log(\upsilon+\eta+s^2) = F^{(0)}$. This already hints that later on, we will be able to interpolate between the potential $F^{(0)}$ in region I and the potential that we will construct in region II, which will be of the form $F^{(1)} + \text{lower order terms}$.
\end{rem}

Noting that
\begin{align*}
    F^{(1)}_\eta + \eta F^{(1)}_{\eta\eta} & = - \frac{1}{2 \mr_s^2\sqrt{\eta+s^2}} + \frac{2\sqrt{\eta+s^2}}{\mr_s^6} - \frac{s^2}{2\mr_s^2 \sqrt{\eta+s^2}^3} - \frac{2s^2}{\mr_s^6\sqrt{\eta+s^2}} \\
        & = -\frac{1}{2\mr_s^2\sqrt{\eta+s^2}} + \frac{1}{\mr_s^4 \sqrt{1+\mt_s}} + \ohat((\eta+s^2)^{-3/2}\mr_s^{-2})
\end{align*}
and the identities
\begin{equation*}
    \frac{\upsilon}{\mr_s^4} = \frac{\mt_s}{4(1+\mt_s)}, ~~  \frac{\eta}{\mr_s^4} = \frac{1}{4(1+\mt_s)} - \frac{s^2}{\mr_s^4}
\end{equation*}
a direct computation using the above formulas shows that there exists a \emph{bounded}, real-analytic function $\alpha \colon [0,+\infty) \to \R$ such that
\begin{equation}        \label{eq: Approx for F1}
    \mA_s(F^{(1)}) + \mB_s(F^{(1)}) + \frac{1}{\mr_s^2} = - \frac{1}{2\mr_s^2\sqrt{\eta+s^2}} +\frac{\alpha(\mt_s)}{\mr_s^4} + \ohat((\eta+s^2)^{-3/2}\mr_s^{-2}) .
\end{equation}
By Lemma \ref{lem: Uniform comparisons in region II}, $(\eta+s^2)^{-1/2} = \ohat(\rho^{-1})$ and $\mr_s^{-2} = \ohat(\rho^{-1})$ so that the error term $\ohat((\eta+s^2)^{-3/2}\mr_s^{-2}) = \ohat(\rho^{-4})$ already decays fast enough, but the other terms are only bounded by $\ohat(\rho^{-2})$, so they need to be cancelled. We first deal with the term $- \frac{1}{2\mr_s^2\sqrt{\eta+s^2}}$.

\begin{lem}
    Let $F^{(2)}(s,\upsilon,\eta) \coloneqq -\frac{1}{2\sqrt{\eta+s^2}}F^{(1)}(s,\upsilon,\eta)$. Then $F^{(2)}$ is an admissible family of potentials in region II, and moreover in this region we have, for any $\epsilon \in (0,1)$,
    \begin{equation*}
        \mA_s(F^{(1)}+F^{(2)}) + \mB_s(F^{(1)}+F^{(2)}) + \frac{1}{\mr_s^2} = \frac{\alpha(\mt_s)}{\mr_s^4} + \ohat(\rho^{-3+\epsilon})
    \end{equation*}
    where $\alpha \colon [0,+\infty) \to \R$ is the same bounded, real-analytic function as in \eqref{eq: Approx for F1}.
\end{lem}

\begin{proof}
    Let us first prove admissibility. The $\upsilon$-derivatives cause no problem, since
    \begin{align*}
        F^{(2)}_\upsilon & = -\frac{1}{2\sqrt{\eta+s^2}} F^{(1)}_\upsilon = \ohat((\eta+s^2)^{-1/2} \mr_s^{-4}) = \ohat(\mr^{-4} \rho^{-1}) = \ohat(\rho^{-3}), \\ 
        F^{(2)}_{\upsilon\upsilon} & = -\frac{1}{2\sqrt{\eta+s^2}}F^{(1)}_{\upsilon\upsilon} = \ohat((\eta+s^{2})^{-1/2}\mr_s^{-8}) = \ohat(\mr^{-8}\rho^{-1}) = \ohat(\rho^{-5}) .
    \end{align*}
    For the derivatives involving $\eta$, we calculate and use Lemma \ref{lem: Uniform comparisons in region II}:
    \begin{align*}
        F^{(2)}_\eta & = \frac{F^{(1)}}{4\sqrt{\eta+s^2}^3} - \frac{F^{(1)}_\eta}{2\sqrt{\eta+s^2}} = \ohat((1+|F^{(1)}|)\rho^{-3}), \\
        F^{(2)}_{\upsilon\eta} & =  \frac{F^{(1)}_\upsilon}{4\sqrt{\eta+s^2}^3} - \frac{F^{(1)}_{\upsilon\eta}}{2\sqrt{\eta+s^2}} = \ohat(\rho^{-3} \mr^{-4}) , \\
        F^{(2)}_{\eta\eta} & = -\frac{3F^{(1)}}{8 \sqrt{\eta+s^2}^5} + \frac{F^{(1)}_\eta}{2\sqrt{\eta+s^2}^3} - \frac{F^{(1)}_{\eta\eta}}{2\sqrt{\eta+s^2}} = \ohat((1+|F^{(1)}|)\rho^{-5}) .
    \end{align*}
    In addition, we see from the expression of $F^{(1)}$ that in there exist constants $C$ and $C^\prime$ which are \emph{independent of $s$} and such that
    \begin{equation*}
        |F^{(1)}| \leq C(1+\log(1+\mt_s) + |\log(s^2 + \eta)|) \leq C^\prime(1 + |\log(\mr_s)| + |\log(s^2 + \eta)|)
    \end{equation*}
    whence we deduce that $|F^{(1)}| = \ohat(\log(1+\rho))$. Thus in region II we deduce the estimates
    \begin{align*}
        F^{(2)}_\eta & = \ohat(\log(1+\rho) \rho^{-3}) = \ohat(\rho^{-3+\epsilon}), \\
        F^{(2)}_{\eta\eta} & = \ohat(\log(1+\rho)\rho^{-5}) = \ohat(\rho^{-5+\epsilon})
    \end{align*}
    for any $\epsilon \in (0,1)$. This in particular proves admissibility in region II.

    Let us now show that $F^{(1)}+F^{(2)}$ satisfies the desired equation. By construction,
    \begin{equation*}
        \upsilon \mr_s^2 F^{(2)}_\upsilon + \left(\mr_s^2 + \frac{2\upsilon}{\mr_s^2}\right) F^{(2)}_{\upsilon\upsilon} = -\frac{1}{2\sqrt{\eta+s^2}} \left( \upsilon \mr_s^2 F^{(1)}_\upsilon + \left(\mr_s^2 + \frac{2\upsilon}{\mr_s^2}\right) F^{(1)}_{\upsilon\upsilon} \right) = \frac{1}{2\mr_s^2\sqrt{\eta+s^2}} \cdot
    \end{equation*}
    Moreover, we see from the above estimates that for any $\epsilon \in (0,1)$,
    \begin{equation*}
        F^{(2)}_\eta + \eta F^{(2)}_{\eta\eta} + F^{(2)}_\upsilon + \upsilon F^{(2)}_{\upsilon\upsilon} = \ohat(\rho^{-3+\epsilon}) .
    \end{equation*}
    From the equation satisfied by $F^{(1)}$, we therefore deduce that for any $\epsilon \in (0,1)$, we have in region II
    \begin{equation*}
        \mA_s(F^{(1)}) + \mA_s(F^{(2)}) + \mB_s(F^{(1)}) = \frac{\alpha(\mt_s)}{\mr_s^4} + \ohat(\rho^{-3+\epsilon}) .
    \end{equation*}
    It remains to prove that all the terms in the expression of $\mB_s(F^{(1)}+F^{(2)})-\mB_s(F^{(1)})$ can be neglected. Treating each term separately and observing that $(2\upsilon+\eta)|F^{(1)}_\upsilon| = \ohat(1)$ and $\upsilon \mr_s^4 |F^{(1)}_{\upsilon\upsilon}| = \ohat(1)$, we see that
    \begin{align*}
        (2\upsilon+\eta)F^{(1)}_\upsilon F^{(2)}_\upsilon & = \ohat(\mr_s^{-4}(\eta+s^2)^{-1/2}) = \ohat(\rho^{-3}), \\
        \upsilon \mr_s^4 F^{(1)}_\upsilon F^{(2)}_{\upsilon\upsilon} & = \ohat(\mr_s^{-4}(\eta+s^2)^{-1/2}) = \ohat(\rho^{-3}), \\
        \upsilon \mr_s^4 F^{(2)}_\upsilon F^{(1)}_{\upsilon\upsilon} & = \ohat(\mr_s^{-4}(\eta+s^2)^{-1/2}) = \ohat(\rho^{-3}), \\
        (2\upsilon+\eta)(F^{(2)}_\upsilon)^2 & = \ohat(\mr_s^{-4}(\eta+s^2)^{-1}) = \ohat(\rho^{-4}), \\
        \upsilon \mr_s^4 F^{(2)}_\upsilon F^{(2)}_{\upsilon\upsilon} & = \ohat(\mr_s^{-4}(\eta+s^2)^{-1}) = \ohat(\rho^{-4})
    \end{align*}
    which finishes the proof of the lemma.
\end{proof}

\begin{rem}
    The above lemma is the \emph{raison d'être} of the divide between regions I and II. Since the leading-order term of the Monge--Amp\`ere equation behaves like an ODE in the variable $\mt_s = \frac{\upsilon}{s^2+\eta}$, the derivatives of the solutions will be weighted by positive (fractional) powers of $(\eta+s^2)^{-1}$. Thus there is no hope of treating the problem as a perturbation of an ODE (let alone in a uniform way when $s \to 0$) unless dividing by $(\eta+s^2)$ improves the decay of solutions, i.e. we must stay away from the asymptotic regions corresponding to infinite distances along the fibres of $\zeta$. In region II however the ODE method works, and since $(s^2+\eta)^{-1/2} \leq \eta^{-1/2} \asymp \rho^{-1}$ we can easily build approximate solutions of the Monge--Amp\`ere equation whose decay at infinity is uniformly controlled as $s \to 0$. 
\end{rem}

In order to cancel the term $\frac{\alpha(\mt_s)}{\mr_s^4}$, we follow the same steps as before. That is, we seek a solution of the equation
\begin{equation*}
    \upsilon \mr_s^2 F_{\upsilon\upsilon} + \left( \mr_s^2 + \frac{2\upsilon}{\mr_s^2} \right) F_\upsilon = - \frac{\alpha(\mt_s)}{\mr_s^4},
\end{equation*}
which, after multiplying by $\frac{\mr_s^2}{2}$ on both side and factoring all the $\eta+s^2$ terms, reads
\begin{equation*}
    2 \mt_s(\mt_s+1) (\eta+s^2)^2 F_{\upsilon\upsilon} + (3\mt_s+2) (\eta+s^2) F_\upsilon = - \frac{1}{4\sqrt{\eta+s^2}} \frac{\alpha(\mt_s)}{\sqrt{1+\mt_s}} \cdot
\end{equation*}
Thus we seek a solution of the form $F = - \frac{1}{4\sqrt{\eta+s^2}} \tilde{F}(\mt_s)$ where $H = \tilde{F}^\prime$ satisfies the first-order ODE
\begin{equation*}
    2t(t+1)H^\prime(t) + (3t+2) H(t) = \frac{\alpha(t)}{\sqrt{1+t}} \cdot
\end{equation*}
On the interval $(0,+\infty)$, we seek a solution $H(t) = \frac{a(t)}{t\sqrt{1+t}}$, where $a(t)$ must satisfy
\begin{equation*}
    2 a^\prime(t) \sqrt{1+t} = \frac{\alpha(t)}{\sqrt{1+t}}
\end{equation*}
Hence the solutions are of the form
\begin{equation*}
    H(t) = \frac{1}{t \sqrt{1+t}} \left( c + \frac{1}{2} \int_0^t \frac{\alpha(\tau)}{1+\tau} \diff \tau \right)
\end{equation*}
and the unique solution which extends smoothly, and in fact real-analytically, to the whole interval $[0,+\infty)$ is
\begin{equation*}
    H_3(t) \coloneqq \frac{1}{2t\sqrt{1+t}} \int_0^t \frac{\alpha(\tau)}{1+\tau} \diff x .
\end{equation*}
Since $\alpha(t)$ is uniformly bounded,
\begin{equation}        \label{eq: Bounds on H3}
    \sqrt{1+t}^3 |H_3(t)| + \sqrt{1+t}^5 |H^\prime_3(t)| \lesssim (1+ \log(1+t)), ~~~~ \forall t \in [0,+\infty) .
\end{equation}
In particular $H_3$ is integrable on $[0,+\infty)$, and we can consider the primitive
\begin{equation*}
    \widetilde{H}_3(t) \coloneqq - \int_t^{+\infty} H_3(\tau) \diff \tau .
\end{equation*}
This primitive satisfies the estimate
\begin{equation}       \label{eq: Bounds on primitive of H3}
    \sqrt{1+t} |\widetilde{H}_3(t)| \lesssim (1+\log(1+t)), ~~~~ \forall t \in [0,+\infty).
\end{equation}

Let us now define the function
\begin{equation*}
    F^{(3)}(s,\upsilon,\eta) \coloneqq -\frac{1}{4\sqrt{\eta+s^2}} \widetilde{H}_3(\mt_s).
\end{equation*}
By construction, it satisfies
\begin{equation*}
    \upsilon \mr_s^2 F^{(3)}_{\upsilon\upsilon} + \left( \mr_s^2 + \frac{2\upsilon}{\mr_s^2} \right) F^{(3)}_\upsilon = - \frac{\alpha(\mt_s)}{\mr_s^4} \cdot
\end{equation*}
We want to prove that this defines an admissible potential in the region II, and derive appropriate bounds on the derivative of $F^{(3)}$. As usual, the $\upsilon$-derivatives are easy to bound; taking into account that by definition $(\eta+s^2)(1+\mt_s) = s^2 + \upsilon + \eta = 4\mr_s^4$ and $\log(1+\mt_s) = \log(\upsilon+\eta+s^2) - \log(\eta+s^2) = \ohat(\log(1+\rho))$, and using \eqref{eq: Bounds on H3} and Lemma \ref{lem: Uniform comparisons in region II}, we obtain for any $\epsilon \in (0,1)$:
\begin{align*}
    F^{(3)}_\upsilon & = - \frac{H_3(\mt_s)}{4\sqrt{\eta+s^2}^3} = \ohat(\mr_s^{-6}\log(1+\rho)) = \ohat(\mr^{-4} \rho^{-1 +\epsilon}), \\
    F^{(3)}_{\upsilon\upsilon} & = - \frac{H^\prime_3(\mt_s)}{4\sqrt{\eta+s^2}^5} = \ohat(\mr_s^{-10}\log(1+\rho)) = \ohat(\mr^{-8} \rho^{-1 + \epsilon}) .
\end{align*}
To estimate the derivatives involving $\eta$, we also use \eqref{eq: Bounds on primitive of H3}, which yields
\begin{align*}
    F^{(3)}_\eta & =  \frac{\widetilde{H}_3(\mt)}{8\sqrt{\eta+s^2}^3} - \frac{\mt_sH_3(\mt_s)}{4\sqrt{\eta+s^2}^3} = \ohat(\mr_s^{-2} (\eta+s^2)^{-1} \log(1+\rho)) = \ohat(\rho^{-3+\epsilon}) \\
    F^{(3)}_{\upsilon\eta} & = \frac{H_3(\mt_s)}{8\sqrt{\eta+s^2}^5} - \frac{H_3(\mt_s)}{4\sqrt{\eta+s^2}^5} - \frac{\mt_s H^\prime_3(\mt_s)}{4\sqrt{\eta+s^2}^5} \\
        & = \ohat(\mr_s^{-6}(\eta+s^2)^{-1} \log(1+\rho)) = \ohat(\mr^{-4} \rho^{-3+\epsilon}) \\
    F^{(3)}_{\eta\eta} & = \frac{3\widetilde{H}_3(\mt_s)}{16\sqrt{\eta+s^2}^5} + \frac{\mt_sH_3(\mt_s)}{2\sqrt{\eta+s^2}^5} + \frac{\mt_s^2 H_3^\prime(\mt_s)}{4\sqrt{\eta+s^2}^5} \\
    & = \ohat((\eta+s^2)^{-2} \mr_s^{-2}\log(1+\rho)) = \ohat(\rho^{-5+\epsilon}) .
\end{align*}
for any $\epsilon \in (0,1)$. In particular, we obtain similar bounds as for the derivatives of $F^{(2)}$, so that we can deduce in the same way that $\mB_s(F^{(1)} + F^{(2)} + F^{(3)}) - \mB_s(F^{(1)} + F^{(2)}) = \ohat(\rho^{-3+\epsilon})$. This proves that $F^{(3)}$ is an admissible potential in region II. Moreover, essentially the same proof as for the previous lemma yields the main result of this part: 

\begin{lem}     \label{lem: Potential in region II}
    The potential $F^{(\mathrm{II})} = F^{(1)}+F^{(2)}+F^{(3)}$ is admissible in region II, and moreover in this region is satisfies, for any $\epsilon \in (0,1)$,
    \begin{equation*}
        \mA_s(F^{(\mathrm{II})}) + \mB_s(F^{(\mathrm{II})}) + \frac{1}{\mr_s^2} = \ohat(\rho^{-3+\epsilon}) .
    \end{equation*}
\end{lem}

    \subsection{Interpolation between regions I \& II}      \label{subsec: Interpolation III}

The goal of this section is to explain how to interpolate between the approximate solutions constructed in regions I and II. First remark that
\begin{align*}
    F^{(0)}(s,\upsilon,\eta) - F^{(1)}(s,\upsilon,\eta) & = - \frac{1}{2} \log \left(\upsilon \frac{\sqrt{\upsilon+\eta+s^2} + \sqrt{\eta+s^2}}{\sqrt{\upsilon+\eta+s^2}-\sqrt{\eta+s^2}} \right) + \frac{1}{2} \log(\upsilon+\eta+s^2) \\
        & = G(\mt_s)
\end{align*}
where
\begin{equation}
    G(t) \coloneqq - \frac{1}{2} \log \left( \frac{t}{1+t} \frac{\sqrt{1+t}+1}{\sqrt{1+t}-1} \right) . 
\end{equation}
We record the following properties of the function $G$:

\begin{lem}
    There is a constant $C > 0$ such that for all $t \in [0,+\infty)$, we have
    \begin{equation*}
        \sqrt{1+t}|G(t)| + \sqrt{1+t}^3 |G^\prime(t)| + \sqrt{1+t}^5 |G^{\prime\prime}(t)| \leq C.
    \end{equation*}
\end{lem}

\begin{proof}
    For the derivatives, we have
    \begin{equation*}
        G^\prime(t) = H_1(t) + \frac{1}{2(t+1)} = \frac{\sqrt{1+t}-1}{2t(1+t)}
    \end{equation*}
    so that $|G^\prime(t)| = O(t^{-3/2})$ at infinity. Similarly, it is obvious that $|G^{\prime\prime}(t)| = O(t^{-5/2})$ at infinity. For $G(t)$ itself, the statement follows from the fact that
    \begin{equation*}
        \frac{t}{1+t} \frac{\sqrt{1+t}+1}{\sqrt{1+t}-1} = 1 + \frac{2}{\sqrt{1+t}} + \mo(t^{-1})
    \end{equation*}
    as $t$ goes to $+\infty$.
\end{proof}

Consider the interpolation region, that is,
\begin{itemize}
    \item Region III: $\{ \frac{1}{2} \eta^2 \leq \upsilon \leq 2 \eta^2, \upsilon + \eta \geq 1 \}$.
\end{itemize}
Since region III is the intersection of regions I and II, we can combine Lemma \ref{lem: Uniform comparisons in region I} and Lemma \ref{lem: Uniform comparisons in region II} to obtain the comparisons
\begin{equation*}
    \rho \asymp \upsilon^{1/4} \asymp \mr \asymp \eta^{1/2} .
\end{equation*}
Using this observation and the properties of the map $G$, we obtain

\begin{lem}     \label{lem: Estimates on the difference of potentials}
    In region III, the following estimates hold for any $\epsilon \in (0,1)$:
    \begin{align*}
        & |F^{(\mathrm{II})}_\upsilon - F^{(0)}_\upsilon| = \ohat(\mr^{-4} \rho^{-1+\epsilon}), ~~ |F^{(\mathrm{II})}_{\upsilon\upsilon}-F^{(0)}_{\upsilon\upsilon}| = \ohat(\mr^{-8} \rho^{-1+\epsilon}), \\
        |F^{(\mathrm{II})}_\eta-F^{(0)}_\eta| & = \ohat(\rho^{-3+\epsilon}), ~~ |F^{(\mathrm{II})}_{\upsilon\eta}-F^{(0)}_{\upsilon\eta}| = \ohat(\mr^{-4}\rho^{-3+\epsilon}), ~~|F^{(\mathrm{II})}_{\eta\eta}-F^{(0)}_{\eta\eta}| = \ohat(\rho^{-5+\epsilon}) .
    \end{align*}
\end{lem}

\begin{proof}
    Since $F^{(\mathrm{II})}-F^{(0)} = F^{(1)}-F^{(0)} + F^{(2)} + F^{(3)}$ and we have already proved such estimates in region II for the functions $F^{(2)}$ and $F^{(3)}$, the only thing left to do is to check the estimates on the derivatives of $F^{(1)}-F^{(0)} = G(\mt_s)$. Noting that the above uniform comparisons imply
    \begin{equation*}
        \frac{1}{\sqrt{1+\mt_s}} = \frac{\sqrt{\eta+s^2}}{2\mr_s^2} = \ohat(\mr_s^{-1}) = \ohat(\rho^{-1})
    \end{equation*}
    in region III, we calculate:
    \begin{align*}
        |F^{(1)}_\upsilon - F^{(0)}_\upsilon| & = \frac{|G^\prime(\mt_s)|}{\eta+s^2} = \ohat(\mr_s^{-4}(1+\mt_s)^{-1/2}) = \ohat(\mr^{-4} \rho^{-1}) , \\
        |F^{(1)}_{\upsilon\upsilon} - F^{(0)}_{\upsilon\upsilon}| & = \frac{|G^{\prime\prime}(\mt_s)|}{(\eta+s^2)^2} = \ohat(\mr_s^{-8}(1+\mt_s)^{-1/2}) = \ohat(\mr^{-8}\rho^{-1}).
    \end{align*}
    Similarly  for the derivatives involving $\eta$, since $\mr^2 \asymp \eta  \asymp \rho^2$ in region III, we have
    \begin{align*}
        |F^{(1)}_\eta - F^{(0)}_\eta| & = \frac{\mt_s|G^\prime(\mt_s)|}{\eta+s^2} = \ohat(\mr_s^{-2}(\eta+s^2)^{-1/2}) = \ohat(\rho^{-3}) \\
        |F^{(1)}_{\upsilon\eta} - F^{(0)}_{\upsilon\eta}| & \leq \frac{|G^\prime(\mt_s)|}{(\eta+s^2)^2} + \frac{\mt_s|G^{\prime\prime}(\mt_s)|}{(\eta+s^2)} = \ohat(\mr_s^{-6}(\eta+s^2)^{-1/2}) = \ohat(\mr^{-4}\rho^{-3}) \\
        |F^{(1)}_{\eta\eta} - F^{(0)}_{\eta\eta}| & \leq \frac{\mt_s|G^\prime(\mt_s)|}{(\eta+s^2)^2} + \frac{\mt_s^2 |G^{\prime\prime}(\mt_s)|}{(\eta+s^2)^2} = \ohat(\mr_s^{-2}(\eta+s^2)^{-3/2}) = \ohat(\rho^{-5})
    \end{align*}
    as claimed.
\end{proof}

Let us now pick a standard cutoff function $\chi \colon [0,+\infty) \to \R$ such that $\chi(t) = 1$ when $t \leq \frac{1}{2}$ and $\chi(t) = 0$ if $t \geq 2$. We introduce the family of symmetric potentials $\{\varphi_s = F^\varphi(s,\upsilon,\eta)\}_{s \geq 0}$ where:
\begin{equation}
    F^\varphi(s,\upsilon,\eta) = (1-\chi(\tfrac{\upsilon}{\eta^2})) F^{(0)}(\upsilon,\eta) + \chi(\tfrac{\upsilon}{\eta^2}) F^{(\mathrm{II})}(\upsilon,\eta) .
\end{equation}
In particular, it coincides with $F^{(0)}$ in the region $\{ \upsilon \geq 2\eta^2\}$ and with $F^{(\mathrm{II})}$ in the region $\{ \upsilon \leq \frac{1}{2}\eta^2 \}$. 

\begin{prop}        \label{prop: Improved solutions at infinity}
    The family of symmetric potentials $\{\varphi_s\}_{s \geq 0}$ is admissible on $\mZ_+$, and
    \begin{equation*}
        (\omega^{\m{srf}}_s + \tfrac{\iq}{2} \partial \overline{\partial} \varphi_s)^3 = (1+ \psi_s) \vol_+
    \end{equation*}
    where for any $\epsilon \in (0,1)$ the functions $\{\psi_s\}_{s \geq 0}$ satisfy
    \begin{equation*}
        |\psi_s| = \ohat(\rho^{-3+\epsilon}) .
    \end{equation*}
\end{prop}

\begin{proof}
    The only thing left to do is check the statement in region III. In this region, we can write $F^\varphi = F^{(0)} + \tilde{F}$ where $\tilde{F} = \chi(\frac{\upsilon}{\eta^2})(F^{(\mathrm{II})} - F^{(0)})$. Let us first remark that in region III, we have for any $\epsilon \in (0,1)$
    \begin{equation*}
        |F^{(\mathrm{II})} - F^{(0)}|  = \ohat(\rho^{-1+\epsilon})
    \end{equation*}
    Indeed, $F^{(II)} - F^{(0)} = G(\mt_s) + F^{(2)} + F^{(3)}$ and we have in region III
    \begin{align*}
        |G(\mt_s)| & = \ohat((1+\mt_s)^{-1/2}) = \ohat(\rho^{-1}), \\
        |F^{(2)}| & = \ohat((\eta+s^2)^{-1/2}\log(1+\rho)) = \ohat(\log(1+\rho)\rho^{-1}) , \\ 
        |F^{(3)}| & = \ohat(\mr_s^{-2}\log(1+\rho)) = \ohat(\log(1+\rho)\rho^{-2}) .
    \end{align*}
    On the other hand, using that $\eta^2 \asymp \upsilon \asymp \mr^4$ and $\rho \asymp \eta^{1/2}$ in region III, we obtain the following estimates (independent of $s$) on the derivatives of the cutoff function $\tilde{\chi} \coloneqq \chi(\frac{\upsilon}{\eta^2})$:
    \begin{align*}
        \tilde{\chi}_\upsilon & = \eta^{-2} \chi^\prime(\tfrac{\upsilon}{\eta^2}) = \mo(\mr^{-4}), ~~ \tilde{\chi}_{\upsilon\upsilon} = \eta^{-4} \chi^{\prime\prime}(\tfrac{\upsilon}{\eta^2}) = \mo(\mr^{-8}), ~~ \tilde{\chi}_\eta = - 2 \eta^{-3} \chi^{\prime}(\tfrac{\upsilon}{\eta^2}) = \mo(\rho^{-6}), \\
        & \tilde{\chi}_{\upsilon\eta} = - 2\eta^{-5}\chi^{\prime\prime}(\tfrac{\upsilon}{\eta^2})= \mo(\mr^{-4} \rho^{-6}), ~~ \tilde{\chi}_{\eta\eta} = 6\eta^{-4} \chi^{\prime}(\tfrac{\upsilon}{\eta^2}) + 4 \eta^{-6} \chi^{\prime\prime}(\tfrac{\upsilon}{\eta^2}) = \mo(\rho^{-8}),
    \end{align*}
    which immediately imply that $\tilde{F}$ is an admissible potential. More precisely, one can easily check using the above estimates that for any $\epsilon \in (0,1)$ we have
    \begin{align*}
        |\tilde{F}_\upsilon| & = \ohat(\mr^{-4} \rho^{-1+\epsilon}), ~~ |\tilde{F}_{\upsilon\upsilon}| = \ohat(\mr^{-8} \rho^{-1+\epsilon}), ~~ |\tilde{F}_\eta| = \ohat(\rho^{-3+\epsilon}), \\
        & ~~ |\tilde{F}_{\upsilon\eta}| = \ohat(\mr^{-4}\rho^{-3+\epsilon}), ~~ |\tilde{F}_{\eta\eta}| = \ohat(\rho_s^{-5+\epsilon}).
    \end{align*}
    These estimates readily imply that $\mA_s(\tilde{F}) = \ohat(\rho^{-3+\epsilon})$, and thus in region III,
    \begin{equation*}
        \mA_s(F^\varphi) + \mB_s(F^\varphi) + \frac{1}{\mr_s^2} + \mB_s(F^{(0)}+\tilde{F})-\mB_s(F^{(0)}) = \ohat(\rho^{-3+\epsilon}).
    \end{equation*}
    Moreover, we can use all the above estimates to argue as in the previous sections that $\mB_s(F^{(0)}+\tilde{F})-\mB_s(F^{(0)}) = \ohat(\rho^{-3+\epsilon})$. Thus the rest of the proposition follows from Proposition \ref{prop: MA equation}, Lemma \ref{lem: Potential in region I} and Lemma \ref{lem: Potential in region II}.
\end{proof}

    \subsection{Higher-order estimates}     \label{subsec: Higher-order estimates}

\begin{lem}     \label{lem: higher-order estimates}
    The family of potentials $\{\varphi_s = F^\varphi(s,\upsilon,\eta)\}_{s \geq 0}$ satisfies the estimates
    \begin{equation*}
        |\partial^k_\upsilon \partial^\ell_\eta F^\varphi | = \ohat(\mr^{-4k} \rho^{-2\ell}), ~~~~ \forall k,\ell \in \N_0, ~ k+\ell \geq 1 .
    \end{equation*}
\end{lem}

\begin{proof}
    We restrict ourselves to proving the estimates for $F^{(0)}$ in region I and $F^{(\m{II})}$ in region II, since there is no difficulty in deducing the result by cutoff in the interpolation region as in the previous section. In region I this is trivial since $\partial^k_\upsilon \partial^\ell_\eta F^{(0)}  = C_{k+\ell} \mr_s^{-4(k+\ell)}$ for some constants $C_{k+\ell}$ and $\rho = \ohat(\mr_s^2)$. In region II we claim that for all $k,\ell \geq 0$ such that $k+\ell \geq 1$ we have
    \begin{equation*}
        |\partial^k_\upsilon\partial^\ell_\eta F^{(1)}| = \ohat(\mr^{-4k}\rho^{-2\ell}), ~~ |\partial^k_\upsilon \partial^\ell_\eta (F^{(2)}+F^{(3)})| = \ohat(\mr^{-4k} \rho^{-2\ell-1+\epsilon})
    \end{equation*}
    for any $\epsilon \in (0,1)$. The proof is exactly the same as for the proof of the admissibility of these potentials,  because the functions $(1+t)^{k+1} \tfrac{\diff^k}{\diff t^k} H_1(t)$ and $(1+t)^{k+\frac{3}{2}} \tfrac{\diff^k}{\diff t^k} H_3(t)$ are smooth and remain bounded on the interval $[0,+\infty)$, for all $k \geq 0$.
\end{proof}

\begin{lem}     \label{lem: higher-order estimates epsilon}
    Let us write the functions $\psi_s$ of Proposition \ref{prop: Improved solutions at infinity} as $\psi_s = F^\psi(s,\upsilon,\eta)$. Then for all $k,\ell \geq 0$ and $\epsilon \in (0,1)$ we have
    \begin{equation*}
        |\partial^k_\upsilon \partial^\ell_\eta F^\psi| = \ohat(\mr^{-4k} \rho^{-2\ell-3+\epsilon}) .
    \end{equation*}
\end{lem}

\begin{proof}
    The function $F^\psi$ can be decomposed as a sum of products of the derivatives of $F^\varphi$ of order $1$ and $2$, multiplied by coefficients depending \emph{polynomially}\footnote{In fact we have used fractional powers of $\upsilon$ and $\eta$ in the proofs, but it is not difficult to see that if we explicitly expanded the Monge--Amp\`ere equation only integral powers of $\upsilon$ and $\eta$ would occur.} on $\upsilon$, $\eta$, and integral powers of $\mr_s$ (of any sign). Now, in the proof of the estimate $|F^\psi| = \hat{\mo}(\rho^{-3+\epsilon})$ (that is, in Lemma \ref{lem: bound on ddbar for admissible potentials}, Proposition \ref{prop: MA equation}, and when neglecting the terms of $\mB_s(F^{(1)}+F^{(2)}) - \mB_s(F^{(1)})$ and so on in \S\ref{subsec: Improvement I}--\ref{subsec: Improvement II}), the only bounds we have used on the monomials in $\upsilon, \eta$ and powers of $\mr_s$ for the coefficients are $\upsilon = \ohat(\mr_s^4)$, $\eta = \mo(\rho^2)$, $\rho^2 = \ohat(\mr_s^4)$. But for such monomials we can also derive estimates on the derivatives, since $\partial_\upsilon^k \partial_\eta^\ell \mr_s = \ohat(\mr^{1-4k-4\ell}) = \ohat(\mr^{1-4k} \rho^{-2\ell})$, and of course $\partial_\upsilon \upsilon = \mo(1)$, $\partial_\eta \eta = \mo(1)$, and $\partial^k_\upsilon \partial^\ell_\eta \upsilon = 0$ if $\ell \geq 1$ or $k \geq 2$ and $\partial^k_\upsilon \partial^\ell_\eta \eta = 0$ if $k \geq 1$ or $\ell \geq 2$. Hence any monomial in $\upsilon, \eta$ and powers of $\mr_s$ have uniform higher-order estimates such that each $\upsilon$-derivative gives an additional $\mr^{-4}$ factor in the decay and any $\eta$-derivative an additional factor of $\rho^{-2}$. Combined with the estimates on the derivatives of $F^\varphi$ from the previous lemma, the result follows.
\end{proof}

We shall also need to express the higher-order estimates intrinsically using covariant derivatives. For this we need the following result:

\begin{prop}     \label{prop: covariant derivatives at infinity}
    For any $k \geq 0$,
    \begin{align*}
        |\nabla_{\omega_s^{\m{srf}}}^k & \diff \upsilon| = \ohat(\mr^{3-k}), ~~~ |\nabla_{\omega_s^{\m{srf}}}^k \diff \zeta| = \ohat(\mr^{-k}), ~~~ |\nabla_{\omega_s^{\m{srf}}}^k \diff \eta| = \ohat(\mr^{2-k}), \\
        &|\nabla_{\omega_s^{\m{srf}}}^k \diff \mr_s| = \ohat(\mr^{-k}), ~~~ |\nabla_{\omega_s^{\m{srf}}}^k \diff \rho| = \ohat(\mr^{-k}) .
    \end{align*}
    Moreover, 
    \begin{equation*}
        |\nabla^k_{\omega_0^{\m{srf}}}(\omega_s^{\m{srf}} - \omega_0^{\m{srf}})| = \ohat(\mr^{-2-k})
    \end{equation*}
    for all $k \geq 0$.
\end{prop}

\begin{proof}
    For the covariant derivatives of the functions $\upsilon,\zeta,\eta,\mr,\rho$, we only need to estimate the covariant derivatives of $\upsilon$ and $\zeta$, since one can check that the other bounds easily follow. We begin with the covariant derivatives with respect to $\omega_0^{\m{srf}}$. In fact we can instead work with the metric $\tilde{\omega}_0 \coloneqq \hat{\omega}_0 + \frac{\iq}{2} \diff \zeta \wedge \diff \overline{\zeta}$, since if we can prove these estimates for this metric then the term $\mr_s^{-2}\diff \zeta \wedge \diff \overline{\zeta}$ will give a negligible error at all order of derivatives. 
    
    With this in mind, let $\lambda \geq 1$ and let us consider the region 
    \begin{equation*}
        U_\lambda \coloneqq  \{\upsilon \leq \lambda^4, \eta \leq \lambda ^4\} \setminus \{\upsilon \leq \tfrac{1}{2}\lambda^4, \eta \leq \tfrac{1}{2}\lambda ^4\} \subset \mZ_0 \cong \mZ^{\m{reg}} .
    \end{equation*}
    Since the action of $\Un(1)$ on $\C^4$ commutes with the action of $\R_+$ by scaling, and the moment map $\mu$ is quadratic, we can identify this region with the standard region $U_1 = \{\upsilon \leq 1, \eta \leq 1\} \setminus \{\upsilon \leq \tfrac{1}{2}, \eta \leq \tfrac{1}{2}\} \subset \mZ_0$ by a rescaling of $\lambda^{-1}$, which has the effect of changing $\tilde{\omega}_0$ to $\lambda^{2}(\hat{\omega}_0 + \lambda^{2}\frac{\iq}{2}\diff \zeta \wedge \diff \overline{\zeta})$ (note that the horizontal connection of the fibration $\zeta$ is scale-invariant). On the other hand, using radial parallel transport we can trivialise the fibration to see the standard region $U_1$ as an open subset of $\C \times T^* \mathbb{S}^2$, and under this identification we can change the base coordinate from $\zeta$ to $\zeta^\prime = \lambda \zeta$. Since $\lambda \geq 1$ this has the effect of flattening the horizontal distribution, and transforms the metric to $\lambda^2 (\hat{\omega}^\prime_\lambda + \frac{\iq}{2} \diff \zeta^\prime \wedge \diff\overline{\zeta}^\prime)$ where the horizontal part $\hat{\omega}^\prime_\lambda$ varies very slowly in the horizontal directions when $\lambda \gg 1$. Thus it is easy to see that in the region $\{\upsilon \leq 1, |\zeta^\prime|^2 \leq \lambda^2\} \setminus \{\upsilon \leq \frac{1}{2}, |\zeta^\prime|^2 \leq \frac{1}{2}\lambda^2\}$ could be covered with an atlas of charts such that the coordinates of $\hat{\omega}^\prime_\lambda + \frac{\iq}{2} \diff \zeta^\prime \wedge \diff\overline{\zeta}^\prime$ and all their derivatives are uniformly bounded with respect to $\lambda \geq 1$. Taking into account that the rescaling and base reparametrisation swap $\upsilon$ for $\lambda^4\upsilon$ and $\zeta$ for $\lambda \zeta^\prime$, we deduce that in the region $U_\lambda$,
    \begin{equation*}
        |\nabla_{\tilde{\omega}_0}^k \diff \upsilon| = \mo(\lambda^{3-k}), ~~~ |\nabla_{\tilde{\omega}_0}^k \diff \zeta| = \mo(\lambda^{-k}), ~~~ \text{when} ~\lambda \gg 1 .
    \end{equation*}
    Since $\lambda$ is uniformly comparable to $\mr$ in $U_\lambda$ this yields the desired bounds for $s = 0$.

    From this we can deduce the estimates 
    \begin{equation*}
        |\nabla^k_{\omega_0^{\m{srf}}}(\omega_s^{\m{srf}} - \omega_0^{\m{srf}})| = \ohat(\mr^{-2-k})
    \end{equation*}
    for all $k \geq 0$, which generalises Lemma \ref{lem: asymptotic equivalence of skrf metrics} to estimates on the derivatives. Indeed, from Proposition \ref{prop: Explicit form of srf ansatz} and the identity $R = \frac{1}{2} \mr^2$ we obtain
    \begin{equation*}
        \omega_s^{\m{srf}} - \omega_0^{\m{srf}} = s \tilde{\omega}^+_{\m{FS}} + \iq \partial \overline{\partial}  \phi^\prime_s
    \end{equation*}
    where
    \begin{equation*}
        \phi^\prime_s \coloneqq \sqrt{\tfrac{1}{4} \mr^4 + s^2} - \tfrac{1}{2}\mr^2 - s \log \left(s + \sqrt{\tfrac{1}{4} \mr^4 + s^2} \right) .
    \end{equation*}
    Moreover, the above estimates yield
    \begin{equation*}
        |\nabla_{\omega^{\m{srf}}_0}^k\iq\partial \overline{\partial}  \phi^\prime_s| = \ohat(\mr^{-2-k})
    \end{equation*}
    for all $k \geq 0$. On the other hand, $\tilde{\omega}_{\m{FS}}^+ = p_+^*\omega_{\m{FS}}$ is invariant under scaling, since the projection $p_+ \colon \mZ_+ \to \mathbb{P}^1$ is invariant under scaling. Thus the same scaling argument as above shows that 
    \begin{equation*}
        |\nabla^k_{\omega^{\m{srf}}_0} \tilde{\omega}^+_{\m{FS}}| = \mo(\mr^{-2-k}) 
    \end{equation*}
    for all $k \geq 0$, which yields the desired bounds on the covariant derivatives of $\omega^{\m{srf}}_s - \omega^{\m{srf}}_0$. This in turn implies the bounds on the covariant derivatives of $\upsilon,\zeta,\eta,\mr_s,\rho$ with respect to $\omega^{\m{srf}}_s$ for all $s > 0$.
\end{proof}

\begin{cor}     \label{cor: Intrinsic higher-order estimates}
    For any $k \geq 1$ and $\epsilon \in (0,1)$, 
    \begin{equation*}
        |\nabla_{\omega_s^{\m{srf}}}^k \phi_s| = \ohat(\mr^{-k}), ~~~ \text{and} ~~~ |\nabla_{\omega_s^{\m{srf}}}^k \psi_s| = \ohat(\mr^{-k}\rho^{-3+\epsilon}) .
    \end{equation*}
\end{cor}

\begin{proof}
    For $k = 1$, we have $\diff \phi_s = F^\phi_\upsilon \diff \upsilon + F^\phi_\zeta \diff \eta$, and $F^\phi_\upsilon \diff \upsilon = \ohat(\mr^{-4} \mr\sqrt{\upsilon}) = \ohat(\mr^{-1})$, $F^\phi_\eta \diff \eta = \ohat(\rho^{-2} \sqrt{\eta}) = \ohat(\rho^{-1}) = \ohat(\mr^{-1})$. For the higher derivatives, any covariant derivative of $\diff \upsilon$, $\diff \eta$ gives an additional $\mr^{-1}$ factor in the decay, and any additional $\upsilon$ or $\eta$-derivative of $F^\phi$ gives an additional $\mr^{-4}|\diff \upsilon| = \ohat(\mr^{-1})$ or $\rho^{-2} |\diff \eta| = \ohat(\rho^{-1}) = \ohat(\mr^{-1})$-factor in the decay. The result follows for the covariant derivatives of $\diff \phi_s$, and the case of $\diff \psi_s$ is similar. 
\end{proof}

\begin{rem}
    Along the fibres of $\zeta$, $\mr \asymp \rho$ and therefore the $k$-th derivative of the potential $\phi_s$ decays like the $k$-th power of the inverse of the distance, which is to be expected since the fibres are asymptotically conical. However, the decay is much slower in the horizontal directions, because we only have the estimate $\rho^{1/2} \lesssim \mr$ globally. This reflects the fact that the metrics $\omega^{\m{srf}}_s$ are not asymptotically conical, but rather have a singular tangent cone at infinity, $\C \times (\C^2/\Z_2)$. 
\end{rem}

We note the following result which will be important later:

\begin{lem}     \label{lem: Derivative of distance function}
    $|\diff \rho|_{\omega_s^{\m{srf}}} + \rho |\partial \overline{\partial}\rho|_{\omega^{\m{srf}}_s} = \ohat(1)$.
\end{lem}

\begin{proof}
    We already know that $|\diff \rho|_{\omega_s^{\m{srf}}} = \ohat(1)$, and for $\partial\overline{\partial}\rho$ we explicitly compute:
    \begin{align*}
        \overline{\partial} \rho = \frac{1}{4} \rho^{-3} (\overline{\partial} \upsilon + 2 \eta \zeta \diff \overline{\zeta}), ~~~ \partial\overline{\partial}\rho = - \rho^{-1}( 3 \partial \rho \wedge \overline{\partial}\rho + \frac{1}{4}\rho^{-2} \partial \overline{\partial}\upsilon + \eta \rho^{-2} \diff \zeta \wedge \overline{\diff} \zeta) .
    \end{align*}
    Since $|\diff \rho|_{\omega_s^{\m{srf}}} = \ohat(1)$, $|\partial \overline{\partial}\upsilon|_{\omega_s^{\m{srf}}} = \ohat(\mr^{2}) = \ohat(\rho^2)$ and $\eta = \ohat(\rho^2)$ this yields the result.
\end{proof}

Finally, we remark that the same scaling argument as in the proof of Proposition \ref{prop: covariant derivatives at infinity} yields the following curvature estimates:

\begin{lem}     \label{lem: curvature of the skrf ansatz}
    For any $k \geq 0$, $|\nabla_{\omega^{\m{srf}}_s}^k \m{Rm}_{\omega^{\m{srf}}_s}| = \ohat(\mr^{-2-k})$.
\end{lem}

%%%%% End of Section 3 %%%%%
%%%%%%%%%%%%%%%%%%%%%%%%%%%%

%% file: section4_newfamilies.tex
%%%%%%%%%%%%%%%%%%%%%%%%%%%%%%%%%%
%%%%% Beginning of Section 4 %%%%%

    \section{New families of complete Calabi--Yau metrics}      \label{sec: New families of CY metrics}    

In the previous section, we improved the semi Ricci-flat ansatz $\omega_s^{\m{srf}}$ in the region where $R$ is large by constructing a family of approximate solutions of the Monge--Amp\`ere equation with fast decay as $R \to \infty$, uniformly in $s$. This will allow us to construct the new family of $\Un(2)$-symmetric Calabi--Yau metrics in \S\ref{subsec: The new smooth CY metrics} building on the works of \cite{hein2010gravitational,li2019new}, thereby proving an effective version of Theorem \ref{thm: A}. However, even with the asymptotic correction at infinity, the improved semi Ricci-flat ansatz behaves poorly near the exceptional curve for small values of $s$, which impedes our understanding of the singular limit when the Kähler class degenerates. For this reason, we shall introduce a further modification of our ansatz when $s \ll 1$ in \S\ref{subsec: The cdlo metrics} and \S\ref{subsec: Interpolation asymptotic regimes}, and then proceed to study of geometry of this modified ansatz in \S\ref{subsec: Geometrical properties} as a preparatory work for Section \ref{sec: The singular limit}.

    \subsection{The new smooth Calabi--Yau metrics}       \label{subsec: The new smooth CY metrics}

Let us fix a cutoff function $\chi \colon \R \to [0,1]$ such that $\chi(t) = 0$ if $t \leq \frac{1}{2}$ and $\chi(t) = 1$ if $t \geq 1$. Then for any fixed $s > 0$, it is not difficult to see that there exists $R_s > 0$ large enough such that 
\begin{equation}
    \omega^\prime_s \coloneqq \omega^{\m{srf}}_s + \tfrac{\iq}{2} \partial\overline{\partial} (\chi(\tfrac{R}{R_s}) \varphi_s)
\end{equation}
defines a Kähler metric on $\mZ_+$, where $\varphi_s$ is the asymptotic solution of the Monge--Amp\`ere equation constructed in Proposition \ref{prop: Improved solutions at infinity} (see for instance the proof of Proposition \ref{prop: interpolation metric} below). Remark that by our choice of cutoff, $\omega_s^\prime$ is invariant under the action of $\Un(2)$. Moreover, it follows from Corollary \ref{cor: Intrinsic higher-order estimates} that 
\begin{equation*}
    (\omega^\prime_s )^3 = e^{f^\prime_s} \vol_+
\end{equation*}
where the function $f^\prime_s \in C^\infty(\mZ_+)$ satisfies 
\begin{equation*}
    |\nabla_{\omega^\prime_s}^k f^\prime_s |_{\omega^\prime_s} = \mo(\mr^{-k} \rho^{-2 -\delta})
\end{equation*}
for any $\delta \in (0,1)$. Using similar arguments as in \cite{li2019new}, we obtain the first main theorem of this article:

\begin{thm}     \label{thm: the new complete metrics}
    For any $s > 0$, there exists a unique potential $u^\prime_s \in C^\infty(\mZ_+)$ such that, for all $\delta \in (0,1)$,
    \begin{equation*}
        | \nabla_{\omega^\prime_s}^k u^\prime_s |_{\omega^\prime_s} = \mo(\rho^{-\delta}\mr^{-k}), ~~~~ \text{as} ~ \rho \to \infty,
    \end{equation*}
    and $\omega_{\m{CY},s} \coloneqq \omega^\prime_s + \iq \partial\overline{\partial} u^\prime_s$ is a complete Calabi--Yau metric on $\mZ_+$ satisfying $\omega_{\m{CY},s}^3 = \vol_+$. Moreover, the following properties hold:
    \begin{enumerate}[(i)]
        \item The potential $u_s$ and the metric $\omega_{\m{CY},s}$ are $\Un(2)$-invariant.
        \item The tangent cone at infinity of $\omega_{\m{CY},s}$ is $\C \times (\C^2 / \Z_2)$.
        \item The family $\{\omega_{\m{CY},s}\}_{s \in (0,\infty)}$ is continuous in the $C^\infty_{\m{loc}}$-topology.
    \end{enumerate}
\end{thm}

\begin{proof}
    The proof of the existence and uniqueness of the potential $u^\prime_s$ and properties (i) and (ii) in the above theorem is entirely similar to Section 5 in \cite{li2019new}. We shall therefore only sketch the main threads of the argument.
    
    The point is that the metrics $\omega_s^\prime$ satisfy the $\mathsf{SOB}(6)$ property (see Definition \ref{def: SOBbeta Spaces} below), and therefore the existence of a solution $u^\prime_s$ of the Monge--Amp\`ere equation follows from \cite[Prop. 4.1]{hein2010gravitational}. The decay $|u^\prime_s| = \mo(\rho^{-\delta})$ of the potential is a classical consequence of weighted Moser iteration (cf. Appendix \ref{app: quantitative moser iteration} for more details), and the decay of the derivatives can be obtained in a similar fashion to \cite[Th. 8.6.11]{joyce2000compact}, using the same scaling argument as in the proof of Proposition \ref{prop: covariant derivatives at infinity}. The uniqueness of the solution $u^\prime_s$ amongst potentials satisfying the desired decay properties is an immediate consequence of the maximum principle, and this implies that $u^\prime_s$ is invariant under the action of $\Un(2)$. Finally, the fact that the tangent cone at infinity is $\C \times (\C^2 / \Z_2)$ is similar to \cite[Prop. 5.6]{li2019new}, see also Lemma \ref{lem: tangent cone at infinity for ansatz} below. 

    The proof of point (iii) in the above theorem does not pose any more difficulty. Indeed, if we restrict the values of $s$ to a compact interval $[s_1,s_2] \subset (0,\infty)$, we can make a fixed choice of $R_s = R_{s_1,s_2}$ for any $s \in [s_1,s_2]$ in the cutoff procedure, so as to construct a family of Kähler metrics $\{\tilde{\omega}^\prime_s\}_{s \in [s_1,s_2]}$ which is continuous in the $C^\infty_{\m{loc}}$-topology. Since $\tilde{\omega}^\prime_s - \omega^\prime_s = \iq \partial\overline{\partial} v_s$ for some compactly supported function $v_s$, this does not change the output Calabi--Yau metric $\omega_{\m{CY},s} \coloneqq \tilde{\omega}_s^\prime + \iq \partial\overline{\partial} \tilde{u}^\prime_s$ by the uniqueness part of the theorem. Now since the geometry of $\{\tilde{\omega}_s^\prime\}_{s \in [s_1,s_2]}$ is uniformly under control, we obtain uniform weighted estimates $|(1+\mr)^k (1+\rho)^\delta \nabla^k \tilde{u}_s^\prime| \leq C_k$ for all $s \in [s_1,s_2]$, which together with the uniqueness implies that $\{\omega_{\m{CY},s}\}_{s \in [s_1,s_2]}$ is continuous in the $C^\infty_{\m{loc}}$-topology.
\end{proof}

\begin{rem} \label{rem: negative branch}
    Using the $\Un(2)$-equivariant biholomorphism $\jmath \colon \mZ_- \to \mZ_+$ described after Proposition \ref{prop: Explicit form of srf ansatz}, we can also define the family of Calabi--Yau metrics $\{\omega_{\m{CY},s}\}_{s < 0}$ on $\mZ_-$ by setting $\omega_{\m{CY},s} \coloneqq \jmath^* \omega_{\m{CY},-s}$ for any $s < 0$. Of course this family of Calabi--Yau metrics satisfies the analogous properties of (i)--(iii) in the above theorem.
\end{rem}

In the remainder of this article, we shall be concerned with the behaviour of the family of Calabi--Yau metrics $\{\omega_{\m{CY},s}\}_{s \in (0,\infty)}$ when $s \to 0$. In this limit, the above argument breaks down, because the metrics $\omega_s^{\m{srf}}$ become singular as the Kähler class degenerates. What is more, when $s \to 0$ the semi Ricci-flat metrics are very far from Calabi--Yau in the region where $R$ is small, i.e. near the exceptional curve, as pointed out in Section \ref{sec: sRfK metrics}. Therefore it is necessary to introduce a further modification of our ansatz near $\mC_+$ in order to study the degenerate limit. Since the parameter $s$ determines the size of the exceptional curve which shrinks as $s\to 0$, one naturally expects this degeneration to be modelled on the family of Calabi--Yau metrics constructed by Candelas and de la Ossa in \cite{candelas1990comments}. We shall briefly describe these metrics in the next part, before explaining how to interpolate between the two asymptotic regimes.

    \subsection{The Candelas--de la Ossa metrics}   \label{subsec: The cdlo metrics}

The starting point is the cohomogeneity-one structure of $\mZ$. After mapping $\Un(1) \times \SU(2)$ to $\Un(2)$ via the homomorphism $(e^{\iq \theta},A) \mapsto e^{\iq \theta} A$, the action of $\Un(2)$ considered in Section \ref{sec: sRfK metrics} induces an action of $\Un(1) \times \SU(2)$ on $\mZ \subset \mathfrak{gl}(2,\C)$. This action in turn extends to an action of $\Un(1) \times \SU(2) \times \SU(2)$ via
\begin{equation*}
    (e^{\iq\theta},A,B) \cdot W=e^{\iq\theta}AWB^{-1}. 
\end{equation*}
This is a cohomogeneity-one action, and on $\mZ^{\m{reg}} = \mZ \setminus \{0\}$ the orbits are diffeomorphic to
\begin{equation*}
    (\SU(2) \times \SU(2) )/\Un(1) \cong \mathbb{S}^2\times \mathbb{S}^3,
\end{equation*}
and the orbit space is parametrised by the invariant function $R^2=|W|^2=u+\eta$. This cohomogeneity-one action lifts to the small resolution $\mZ_+$ with the rational curve $\mC_+$ corresponding to the unique singular orbit. On $\mZ$ or $\mZ_+$, any function invariant under the action of $\Un(1) \times \SU(2) \times \SU(2)$ must be a function of $R^2$ alone. Since the real volume forms $\vol, \vol_+$ defined in \S\ref{subsec: Relation to the flop} are also invariant under this action, the Monge--Amp\`ere equation for invariant Kähler potentials reduces to an ODE in the variable $R^2$. This is in particular the case of the conical Stenzel metric on $\mZ$ \cite{calabi1979hyperkahler,candelas1990comments,stenzel1993ricci}, which is of the form $\iq \partial \overline{\partial} (F(R^2))$. Since $R^2 = \upsilon + \eta = \frac{1}{4}\mr^4$ we can expand
\begin{align*}
    \iq \partial \overline{\partial}(F(R^2)) & = F^\prime \tfrac{\iq}{2} \mr^2  \partial \overline{\partial}(\mr^2) + \iq F^{\prime\prime}  (\partial \upsilon + \overline{\zeta}\diff \zeta) \wedge (\overline{\partial} \upsilon + \zeta \diff \overline{\zeta}) \\
        & = 2 R F^\prime \omega^\m{red}_0 + \iq F^{\prime\prime} (\partial \upsilon + \overline{\zeta}\diff \zeta) \wedge (\overline{\partial} \upsilon + \zeta \diff \overline{\zeta}) .
\end{align*}
Noticing that $((\partial \upsilon + \overline{\zeta} \diff \zeta) \wedge (\overline{\partial} \upsilon \wedge \zeta \diff \overline{\zeta}))^2 = 0$, we obtain
\begin{align*}
    (\iq \partial \overline{\partial} F)^3 & = 8R^3 \cdot(F^\prime)^2 (\omega_0^\m{red})^3 + 4\iq  R^2\cdot (F^\prime)^2 F^{\prime\prime} (\omega_0^\m{red})^2 \wedge (\partial \upsilon \wedge \overline{\partial} \upsilon + \eta \diff \zeta \wedge \diff \overline{\zeta}) \\
        & = \left( 4 R^2 \cdot(F^\prime)^3 + \tfrac{8}{3} R^4 \cdot(F^\prime)^2 F^{\prime\prime} \right) \vol .
\end{align*}
The Stenzel metric corresponds to the solution of the ODE given by
\begin{equation*}
    \mathcal{K}_0 = \frac{3\gamma}{2}R^{4/3} = \gamma \widetilde{\mathcal{K}}_0(R^2) ,
\end{equation*}
where $\widetilde{\mathcal{K}}_0(t) = \frac{3}{2}t^{2/3}$ and the constant $\gamma > 0$ is given by $\gamma^3 = \frac{28}{9}$. The specific value of $\gamma$ is unimportant, and merely reflects our choice of scaling for the volume form. With this choice of normalisation let us define the conical Stenzel metric as
\begin{equation}
    \omega_{\m{Stz}} \coloneqq \iq \partial\overline{\partial} \mathcal{K}_0 .
\end{equation}

In \cite{candelas1990comments}, Candelas and de la Ossa show that $\widetilde{\mathcal{K}}_0$ is part of a $1$-parameter family of functions $\widetilde{\mathcal{K}}_a \colon [0,+\infty) \to \R$, defined for $a \geq 0$, satisfying the ODE
\begin{equation*}
    (t\widetilde{\mathcal{K}}_a^\prime)^3 + 6 a^2 (t\widetilde{\mathcal{K}}_a^\prime)^2 = t^2 
\end{equation*}
and such that for any $a > 0$, the $(1,1)$-forms $\widetilde{\omega}_a = 4 a^2 \widetilde{\omega}^+_{\m{FS}} + \iq \partial \overline{\partial}(\widetilde{\mathcal{K}}_a(R^2))$ defined on the small resolution $\mZ_+$ are Ricci-flat Kähler forms. More precisely, $\tilde{\omega}_a$ satisfies $(\widetilde{\omega}_a)^3 = \widetilde{\vol}_+$ where the volume form $\widetilde{\vol}_+$ on $\mZ_+$ must be given by $\widetilde{\vol}_+ = \gamma^{-3}\vol_+$. In order to make a choice of scaling for the Candelas--de la Ossa metrics consistent with the other conventions used in the present article, we shall define for any $s \geq 0$ a constant $a_s \geq 0$ and a function $\mathcal{K}_s(R)$ through the identities
\begin{equation}
    s = 4 \gamma a^2_s, ~~~ \mathcal{K}_s(R) \coloneqq \gamma \widetilde{\mathcal{K}}_{a_s}(R^2) ,
\end{equation}
and set
\begin{equation}
    \omega_{\m{CdlO},s} \coloneqq s \widetilde{\omega}^+_{\m{FS}} + \iq \partial \overline{\partial} \mathcal{K}_s.
\end{equation} 
Therefore, for $s = 0$ we have $\omega_{\m{CdlO},0} = \omega_{\m{Stz}}$, and for any $s > 0$ we have
\begin{equation*}
    \omega_{\m{CdlO},s}^3 = \vol_+ .
\end{equation*}
Thus $\{\omega_{\m{CdlO},s}\}_{s > 0}$ is a family of Calabi--Yau metrics on $\mZ_+$, and moreover $(\mZ_+,\omega_{\m{CdlO},s},o_+)$ converges to $(\mZ,\omega_{\m{Stz}},o)$ in the pointed Gromov--Hausdorff sense, and also locally smoothly away from $\mC_+$ (under the identification $\mZ_+ \setminus \mC_+ \cong \mZ^{\m{reg}}$). In fact, we even have smooth convergence at the level of potentials: the family $\{\mathcal{K}_s\}_{s \geq 0}$ is continuous in the $C^{\infty}_{\m{loc}}(\mZ_+ \setminus \mC_+)$ sense; in particular, for any $0 < R_1 < R_2$, $\left. \mathcal{K}_s \right|_{[R_1,R_2]} \to \left. \mathcal{K}_0 \right|_{[R_1,R_2]}$ in the $C^\infty$-sense.

For later use, we shall define the $\Un(1) \times \SU(2) \times \SU(2)$-invariant functions
\begin{equation}
    \Lambda_s(R)\coloneqq \gamma R^2\widetilde{\mathcal{K}}_{a_s}'(R^2) ~~ \text{and} ~~\varrho_s^2 \coloneqq 3 \left(\Lambda_s(R) +s\right)\geq \tfrac{3}{2}s ,
\end{equation}
for any $s \geq 0$. Notice that $\Lambda_s$ is a positive solution of $\Lambda_s^3+\tfrac{3}{2}s\Lambda^2_s=\gamma^3R^4$. Moreover, if we define
\begin{equation}
    \varrho^2 \coloneqq \varrho^2_0 = 3\gamma R^{4/3}
\end{equation}
then we see that $\omega_{\m{Stz}} = \iq\partial\overline{\partial}(\frac{\varrho^2}{2})$, so that $\varrho$ is the radius function of the conical metric $\omega_{\m{Stz}}$. For $s > 0$, $\varrho_s$ has the interpretation of a regularised radius function for $\omega_{\m{CdlO},s}$; in particular there are a constants $\kappa \geq 1$, $D > 0$ such that
\begin{equation*}
    \kappa^{-1} \varrho_s - D s^{1/2} \leq \dist_{\omega_{\m{CdlO},s}}(o_+,\cdot) \leq \kappa \varrho_s + D s^{1/2}
\end{equation*}
for any $s > 0$.

    \subsection{Interpolation between the two asymptotic regimes}       \label{subsec: Interpolation asymptotic regimes}

We now explain how to interpolate between the Candelas--de-la-Ossa metrics (near $\mC_+$) and our asymptotic ansatz of solution constructed in Section \ref{sec: Asymptotic solutions of the MA equation} (at infinity), at least when the parameter $s$ is small enough. More precisely, we prove:

\begin{prop}    \label{prop: interpolation metric}
    There exist constants $s_0,c_0,R_1,R_2 > 0$ and a $1$-parameter family of $\Un(2)$-symmetric potentials $\{\phi_s = F(s,R,\eta)\}_{s \in [0,s_0]}$ satisfying the following:
    \begin{enumerate}[(i)]
        \item $\phi_0$ is smooth on $\mZ^{\m{reg}}$ and $\omega_0 \coloneqq \iq \partial \overline{\partial} \phi_0$ is Kähler.
        \item For any $s \in (0,s_0]$, $\phi_s$ is smooth on $\mZ_+$ and $\omega_s \coloneqq s \widetilde{\omega}^{+}_{\m{FS}} + \iq \partial \overline{\partial}\phi_s$ is Kähler.
        \item For any $s \in [0,s_0]$, $\left. \phi_s \right|_{\{R \leq R_1\}} = \mathcal{K}_s(R) + c_0$; in particular $\omega_s$ coincides with the Candelas--de la Ossa metric $\omega_{\m{CdlO},s}$ on the domain $\{R \leq R_1\}$.
        \item For any $s \in [0,s_0]$, $\left. \omega_s \right|_{\{R \geq R_2 \}} = \left. \omega^{\m{srf}}_s + \tfrac{\iq}{2} \partial \overline{\partial} \varphi_s \right|_{\{R \geq R_2 \}}$, where $\varphi_s$ is the admissible potential of Proposition \ref{prop: Improved solutions at infinity}.
        \item $\{\phi_s\}_{s \in [0,s_0]}$ is continuous in the $C^\infty_{\mathrm{loc}}$ sense away from the locus $\{R = 0\}$.
    \end{enumerate}
\end{prop}

\begin{proof}
    The crucial part is to understand how to construct $\phi_0$, then we will be able to construct $\phi_s$ using cutoffs for $s > 0$ sufficiently small. For $s = 0$, the normalised Stenzel metric is given by $\omega_{\m{CdlO},0} = \iq \partial \overline{\partial} \mathcal{K}_0 =  \iq \partial \overline{\partial}(\frac{3\gamma}{2} R^{4/3})$, while $\omega_0 = \iq \partial \overline{\partial}(R + \frac{\eta}{2})$. In particular, the function $R + \frac{\eta}{2}$ is plurisubharmonic on $\mZ^{\mathrm{reg}}$, and therefore it is a classical fact that for any non-decreasing function $F \colon (0,+\infty) \to \R$ such that $F^{\prime\prime} > 0$, $F(R+\frac{\eta}{2})$ will remain plurisubharmonic. 

    The idea is to first pick a convex function $F \colon [0,+\infty) \to \R$ such that $F^\prime, F^{\prime\prime} > 0$ on $(0,+\infty)$ and such that there exist $T_2 > T_1 >0$ and a constant $c_0 > 0$ such that $F(t) = c_0 + \frac{3\gamma}{2} t^{4/3}$ if $t \leq T_1$ and $F(t) = t$ if $t \geq T_2$. To construct such a function, we can for instance pick a cutoff function $\chi \colon [0,+\infty) \to \R$ such that $\chi(t) = 1$ near $0$ and $\chi(t) = 0$ when $t$ is large enough, and such that $\int_0^\infty \frac{2\gamma}{3} \chi(t) t^{-2/3} \diff t = 1$; such a cutoff function exists since $t^{-2/3}$ is integrable on $[0,1]$. Thus the function 
    \begin{equation*}
        F_0(t) = \int_0^t \int_0^x \frac{2\gamma}{3} \chi(y) y^{-2/3} \diff y \diff x
    \end{equation*}
    satisfies $F_0^\prime, F_0^{\prime\prime} > 0$, $F_0(t) = \frac{3\gamma}{2}t^{4/3}$ for $t$ small enough, and $F^\prime(t) = 1$ for $t$ large enough. Given the convexity of the function $F_0$, it follows that there exists a constant $c_0 > 0$ such that $F(t) = t-c_0$ for $t$ large enough\footnote{In fact it is not difficult to see that we could make the constant $c_0 > 0$ as small as we want by choosing an appropriate cutoff $\chi$, but we cannot achieve $c_0 = 0$ since we want $F$ to be convex.}. Hence the function $F(t) = c_0 + F_0(t)$ satisfies the desired properties. In particular, the function $F(R + \frac{\eta}{2})$ is plurisubharmonic. 

    We want to study the expansion of the Kähler form $\iq \partial \overline{\partial}(F(R + \frac{\eta}{2}))$ in the two asymptotic regimes $R \to 0$ and $R \to \infty$. First, we begin with the limit $R \to 0$. Since $\eta \leq R^2$, there exists $R_1 > 0$ small enough such that $F(R+\frac{\eta}{2}) = c_0 + \frac{3\gamma}{2}(R + \frac{\eta}{2})^{4/3}$ in the region $\{ R \leq R_1 \}$. Thus, we may write
    \begin{equation*}
        F(R,\eta) = c_0 + \frac{3\gamma}{2} R^{4/3} + \tilde{F}(R,\eta)
    \end{equation*}
    where in the limit $R \to 0$ we have $\tilde{F}(R,\eta) = \frac{3\gamma}{2}R^{4/3}(\frac{2\eta}{3R} + \frac{\eta^2}{18R^2} + \cdots)$. In particular,
    \begin{align*}
        |\tilde{F}| = \mo(R^{7/3}), ~~ |\tilde{F}_R| = \mo(R^{4/3}), ~~ |\tilde{F}_{RR}| = \mo(R^{1/3}), ~~ |\tilde{F}_\eta| = \mo(R^{1/3}) \\
        |\tilde{F}_{R\eta}| = \mo(R^{-2/3}) , ~~ |\tilde{F}_{\eta\eta}| = \mo(R^{-2/3}) . ~~~~~~~~~~~~~~~~~~~~~~~~~~
    \end{align*}
    In order to interpolate with the potential of the Stenzel metric, we take a cutoff function $\chi$ as before and seek a potential of the form $\phi_\epsilon = c_0 + \frac{3\gamma}{2} R^{4/3} + (1-\chi(\frac{R}{\epsilon})) \tilde{F}(R,\eta)$. In order to show that this potential is plurisubharmonic for $\epsilon > 0$ small enough, we calculate
    \begin{align*}
        \iq \partial \overline{\partial} \phi_\epsilon & = \iq \partial \overline{\partial}\mathcal{K}_0 + (1-\chi(\tfrac{R}{\epsilon})) \iq \partial \overline{\partial} \tilde{F} - \epsilon^{-1} \chi^\prime(\tfrac{R}{\epsilon}) \iq (\partial R \wedge \overline{\partial} F + \partial F \wedge \overline{\partial} R) \\
            & ~~~~ - \epsilon^{-2} \chi^{\prime\prime}(\tfrac{R}{\epsilon}) \tilde{F}  \iq \partial R \wedge \overline{\partial} R .
    \end{align*}
    Since the distance to the vertex of the cone $o \in \mZ$ for the Stenzel metric is $\varrho_0 \asymp R^{2/3}$ and the fibration $\zeta$ satisfies $\zeta(\lambda \cdot W) = \lambda \zeta(W)$ for any $W \in \mZ$ and $\lambda > 0$, we deduce that with respect to this metric,
    \begin{equation*}
        |\diff R|, ~ |\partial R|, ~ |\overline{\partial} R|, ~ |\diff \zeta| , ~ |\diff \overline{\zeta}| = \mo(R^{1/3}), ~~ |\partial \overline{\partial} R| = \mo(R^{-1/3})
    \end{equation*}
    as $R \to 0$. This notably implies that
    \begin{align*}
        |\partial \tilde{F}| & = |\tilde{F}_R \partial R + \overline{\zeta} \tilde{F}_\eta \diff \zeta| = \mo(R^{4/3}), ~~ |\overline{\partial}\tilde{F}| = \mo(R^{4/3}), \\
            |\partial \overline{\partial} \tilde{F}| & = |\tilde{F}_R \partial \overline{\partial} R + \tilde{F}_{RR} \partial R \wedge \overline{\partial}R + \tilde{F}_{R\eta} (\zeta \partial R \wedge \diff \overline{\zeta} + \overline{\zeta} \diff \zeta \wedge \overline{\partial} R + (\tilde{F}_\eta + \eta \tilde{F}_{\eta\eta}) \diff \zeta \wedge \diff \overline{\zeta}| \\
            & = \mo(R)
    \end{align*}
    as $R \to 0$. Combining these estimates together, it follows that $\iq \partial \overline{\partial}\phi_\epsilon = \iq \partial \overline{\partial} \mathcal{K}_0 + \mo(\epsilon^{2/3})$, and thus by choosing $\epsilon_0 > 0$ small enough the function $\phi_{\epsilon_0}$ can be made plurisubharmonic.

    For this fixed choice of $\epsilon_0 > 0$, the Kähler form $\iq \partial \overline{\partial}\phi_{\epsilon_0}$ coincides with the semi Ricci-flat Kähler metric $\omega_0$ for $R = \tfrac{\mr^2}{2}$ large enough. On the other hand, since the potential $\varphi_0$ is adapted, we have already proved that $|\partial \overline{\partial}\varphi_0|_{\omega_0} = \mo(\mr^{-2}) = \mo(R^{-1})$. Similarly, it is easy to see from the definition of adapted potentials that $|\partial \varphi_0| \asymp |\overline{\partial}\varphi_0| = \mo(\mr^{-1}) = \mo(R^{-1/2})$, and from the expression of $\varphi_0$ we have $|\varphi_0| = \mo(\log(R))$, as $R \to \infty$. Thus we can seek a potential of the form $\phi_{\epsilon_0,C} = \phi_{\epsilon_0} + (1-\chi(\frac{R}{C})) \varphi_0$, for some $C > 0$ large enough. As before,
    \begin{align*}
        \iq \partial \overline{\partial} \phi_{\epsilon_0,C} & = \iq \partial \overline{\partial}\phi_{\epsilon_0} + (1-\chi(\tfrac{R}{C})) \iq \partial \overline{\partial} \varphi_0 - C^{-1} \chi^\prime(\tfrac{R}{C}) \iq (\partial R \wedge \overline{\partial} \varphi_0 + \partial F \wedge \overline{\partial} \varphi_0) \\
        & ~~~~ - C^{-2} \chi^{\prime\prime}(\tfrac{R}{C}) \varphi_0  \iq \partial R \wedge \overline{\partial} R .
    \end{align*}
    Note that as $R \to \infty$, $|\diff (R^2)|_{\omega^{\m{srf}}_0} \leq |\diff \upsilon|_{\omega^{\m{srf}}_0} + |\diff \eta|_{\omega^{\m{srf}}_0} = \mo(\mr^3) = \mo(R^{3/2})$ and thus $|\diff R |_{\omega^{\m{srf}}_0} \asymp |\partial R|_{\omega^{\m{srf}}_0} \asymp |\overline{\partial}R|_{\omega^{\m{srf}}_0} = \mo(R^{1/2})$ as $R \to \infty$. In particular, we deduce that $\iq \partial \overline{\partial} \phi_{\epsilon_0,C} = \omega_0 + \mo(C^{-1/2})$ for $C$ sufficiently large, and therefore we can choose $C_0 > 0$ large enough so that $\phi_0 = \phi_{\epsilon_0,C_0}$ is plurisubharmonic.

    At this stage, we obtain $\phi_0$ satisfying the desired property for some radii $\tilde{R}_1=2R_1,\tilde{R}_2 = \tfrac{R_2}{2}$. To construct $\phi_s$, we take a cutoffs between $R_1$ and $2R_1$, and $\tfrac{R_2}{2}$ and $R_2$, to glue in the potentials $\mathcal{K}_s-\mathcal{K}_0$ in the region $\{R_1 \leq R \leq 2R_1\}$ and $\sqrt{R^2+s^2}-R - s \log(s + \sqrt{R^2+s^2}) + \tfrac{1}{2} (\varphi_s-\varphi_0)$ in the region $\{\tfrac{R_2}{2} \leq R \leq R_2\}$. Since the difference of the potentials is continuous in the $C^\infty_{\mathrm{loc}}$ sense, there exists $s_0 > 0$ such that for any $s \in [0,s_0]$, the resulting family of functions $\phi_s$ is plurisubharmonic and satisfies the desired properties.
\end{proof}

We now list a few properties which immediately follow from the construction of $\omega_s$.

\begin{lem}     \label{lem: Distance for main ansatz}
    For any $s \in (0,s_0]$, $(\omega_s,\mZ_+)$ is complete, and $(\mZ^{\m{reg}},\omega_0)$ is an incomplete Kähler manifold whose completion is homeomorphic to $\mZ$. Moreover, if we define the function
    \begin{equation*}
        \rho_s \coloneqq \chi(\tfrac{R}{R_1}) \varrho_s + (1-\chi(\tfrac{R}{R_1})) \rho
    \end{equation*}
    where $\chi \colon \R \to [0,1]$ is a smooth cutoff function such that $\chi(t) = 1$ if $t \leq 1$ and $\chi(t) = 0$ if $t \geq 2$, then there are constants $\kappa > 1$, $D > 0$ such that
    \begin{itemize}
        \item for all $s \in (0,s_0]$, $\kappa^{-1} \rho_s - D s^{1/2} \leq \dist_{\omega_s}(o_+,\cdot) \leq \kappa \rho_s + D s^{1/2}$, and
        \item $\kappa^{-1} \rho_0 \leq \dist_{\omega_0}(o,\cdot) \leq \kappa \rho_0$.
    \end{itemize}
\end{lem}

\begin{lem}     \label{lem: tangent cone at infinity for ansatz}
    For any $s \in [0,s_0]$, the tangent cone at infinity of $\omega_s$ is $\C \times (\C^2/\Z_2)$.
\end{lem}

\begin{proof}
    This is similar to \cite[Prop. 5.6]{li2019new}: $\omega_s$ is asymptotically equivalent to a fibration over the Euclidean plane $\C$, and the exceptional sphere in the Eguchi--Hanson fibres has radius of order $(\eta+s^2)^{1/4} \ll \rho_s$, thus the tangent cone at infinity is indeed $\C \times (\C^2/\Z_2)$. 
\end{proof}

\begin{lem}     \label{lem: functions fs}
    The $1$-parameter family of $\Un(2)$-invariant functions $\{f_s \colon \mZ_+ \to \R\}_{s \in [0,s_0]}$ defined through the identity
    \begin{equation*}
        \omega_s^3 = e^{-f_s} \vol_+, 
    \end{equation*}
    for any $s \in [0,s_0]$, satisfies the following:
    \begin{enumerate}[(i)]
        \item $\left. f_s \right|_{R \leq R_1} \equiv 0$ for any $s \in [0,s_0]$.
        \item For any $\mu \in (2,3)$, there exists a constant $A_\mu > 0$ such that 
        \begin{equation*}
            \sup_{s \in [0,s_0]} \sup \{\rho^\mu|1-e^{f_s}|\} \leq A_\mu,
        \end{equation*} 
        where $\rho = (\upsilon + \eta^2)^{\frac{1}{4}}$ as in \S\ref{subsec: The distance function}.
        \item For any $\mu \in (2,3)$ and $k \geq 1$, there exists a constant $C_{k,\mu} > 0$ such that for any $s \in [0,s_0]$, 
        \begin{equation*}
            |\nabla_{\omega_s}^{k-1} \diff f_s|_{\omega_s} \leq C_{k,\mu} \mr^{-k}\rho^{-\mu} .
        \end{equation*}
        In particular, there exists a constant $C_\mu > 0$ such that for any $s \in [0,s_0]$,
        \begin{equation*}
            |\partial\overline{\partial}f_s|_{\omega_s} \leq C_\mu \mr^{-2} \rho_s^{-\mu} .
        \end{equation*}
    \end{enumerate}
\end{lem}

\begin{proof}
    Point (i) is obvious by construction, point (ii) follows from Proposition \ref{prop: Improved solutions at infinity} and point (iii) from Corollary \ref{cor: Intrinsic higher-order estimates}.
\end{proof}

    \subsection{Geometrical properties of the modified ansatz}      \label{subsec: Geometrical properties}

In this section we discuss the geometry of the family of Kähler metrics $\{\omega_s\}_{s \in (0,s_0]}$ on $\mZ_+$. We begin with estimates on the curvatures and the covariant derivatives of useful functions.

\begin{lem}     \label{lem: curvature estimates for SOB}
    For any $k \geq 0$, there exists a constant $C_k > 0$ such that for any $s \in (0,s_0]$,
    \begin{align*}
            |\nabla_{\omega_s}^k\m{Rm}_{\omega_s}|_{\omega_s} \leq \left\{\begin{array}{ll}
                 C_k \rho_s^{-2-k} & \text{in } \{ R \leq R_1\}, \\
                 C_k \mr_s^{-2-k} & \text{in } \{ R \geq R_1\}.
            \end{array}\right.
        \end{align*}
\end{lem}

\begin{proof}
    In the region $\{R_1 \leq R \leq R_2\}$ this is clear since the family of metrics $\{\omega_s\}_{s \in [0,s_0]}$ is continuous in the $C^\infty_{\m{loc}}(\mZ_+ \setminus \mC_+)$ topology and we have uniform two-sided bounds on $r_s$, $\rho_s$. In the region $\{R \geq R_2\}$, this follows from the admissibility of the family of potentials $\{\varphi_s\}_{s \in [0,s_0]}$ together with Proposition \ref{prop: covariant derivatives at infinity} and Lemma \ref{lem: curvature of the skrf ansatz}. In the region $\{R \leq R_1\}$, $\omega_s$ coincides with the Candelas-de la Ossa metric $\omega_{\m{CdlO},s}$ and $\rho_s$ with $\varrho_s$. Since $\omega_{\m{CdlO},1}$ is asymptotically conical, we have estimates $|\nabla_{\omega_{\m{CdlO},1}}^k \m{Rm}_{\omega_{\m{CdlO},1}}| \lesssim \varrho^{-2-k}_1$, and since all the Candelas-de la Ossa metrics are isometric to rescalings of $\omega_{\m{CdlO},1}$ this yields the desired estimates.
\end{proof}

\begin{lem}     \label{lem: radius derivative estimates}
    For every $k\geq 1$ there exists a constant $C^\prime_k > 0$ such that for all $s \in (0,s_0]$,
    \begin{align*}
        |\nabla_{\omega_s}^{k-1} \diff \rho_s|_{\omega_s} &\leq C^\prime_k\rho_s^{1-k} ~~~~~ \text{in } \{ R \leq R_1\}, \\
        |\nabla_{\omega_s}^{k-1} \diff \rho_s|_{\omega_s} &\leq C^\prime_k \mr_s^{1-k} ~~~~~ \text{in } \{R \geq R_1\}, \\
        |\nabla_{\omega_s}^{k-1}\diff r_s|_{\omega_s} &\leq C^\prime_k \mr_s^{1-k} ~~~~~ \text{in } \{R \geq R_1\}.
    \end{align*}
\end{lem}

\begin{proof}
    As before, the estimates are obvious in the transition region $\{ R_1 \leq R \leq R_2 \}$, and in region $\{ R \geq R_2\}$ the estimates easily follow from Proposition \ref{prop: covariant derivatives at infinity} and the admissibility of the family of potentials $\{\varphi_s\}_{s \in [0,s_0]}$. As for the region $\{ R \leq R_1\}$, after reparametrisation of $\mZ_+$ all the Candelas--de la Ossa metrics are isometric to rescalings of $\omega_{\m{CdlO},1}$, and $\varrho_s$ is the natural (rescaled) distance function. On the other hand, $\omega_{\m{CdlO},1}$ is asymptotically conical, and we have estimates $|\nabla^{k-1} \diff \varrho_1| \lesssim \varrho_1^{1-k}$. Since these estimates are invariant under a rescaling of the metric, this implies the desired estimates in $\{R \leq R_1\}$.
\end{proof}

We now recall the definition of $\mathsf{SOB}(\beta)$-spaces, first introduced by Hein \cite[Def. 3.1]{hein2010gravitational} in order to study the complex Monge--Am\`ere equation on noncompact Kähler manifolds whose geometry is well under control. We state it according to Li's version\footnote{Cf. the remark following Definition 5.7 in \cite{li2019new} for the difference between the two definitions. The point is that this does not affect the Sobolev inequalities that come with the $\mathsf{SOB}(\beta)$-package.} \cite[Def. 5.7]{li2019new}.

\begin{Def}     \label{def: SOBbeta Spaces}
    Let $(M,g)$ be a complete Riemannian manifold and $x_0\in M$. Then $(M,g)$ satisfies $\mathsf{SOB}(\beta)$, for some $\beta > 0$, if there exists $\delta \in (0,1)$, $C > 1$ such that:
    \begin{itemize}
        \item \textit{Volume growth:} $\mathrm{Vol}(B(x_0,r)) \leq C r^\beta $ for all $r \geq C$.
        %\item \textit{Local non-collapsing:} there exists $c>0$ such that for every $x \in M$ and $r \leq 1$ we have $cr^n\leq \mathrm{Vol}(B(x,r))$.
        \item \textit{Large-scale non-collapsing:} for all $x \in M$ with $r(x) \coloneqq \dist(x_0,x) \geq C$ we have $\mathrm{Vol}(B(x,(1-\frac{1}{C})r(x))) \geq \frac{1}{C} r(x)^\beta$.
        \item \textit{Annular connectedness:} any two points $x,y \in M$ with $\dist(x_0,x) = \dist(x_0,y) = r \geq C$ are connected within the annulus $A(x_0,(1-\delta)r,(1+\delta)r)$.
        \item \textit{Ricci bound:} for any $x$ such that $r(x) \coloneqq \dist(x_0,x) \geq C$, $\m{Ric} \geq - C r(x)^{-2}$.
    \end{itemize}
\end{Def} 

\begin{prop}      \label{prop: Uniform SOB(6) property}
    The family of Kähler metrics $\{\omega_s\}_{s \in (0,s_0]}$ on $\mZ_+$ satisfies the $\mathsf{SOB}(6)$-property uniformly: there exist $\delta \in (0,1)$, $C > 1$ such that:
    \begin{itemize}
        \item \emph{Uniform volume growth:} for any $s \in (0,s_0]$, $\frac{1}{C} r^6 \leq \mathrm{Vol}_{\omega_s}(B_{\omega_s}(o_+,r))\leq C r^6 $ for all $r \geq C$.
        \item \emph{Uniform large-scale non-collapsing:} for all $s \in (0,s_0]$ and for all $x\in \mathcal{Z}_+$ with $r(x) \coloneqq \dist_{\omega_s}(o_+,x) \geq C$, $\frac{1}{C} r(x)^6\leq \m{Vol}(B_{\omega_s}(x,(1-\frac{1}{C})r(x))$
        \item \emph{Uniform annular connectedness:} for any $s \in (0,s_0]$, any two points $x,y \in \mZ_+$ with $\dist_{\omega_s}(x_0,x) = \dist_{\omega_s}(x_0,y) = r \geq C$ are connected within the annulus $A_{\omega_s}(x_0,(1-\delta)r,(1+\delta)r)$.
        \item \emph{Uniform Ricci bound:} For all $s \in (0,s_0]$ and $x \in \mZ_+$ with $r(x) \coloneqq \dist_{\omega_s}(o_+,x)$,
        \begin{align*}
            |\m{Ric}_{\omega_s}(x)|_{\omega_s} \leq \left\{\begin{array}{ll}
                 0& \text{in } \{ R \leq R_1\}, \\
                 C r(x)^{-3}& \text{in } \{ R \geq R_1 \}.
            \end{array}\right.
        \end{align*}
    \end{itemize}
    Moreover, the following also holds:
    \begin{itemize}
        \item \emph{Uniform local non-collapsing:} for all $s \in (0,s_0]$, $x\in \mathcal{Z}_+$ and $r \leq 1$ we have $\frac{1}{C} r^6 \leq \m{Vol}_{\omega_s}(B_{\omega_s}(x,r))$.
    \end{itemize}
\end{prop}

\begin{proof}
    We argue as in \cite[Sec. 5]{li2019new}, adapting Li's distance and volume comparisons to our setting in order to prove the uniformity with respect to $s$.

    \textit{Uniform volume growth:} In $\{R \geq R_2\}$, the distance $\dist_{\omega_s}(o_+,\cdot)$ is equivalent to $\rho = (\upsilon + \eta^2)^\frac{1}{4}$, uniformly in $s$. Moreover, in this region $\omega_s = \omega_s^{\m{srf}} + \tfrac{\iq}{2}\partial \overline{\partial} \varphi_s$ coincides with the corrected semi Ricci-flat metric. Using Lemma \ref{lem: asymptotic equivalence of skrf metrics} and the admissibility of the family of potentials $\{\varphi_s\}_{s \in [0,s_0]}$ it follows that there exists a constant $C > 1$ such that for any $s \in (0,s_0]$,
    \begin{align*}
        C^{-1}\left. \omega_s \right|_{\{R \geq R_2\}} \leq \left. \omega^{\m{srf}}_0 \right|_{\{R \geq R_2\}} \leq C \left.\omega_s \right|_{\{R \geq R_2\}} .
    \end{align*}
    In addition, since $f_s$ is uniformly bounded, the volume forms of $\omega_s$ are uniformly equivalent to $\vol_+$. Thus it suffices to prove the result for $\omega_0$ (seen as a Kähler form on $\mZ^{\m{reg}} \cong \mZ_+ \setminus \mC_+$).

    Now using the equivalence of $\dist_{\omega_0}(o,\cdot)$ and $\rho \asymp \upsilon^{1/4} + \eta^{1/2}$ in $\{ R \geq R_2\}$, there exist $0<c<C$ such that for any sufficiently large $r$,
    \begin{align*}
        \bigl\{\eta^{1/2}+\upsilon ^{1/4}\leq c r\bigr\} \subset B_{\omega_0}(o,r) \subset \bigl\{\eta^{1/2}+\upsilon ^{1/4}\leq C r\bigr\}.
    \end{align*}
    Equivalently,
    \begin{align*}
        \{\eta \leq c^2 r^2,\upsilon \leq c^4 r^4\} \subset B_{\omega_0}(o,r) \subset \{\eta \le C^2 r^2,\upsilon \le C^4 r^4\}.
    \end{align*}
    Thus it suffices to estimate the $\omega^{\m{srf}}_0$-volume of the regions $\{\eta\le r^2,\upsilon \le r^4\} \cap \{ R \geq R_2\}$. In $\{R \geq R_2\}$, as in the proof of Proposition \ref{prop: covariant derivatives at infinity} we know that $\omega_0^{\m{srf}}$ is uniformly equivalent to $\tilde{\omega}_0 \coloneqq \hat{\omega}_0 + \frac{\iq}{2} \diff \zeta \wedge \diff \overline{\zeta}$, which is a Riemannian fibration over the $\zeta$-Euclidean plane, with fibre metric $\hat \omega_0$ being the Eguchi--Hanson metric corresponding to the parameter $\zeta$. By the previous comparison, it is enough to estimate the $\tilde{\omega}_0$-volume of the region
    \begin{equation*}
        A_r \coloneqq \{\eta \leq c^2 r^2,\upsilon \leq c^4 r^4\} .
    \end{equation*}
    Since the base disc $D_r = \{ \eta \leq r^2\}$ has area of order $r^2$, it only remains to show that the fibres have area of order $r^4$ for the vertical metric $\hat{\omega}_0$. 

    To estimate the volume along the fibres, note that for any $a \in \C$ such that $|a|^2 \leq C r^2$, the fibre region $\{\upsilon \leq r^4\}\cap \zeta^{-1}(a)$ has diameter uniformly equivalent to $r$ for $r$ large enough (cf. Proposition \ref{prop: distance estimate})\footnote{The proposition gives a uniform upper bound, but one can use the expression of the length of a vertical radial path to deduce a uniform lower bound as long as $\eta \leq C r^2 \ll r^4$.}. Since the Eguchi--Hanson metrics are complete Ricci-flat metrics with maximal volume growth and tangent cone at infinity $\C^2 / \Z_2$, Bishop--Gromov volume monotonicity implies that there exist constants $c^\prime, C^\prime > 0$ independent of $r$ large enough such that
    \begin{equation*}
        c^\prime r^4 \leq \m{Vol}_{\hat{\omega}_0}(\{\upsilon \leq r^4\}\cap \zeta^{-1}(a)) \leq C^\prime r^4
    \end{equation*}
    for any $a \in \C$ such that $|a|^2 \leq C r^2$. The uniform volume growth follows.

    \emph{Uniform large-scale non-collapsing:} As before, it is enough to prove the result for $\tilde{\omega}_0$, provided that we let $C$ be large enough so that if $r(x) \coloneqq \dist_{\omega_s}(o_+,x) \geq C$ then $B_{\omega_s}(x,(1-\frac{1}{C})r(x)) \subset \{R \geq R_2\}$, which we can do uniformly with respect to $s \in (0,s_0]$ given our distance estimate. Then we argue as above for the lower bound on $B_{\tilde{\omega}_0}(o_+,r)$: since $\tilde{\omega}_0$ is a Riemannian fibration over $\C$ with Eguchi--Hanson fibres, volume monotonicity along the fibres implies the desired bound.

    \emph{Uniform annular connectedness:} Given Lemma \ref{lem: Distance for main ansatz} and Lemma \ref{lem: asymptotic equivalence of skrf metrics}, it is clearly enough to prove the property for $\omega_0$. Now the tangent cone of $\omega_0$ at infinity is $\C \times (\C^2/\Z_2)$ which has connected link, and as in \cite[Prop. 5.8]{li2019new} this implies the connectedness of sufficiently large annuli.

    \emph{Uniform Ricci bound:} since $\omega_s^3 = e^{-f_s} \vol_+$ this immediately follows from Lemma \ref{lem: functions fs}. 

    \textit{Uniform local non-collapsing:} Since the Candelas-de la Ossa metrics $\{\omega_{\m{CdlO},s}\}_{s > 0}$ form a non-collapsed family of Ricci-flat metrics on $\mZ_+$, the result clearly holds in the region $\{ R \leq R_1 \}$. Moreover, the family of Kähler metrics $\{\omega_s\}_{s \in (0,s_0]}$ converges to $\omega_0 = \iq \partial \overline{\partial}\phi_0$ in $C^\infty_{\m{loc}}(\mZ_+ \setminus \mC_+)$ by Proposition \ref{prop: interpolation metric}, and therefore the result also holds on $\{ R_1 \leq R \leq R_2\}$. Finally, in the region $\{ R \geq R_2\}$, $\omega_s$ is uniformly equivalent to $\tilde{\omega}_0$ which has Eguchi--Hanson fibres over a flat base and the local non-collapsing follows as before from volume monotonicity along the fibres.
\end{proof}

From the analytical point of view, one of the advantages of $\mathsf{SOB}(\beta)$-spaces is that they automatically satisfy certain weighted Sobolev inequalities. In \cite[Prop. 3.2]{hein2010gravitational}, Hein proved that if $(M^n,g)$ is an $n$-dimensional $\mathsf{SOB}(\beta)$-space and $r$ the distance to a fixed point of $M$, then for any $\alpha \in [1,\frac{n}{n-2}]$ there exists a constant $A > 0$ such that for any $u \in C^\infty_c(M)$,
\begin{equation*}
    \left( \int_M (1+r)^{\alpha(\beta-2)-\beta} |u|^{2\alpha} \vol_g\right)^{\frac{1}{\alpha}} \leq A \int_M |\diff u|_g^2 \vol_g .
\end{equation*}
Notice that in the particular case where $\beta = n$ and $\alpha = \frac{n}{n-2}$, this becomes an \emph{unweighted} Sobolev bound,
\begin{equation*}
    \left( \int_M |u|^{\frac{2n}{n-2}} \vol_g \right)^{\frac{n-2}{n}} \leq A \int_M |\diff u|_g^2 \vol_g .
\end{equation*}
Notice further that such a bound is invariant under a rescaling of the metric $g$. Using this observation, we shall prove the following uniform Sobolev inequality:

\begin{lem}    \label{lem: uniform sobolev constant A1}
    There exists a constant $A_1 > 0$ such that for any $s \in (0,s_0]$,
    \begin{equation*}
        \left( \int_{\mZ_+} |u|^3 \omega_s^3 \right)^{\frac{2}{3}} \leq A_1 \int_{\mZ_+} |\diff u|^2 \omega_s^3, ~~~~ \forall u \in C^\infty_c(\mZ_+) .
    \end{equation*}
\end{lem}

\begin{proof}
    As before, we know that there is a constant $C > 1$ such that for any $s \in (0,s_0]$,
    \begin{equation*}
        C^{-1} \left. \omega_0 \right|_{\{R \geq \tfrac{1}{2}R_1\}} \leq \left. \omega_s \right|_{\{R \geq \tfrac{1}{2}R_1\}} \leq C \left. \omega_0 \right|_{\{R \geq \tfrac{1}{2}R_1\}}
    \end{equation*}
    which combined with \cite[Prop. 3.2]{hein2010gravitational} shows that there is a constant $A^\prime_1$ such that for any $s \in (0,s_0]$,
    \begin{equation*}
        \left( \int_{\mZ_+} |u|^3 \omega_s^3 \right)^{\frac{2}{3}} \leq A^\prime_1 \int_{\mZ_+} |\diff u|^2 \omega_s^3, ~~~ \forall u \in C^\infty_c(\mZ_+ \setminus \{R \leq \tfrac{1}{2} R_1\}).
    \end{equation*}
    On the other hand, in the region $\{ R \leq R_1\}$, $\omega_s = \omega_{\m{CdlO},s}$, and using Hein's Sobolev bound again we know that there exists a constant $A^{\prime\prime}_1 > 0$ such that for any $u \in C^\infty_c(\mZ_+)$, 
    \begin{equation*}
        \left( \int_{\mZ_+} |u|^3 \omega_{\m{CdlO},1}^3 \right)^{\frac{2}{3}} \leq A^{\prime\prime}_1 \int_{\mZ_+} |\diff u|^2 \omega_{\m{CdlO},1}^3 .
    \end{equation*}
    But since this Sobolev constant is invariant under a rescaling of the metric and $\omega_{\m{CdlO},s}$ is isometric to a rescaling of $\omega_{\m{CdlO},1}$ (under an appropriate diffeomorphism of $\mZ_+$), it follows that for any $s \in (0,s_0]$, 
    \begin{equation*}
        \left( \int_{\mZ_+} |u|^3 \omega_s^3 \right)^{\frac{2}{3}} \leq A^{\prime\prime}_1 \int_{\mZ_+} |\diff u|^2 \omega_s^3, ~~~ \forall u \in C^\infty_c(\mZ_+ \setminus \{R \geq  R_1\}).
    \end{equation*}
    Taking a standard cutoff function $\chi : [0,+\infty) \to [0,1]$ such that $\chi(t) = 1$ for $t \leq \frac{1}{2}$ and $\chi(t) = 0$ for $t \geq 1$ and writing any $u \in C^\infty_c(\mZ_+)$ as $u = \chi(\tfrac{R}{R_1})u + (1-\chi(\tfrac{R}{R_1}))u$, it is easy to deduce that there exists a constant $\tilde{A}_1 >0$, independent of $u$ and $s$, such that for any $s \in (0,s_0]$ we have
    \begin{equation*}
        \left( \int_{\mZ_+} |u|^3 \omega_s^3 \right)^{\frac{2}{3}} \leq \tilde{A}_1 \left( \int_{\mZ_+} |\diff u|^2 \omega_s^3   + \int_{\tfrac{1}{2}R_1 \leq R \leq R_1} |u|^2 \omega^2_s\right)
    \end{equation*}
    
    This is almost what we need, up to the second term on the right-hand side. Now the proof \cite[Prop. 3.2]{hein2010gravitational} shows that this term can be absorbed after increasing the constant $\tilde{A}_1$. The idea is to use the $\mathsf{SOB}(6)$ property to obtain certain weak Neumann--Sobolev inequalities on balls and annuli to show that the $L^2$-norm of the function $u$ minus its average is controlled by the $L^2$ norm of the gradient on slightly larger balls and annuli, and then to deal with the average part using the Cauchy--Schwarz inequality and a telescopic sum argument. In particular, there is no problem with the uniformity of the constants since we only need this argument in the region $\{R \geq \frac{1}{2}R_1\} \subset \mZ_+$ where the Kähler metrics $\omega_s$ are uniformly equivalent.
\end{proof}

%%%%% End of Section 4 %%%%%
%%%%%%%%%%%%%%%%%%%%%%%%%%%%

%% file: section5_singularlimit.tex
%%%%%%%%%%%%%%%%%%%%%%%%%%%%%%%%%%
%%%%% Beginning of Section 5 %%%%%

    \section{The singular limit}        \label{sec: The singular limit}

In this section, we use the modified ansatz $\omega_s$ constructed in the previous section to study the limit of $\omega_{\m{CY},s}$ as $s \to 0$. Note that for any $s > 0$, the Kähler metric $\omega_s$ only differs from the metric $\omega^\prime_s$ used in \S\ref{subsec: The new smooth CY metrics} to construct $\omega_{\m{CY},s}$ by $\iq\partial\overline{\partial}$ of a compactly supported function, which by the uniqueness part of Theorem \ref{thm: the new complete metrics} implies that there exists a unique $\Un(2)$-invariant function $u_s$ such that $\omega_{\m{CY},s} = \omega_s + \iq \partial\overline{\partial} u_s$ and $|\nabla^k_{\omega_s} u_s| = \mo(\rho^{-\delta} \mr^{-k})$, for any $\delta \in (0,1)$ and $k \geq 0$. We shall exploit the uniform $\mathsf{SOB}(6)$-property satisfied by the metrics $\{\omega_s\}_{s \in (0,s_0]}$ and the fact that $\omega_s$ is already Ricci-flat in a neighbourhood of $\mC_+$ to prove that it admits a singular limit $\omega_{\m{CY},0}$ when $s \to 0$ and study its properties.

This section is organised as follows. In \S\ref{subsec: L infinity bounds}, we gather the results of the previous sections to derive uniform $L^\infty$-bounds on the potentials $u_s$, using the method of viscosity solutions. In order to make the exposition more streamlined, we deferred the relevant results concerning Moser iteration for $\mathsf{SOB}(\beta)$-spaces according to Hein's thesis \cite{hein2010gravitational} to Appendix \ref{app: quantitative moser iteration}. In \S\ref{subsec: Second-order estimates}, we derive uniform $C^2$-estimates at the level of viscosity solutions away from the exceptional curve, using the Chern--Lu inequality and a well-designed (semi-positive) background metric $\omega_{\m{c}}$ on $\mZ_+$. These estimates are used to prove uniform \emph{weighted} $L^\infty$-bounds for the potentials $u_s$ in \S\ref{subsec: Uniform decay of the potential}. We then prove in \S\ref{subsec: Convergence away from curve} that the Calabi--Yau metrics $\omega_{\m{CY},s}$ converge smoothly away from $\mC_+$ to a singular Calabi--Yau metric $\omega_{\m{CY},0}$ on $\mZ^{\m{reg}}$ (Theorem \ref{thm: convergence to omegaCY0}). Finally, in \S\ref{subsec: GH convergence} we prove that $(\mZ_+,\omega_{\m{CY},s})$ converges to the metric completion of $(\mZ^{\m{reg}},\omega_{\m{CY},0})$ in the Gromov--Hausdorff sense (Theorem \ref{thm: metric completion}), and that $\omega_{\m{CY},0}$ has an isolated conical singularity modelled on the Stenzel metric at the ordinary double point (Proposition \ref{prop: tangent cone at o}). Combined together, these results prove Theorem \ref{thm: B}.

    \subsection{Uniform $L^\infty$-bounds}  \label{subsec: L infinity bounds}

In \S\ref{subsec: L infinity bounds}--\S\ref{subsec: Uniform decay of the potential}, we will prove uniform estimates up to order $2$ for the solutions $u_s$ of the Monge--Amp\`ere equation
\begin{equation}        \label{eq: MA equation for SOB}
    (\omega_s + \iq \partial \overline{\partial}u_s)^3 = e^{f_s} \omega_s^3 = \vol_+ .
\end{equation}
Recall that the Kähler metrics $\{\omega_s\}_{s \in (0,s_0]}$ satisfy the $\mathsf{SOB}(6)$-property, and that in the proof of Theorem \ref{thm: the new complete metrics} we obtained the solutions by applying \cite[Prop. 4.1]{hein2010gravitational} in Hein's thesis. Hein's method, which generalises the constructions of Tian--Yau \cite{tian1990completei,tian1991completeii}, is to show that one can solve the viscosity equation
\begin{equation}
    (\omega_s + \iq \partial\overline{\partial}u_{\varepsilon,s})^3 = e^{f_s + \varepsilon u_{\varepsilon,s}} \omega_s^3
\end{equation}
for all $\varepsilon \in (0,1)$, where the solutions $u_{\varepsilon,s}$ and all their derivatives are uniformly bounded\footnote{This holds because $f_s$ and all its derivatives are bounded. In fact Hein's method only requires H\"older bounds for the function $f_s$ and its derivatives of order $1$ and $2$, and yields solutions with weaker H\"older bounds on the derivatives up to order $4$, but we shall not need this more general version.} with respect to $\varepsilon$. Thus a solution $u_s$ of \eqref{eq: MA equation for SOB} can be obtained as a limit of $u_{\varepsilon,s}$ as $\varepsilon \to 0$ using the Arzel\`a--Ascoli theorem. Note that in the general setting of \cite[Prop. 4.1]{hein2010gravitational}, the solution might not be unique, but in our case uniqueness follows from the decay of $u_s$ and all its derivatives at infinity (cf. the discussion below Theorem \ref{thm: the new complete metrics}). 

With these preliminaries in mind, we first claim that the following uniform $L^\infty$-bounds can be extracted from the arguments of \cite[Ch. 3 \& 4]{hein2010gravitational} (see also Appendix \ref{app: quantitative moser iteration}):

\begin{prop}        \label{prop: L infinity bounds for us}
    There exists $Q_1 > 0$ such that for all $\varepsilon \in (0,1)$ and $s \in (0,s_0]$:
    \begin{equation*}
        \|u_{\varepsilon,s}\|_{L^\infty} \leq Q_1 ~~ \text{and thus} ~~ \|u_s\|_{L^\infty} \leq Q_1 .
    \end{equation*}
\end{prop}

\begin{proof}
    The proof of the above $L^\infty$-bounds is a classical consequence of Moser iteration. The uniformity follows from the observation that we have a uniform control over the three key ingredients of Moser iteration, with respect to the parameter $s$. First of all, by Lemma \ref{lem: uniform sobolev constant A1} there is a uniform Sobolev constant $A_1 > 0$ such that for all $s \in (0,s_0]$,
    \begin{equation*}
        \|u\|^2_{L^{\frac{3}{2}}} \leq A_1 \| \diff u \|^2_{L^2}, ~~~~ \forall u \in C^\infty_c(\mZ_+) .
    \end{equation*}
    Second, if we fix a parameter $\mu \in (2,3)$, we deduce from Lemma \ref{lem: functions fs} that
    \begin{equation}
        A_2 \coloneqq \sup_{s \in (0,s_0]} \sup_{\mZ_+} \{ \rho^\mu |1-e^{f_s}|\},
    \end{equation}
    is finite. Finally, let us pick a weight function $\tilde{\rho} \colon \mZ_+ \to [1,+\infty)$ such that $\left. \tilde{\rho} \right|_{\{R \leq R_1\}} \equiv 1$ and $\left.\rho \right|_{\{R \geq 2 R_1\}} = \rho$. Then for any $q > 2m$, there exists a constant $A_3 > 0$ independent of $s \in (0,s_0]$ such that
    \begin{equation*}
        \int_{\mZ_+} \tilde{\rho}^{-q} \omega_s^3 \leq A_3 .
    \end{equation*}
    Indeed, it follows from Lemma \ref{lem: functions fs} that there exists a constant $C > 1$ independent of $s$ such that $C^{-1} \vol_+ \leq \omega_s^3 \leq C \vol_+$, and on the other hand for any $q > 2m$ the function $\tilde{\rho}^{-q}$ is integrable with respect to the measure $\vol_+$ since $\tilde{\rho}$ is comparable to the distance to $o_+ \in \mZ_+$ outside a compact subset and $(\mZ_+,\omega_s)$ satisfies the $\mathsf{SOB}(6)$ property.

    While it should be clear to the reader familiar with Chapters 3 and 4 in Hein's thesis \cite{hein2010gravitational} that the uniform control over $A_1,A_2,A_3 > 0$ is enough to deduce Proposition \ref{prop: L infinity bounds for us} from the arguments therein, we have included for completeness a proof of a more general statement for solutions of the Monge--Amp\`ere equation on Kähler $m$-folds satisfying the $\mathsf{SOB}(2m)$-property in Appendix \ref{app: quantitative moser iteration} (Proposition \ref{prop: Quantitative L infinity bound}). 
\end{proof}

    \subsection{Second-order estimates for viscosity solutions}     \label{subsec: Second-order estimates}

Let us define the Kähler metrics $\omega_{\varepsilon,s} \coloneqq \omega_s + \iq \partial\overline{\partial}u_{\varepsilon,s}$, for $\varepsilon \in (0,1)$ and $s \in (0,s_0]$, where as before we denote by $u_{\varepsilon,s}$ the solution of the viscosity equation $(\omega_s + \iq \partial \overline{\partial} u_{\varepsilon,s})^3 = e^{f_s + u_{\varepsilon,s}}\omega_s^3$. In this section we shall prove:

\begin{prop}        \label{prop: C2 bounds for viscosity solutions}
    For any $R_0 > 0$, there exists a constant $Q_{2,R_0} > 1$ such that for all $\varepsilon \in (0,1)$ and $s \in (0,s_0]$,
    \begin{equation*}
        Q_{2,R_0}^{-1} \left. \omega_s \right|_{\{ R \geq R_0\}} \leq \left. \omega_{\varepsilon,s} \right|_{\{ R \geq R_0\}} \leq Q_{2,R_0} \left. \omega_s \right|_{\{R \geq R_0\}}  .
    \end{equation*}
    In particular this comparison also holds for $\varepsilon = 0$, i.e.
    \begin{equation*}
        Q_{2,R_0}^{-1} \left. \omega_s \right|_{\{ R \geq R_0\}} \leq \left. \omega_{\m{CY},s} \right|_{\{ R \geq R_0\}} \leq Q_{2,R_0} \left. \omega_s \right|_{\{R \geq R_0\}}  .
    \end{equation*}
\end{prop}

To prove the proposition, we must deal with the issue that the bisectional curvature of $\omega_s$ is not uniformly bounded above with respect to the parameter $s \in (0,s_0]$ near the exceptional curve $\mC_+$. To fix this, we will use the following result:

\begin{lem}     \label{lem: comparison form}
    There exist $R^\prime_1,R^\prime_2 > 0$ and a smooth $(1,1)$-form $\omega_{\m{c}}$ on $\mZ_+$ such that:
    \begin{enumerate}[(i)]
        \item $\omega_{\m{c}}$ is Kähler on $\mZ_+ \setminus \mC_+$.
        \item $\omega_c = \frac{1}{2}\iq \partial\overline{\partial}(R^2)$ in the region $\{R \leq R^\prime_1\}$.
        \item $\omega_{\m{c}} = \omega_0$ in the region $\{R \geq R^\prime_2\}$.
    \end{enumerate}
\end{lem}

\begin{proof}
    This is exactly the same convexity argument as in the proof of Proposition \ref{prop: interpolation metric}. Namely, we first pick an increasing, convex function $F \colon [0,+\infty) \to \R$ such that $F(t) = \frac{t^2}{2} + c$ in a neighbourhood of $0$ and $F(t) = t$ in a neighbourhood of $+\infty$, so that $F(R+\frac{\eta}{2})$ is strictly plurisubharmonic on $\mZ_+ \setminus \mC_+$. In particular, it coincides with $\omega^{\m{srf}}_0$ at infinity, and using the potential $\varphi_0$ and an appropriate cutoff we can modify it to coincide with $\omega_0$ in the region where $R$ is large enough. On the other hand, using $\eta = \mo(R^2)$ as $R \to 0$, we can use another cutoff to find a plurisubharmonic function interpolating between $R^2$ and $(R+\frac{\eta}{2})^2$ in the region where $R$ is small, so as to obtain a Kähler form on $\mZ_+ \setminus \mC_+$. Finally, since the function $R^2$ is smooth on $\mZ_+$, $\omega_{\m{c}}$ extends smoothly to a semi-positive $(1,1)$-form.
\end{proof}

\begin{lem}     \label{lem: lower bound for comparison form}
    Without loss of generality, assume that $R^\prime_1 \leq R_1$. Then there exists a constant $c_0 > 0$ such that for any $s \in (0,s_0]$,
    \begin{equation*}
        \omega_{s} \geq c_0 \omega_{\m{c}} ~~~~ \text{in the region}~ \{R \leq R^\prime_1\} .
    \end{equation*}
\end{lem}

\begin{proof}
    Under the identification $\mZ_+ \setminus \mC_+ \cong \mZ^{\m{reg}} \hookrightarrow \mathfrak{gl}(2,\C) \cong \C^4$, the metric $\omega_{\m{c}}$ coincides with the restriction of the Euclidean metric in the region $\{0 < R \leq R^\prime_1\}$, and therefore it has non-positive bisectional curvature therein. On other hand, $\omega_s$ is Ricci-flat in this region, and therefore the Chern--Lu inequality (cf. \cite[Prop. 7.2]{rubinstein2014smooth} for instance) yields 
    \begin{equation*}
        \Delta_{\omega_s} \log( \m{tr}_{\omega_s} (\omega_{\m{c}}) ) \geq 0, ~~~~ \text{if} ~R \leq R_1,
    \end{equation*}
    where in our convention $\Delta_\omega u = \m{tr}_{\omega}(\iq \partial \overline{\partial} u)$ has positive-definite principal symbol. By the maximum principle,
    \begin{equation*}
        \sup_{\{ R \leq R^{\prime}_1\}} \m{tr}_{\omega_s}(\omega_{\m{c}}) = \sup_{\{R = R^{\prime}_1\}} \m{tr}_{\omega_s}(\omega_{\m{c}})
    \end{equation*}
    for any $s \in (0,s_0]$. On the other hand, the Kähler forms $\omega_s$ are all uniformly equivalent away from $\mC_+$, and therefore $\m{tr}_{\omega_s}(\omega_{\m{c}})$ has a uniform upper bound along $\{R = R^\prime_1\}$, which yields the desired result.
\end{proof}

With this in hand, we shall prove Proposition \ref{prop: C2 bounds for viscosity solutions} by the standard argument which relies on the Chern--Lu inequality and the maximum principle. Note that because all the derivatives of $u_{\varepsilon,s}$ are bounded, the Kähler metrics $\omega_{\varepsilon,s}$ are complete and the Riemann curvature tensor $\m{Rm}_{\omega_{\varepsilon,s}}$ and all its derivatives are bounded (albeit not uniformly so in $\varepsilon,s$). Consequently, we may apply Yau's maximum principle \cite[Th. 1]{yau1975harmonic}, which implies (cf. \cite[Lem. 4.4]{hein2010gravitational}) that for any function $H \in C^{2}_{\m{loc}}(\mZ_+)$ such that $|H| + |\diff H|_{\omega_{\varepsilon,s}} + |\nabla^{\omega_{\varepsilon,s}}\diff H|_{\omega_{\varepsilon,s}} \lesssim 1$, either $H$ attains a maximum, or there is a sequence of points $\{x_j\}_{j \in \N}$ such that 
\begin{equation*}
    \dist_{\omega_{\varepsilon,s}}(o_+,x_j) \to \infty, ~~ H(x_j) \to \sup H, ~~ |\diff H(x_j)|_{\omega_{\varepsilon,s}} \to 0 ~~ \text{and} ~~\limsup \Delta_{\omega_{\varepsilon,s}} H(x_j) \leq 0. 
\end{equation*}
In particular, for any such function there is a point $x \in \mZ_+$ such that $H(x) \geq \sup H - 1$ and $\Delta_{\omega_{\varepsilon,s}} H(x) \leq 1$. Using this observation, we first prove:

\begin{lem}
    There exists a constant $C > 0$ such that 
    \begin{equation*}
        \m{tr}_{\omega_{\varepsilon,s}}(\omega_{\m{c}}) \leq C 
    \end{equation*}
    for all $\varepsilon \in (0,1)$ and $s_0 \in (0,s_0]$.
\end{lem}

\begin{proof}
    Remark that
    \begin{equation*}
        (\omega_s + \iq \partial \overline{\partial}u_{\varepsilon,s})^3 = e^{f_s+\varepsilon u_{\varepsilon,s}}\omega_s^3 = e^{\varepsilon u_{\varepsilon,s}} \vol_+
    \end{equation*}
    and therefore
    \begin{equation*}
        \m{Ric}_{\omega_{\varepsilon,s}} = - \iq \varepsilon \partial \overline{\partial} u_\varepsilon = \varepsilon \omega_s - \varepsilon \omega_{\varepsilon,s} \geq - \omega_{\varepsilon,s}
    \end{equation*}
    since $\varepsilon \in (0,1)$. On the other hand, $\omega_{\m{c}}$ has non-positive bisectional curvature in the region $\{ R \leq R^\prime_1\}$. In addition, the bisectional curvature of $\omega_0$ is bounded away from the exceptional curve and hence so is that of $\omega_{\m{c}}$. Using the Chern--Lu inequality (cf. \cite[Prop. 7.2]{rubinstein2014smooth} for instance), it follows that there exists a constant $C > 0$ such that
    \begin{align*}
        \Delta_{\omega_{\varepsilon,s}} \log \m{tr}_{\omega_{\varepsilon,s}} (\omega_{\m{c}}) \geq \left\{\begin{array}{ll}
                 -1 & \text{if } R \leq R^\prime_1, \\
                 -1 - C \m{tr}_{\omega_{\varepsilon,s}}(\omega_c) & \text{if } R \geq R^\prime_1.
        \end{array}\right.
    \end{align*}
    Since the metrics $\omega_0$, $\omega_{\m{c}}$ and $\omega_s$ are all uniformly comparable in the region $\{R \geq R_1\}$, we deduce (after possibly enlarging the constant $C$) that
    \begin{equation*}
        \Delta_{\omega_{\varepsilon,s}} \log \m{tr}_{\omega_{\varepsilon,s}} (\omega_{\m{c}}) \geq -1 - C \m{tr}_{\omega_{\varepsilon,s}}(\omega_s)
    \end{equation*}
    for all $\varepsilon \in (0,1)$ and $s \in (0,s_0]$.

    Let us now introduce the function
    \begin{equation}
        H_{\varepsilon,s} \coloneqq \log \m{tr}_{\omega_{\varepsilon,s}} (\omega_{\m{c}}) -  (C+1) u_{\varepsilon,s} .
    \end{equation}
    Since $\m{tr}_{\omega_{\varepsilon,s}}(\omega_s) = 3 - \Delta_{\varepsilon,s} u_{\varepsilon,s}$, the previous inequality yields
    \begin{equation*}
        \Delta_{\omega_{\varepsilon,s}} H_{\varepsilon,s}  \geq \m{tr}_{\omega_{\varepsilon,s}}(\omega_s) - 3C - 4 .
    \end{equation*}
    Since all the derivatives of $u_{\varepsilon,s}$ are bounded and $\omega_{\m{c}}$ is a smooth $(1,1)$-form on $\mZ_+$ coinciding with $\omega_0$ at infinity, we see using Proposition \ref{prop: covariant derivatives at infinity} that $H_{\varepsilon,s}$ is smooth, bounded and has all derivatives bounded. Using Yau's maximum principle, we find $x_{\varepsilon,s} \in \mZ_+$ such that $\sup H_{\varepsilon,s} \leq H_{\varepsilon,s}(x_{\varepsilon,s}) + 1$ and $\Delta_{\omega_{\varepsilon,s}} H_{\varepsilon,s}(x_{\varepsilon,s}) \leq 1$. Thus $\m{tr}_{\omega_{\varepsilon,s}} (\omega_s)_{x_{\varepsilon,s}} \leq 3C + 5$, and by Lemma \ref{lem: lower bound for comparison form} it follows that 
    \begin{equation*}
        \m{tr}_{\omega_{\varepsilon,s}} (\omega_{\m{c}})_{x_{\varepsilon,s}} \leq c_0^{-1} \m{tr}_{\omega_{\varepsilon,s}} (\omega_s)_{x_{\varepsilon,s}} \leq C_0 \coloneqq c_0^{-1}(3C+5) .
    \end{equation*}
    This inequality together with Proposition \ref{prop: L infinity bounds for us} yields a uniform upper bound on $H_{\varepsilon,s}(x_{\varepsilon,s})$ and hence on $\sup H_{\varepsilon,s}$, which implies the lemma.
\end{proof}

\begin{proof}[Proof of Proposition \ref{prop: C2 bounds for viscosity solutions}]
    By the previous lemma, $\omega_{\m{c}} \leq C \omega_{\varepsilon,s}$ for some constant $C > 0$ independent of $\varepsilon \in (0,1)$ and $s \in (0,s_0]$. To obtain a bound in the other direction, notice that if we pick a volume form $\vol_{\mathbb{S}^2 \times S_3}$ on $\mathbb{S}^2 \times \mathbb{S}^3$ and identify $\mZ_+ \setminus \mC_+ \cong (0,1)_R \times \mathbb{S}^2 \times \mathbb{S}^3$, then the volume form of $\omega_c$ satisfies 
    \begin{equation*}
        \omega_c \asymp R^5 \diff R \wedge \vol_{\mathbb{S}^2 \times \mathbb{S}^3} 
    \end{equation*}
    in the region $\{R \leq R^\prime_1\}$ since $\omega_{\m{c}}$ is the restriction of the Euclidean metric therein. On the other hand, $\omega_{\varepsilon,s}^3 = e^{\varepsilon u_{\varepsilon,s}} \vol_+$ is uniformly comparable to $\vol_+ = \omega_{\m{Stz}}^3$ by Proposition \ref{prop: L infinity bounds for us}, and since $\omega_{\m{Stz}}$ is conical with radius function $\varrho \asymp R^{2/3}$ we see that 
    \begin{equation*}
        \vol_+ \asymp (R^{2/3})^5 \diff (R^{2/3}) \wedge \asymp R^3 \diff R \wedge \vol_{\mathbb{S}^2 \times \mathbb{S}^3} \asymp R^{-2} \omega_{\m{c}}^3
    \end{equation*} 
    in $\{ R \leq R^\prime_1\}$. Thus if we diagonalise $\omega_{\m{c}}$ with respect to $\omega_{\varepsilon,s}$ in this region, each eigenvalue is bounded above by $C$ and below by $C^{-1} R^2$ in $\{ R \leq R^\prime_1\}$, for some $C > 1$ independent of $\varepsilon \in (0,1)$ and $s \in (0,s_0]$. On the other hand, the volume forms $\omega_{\m{c}}^3$, $\omega_s^3$ and $\omega_{\varepsilon,s}^3$ are uniformly comparable in the region $\{R \geq R^\prime_1\}$, and thus we deduce that there exists a constant $C > 1$ such that
    \begin{equation}        \label{eq: comparison with omegac}
        C^{-1} \omega_{\m{c}} \leq \omega_{\varepsilon,s} \leq C \max\{1, R^{-2}\} \omega_{\m{c}}
    \end{equation}
    for all $\varepsilon \in (0,1)$ and $s \in (0,s_0]$.
    
    Since the Kähler forms $\omega_s,\omega_{\m{c}}$ are uniformly equivalent away from the exceptional curve, it follows that for any $R_0 > 0$, there exists a constant $C_{R_0} > 1$ such that for any $s \in (0,s_0]$, 
    \begin{equation*}
        C_{R_0}^{-1} \left. \omega_s \right|_{\{R \geq R_0\}} \leq \left.  \omega_{\m{c}} \right|_{\{ R \geq R_0\}} \leq C_{R_0} \left. \omega_s \right|_{\{R \geq R_0\}} .
    \end{equation*}
    Together with \eqref{eq: comparison with omegac} this completes the proof of the proposition.
\end{proof}

    \subsection{Uniform decay of the potential}     \label{subsec: Uniform decay of the potential}

Thanks to the $C^2$-estimates proved in the previous part, we know that there exists a constant $A_4 > 0$ such that 
\begin{equation*}
    |\partial\overline{\partial} u_{s,\varepsilon}| \leq A_4, ~~~ \text{in} ~ \{ R \geq R_1\},
\end{equation*}
for all $\varepsilon \in (0,1)$ and $s \in (0,s_0]$. Moreover, we picked in \S\ref{subsec: L infinity bounds} a weight function $\tilde{\rho} \geq 1$ equal to $1$ in the region $\{R \leq R_1\}$ and to $\rho$ in the region $\{R \geq 2 R_1\}$, and since $\omega_s$ is uniformly equivalent to $\omega_s^{\m{srf}}$ in the region $\{R \geq R_1\}$ it follows from Lemma \ref{lem: Derivative of distance function} that there exists a constant $A_5 > 0$ such that for all $s \in (0,s_0]$,
\begin{equation*}
    |\diff \tilde{\rho}|_{\omega_s} + \tilde{\rho} |\diff\diff^c\tilde{\rho}|_{\omega_s} \leq A_5, ~~~ \text{on} ~\mZ_+ .
\end{equation*}
Using the uniform constants $A_1,A_2,A_3$ introduced in \S\ref{subsec: L infinity bounds}, the above constants $A_4,A_5$ and weighted Moser iteration we can deduce:

\begin{prop}    \label{prop: Uniform decay of potential}
    For any $\delta \in (0,1)$, there exists a constant $Q_3 > 0$ such that for all $s \in (0,s_0]$, $\|\tilde{\rho}^\delta u_s\|_{L^\infty} \leq Q_3$.
\end{prop}

As for Proposition \ref{prop: L infinity bounds for us}, it should be clear that this can be deduced from the arguments of \cite[Ch. 3 \& 4]{hein2010gravitational}, but for the convenience of the reader we have included the proof of a more general statement for $\mathsf{SOB}(2m)$ $m$-folds in Appendix \ref{app: quantitative moser iteration} (Proposition \ref{prop: weighted L infinity bound}).

    \subsection{Convergence away from the exceptional curve}        \label{subsec: Convergence away from curve}

Using Evans--Krylov estimates, which are local in nature, we are now in position to prove:

\begin{prop}    \label{prop: uniform higher-order estimates}
    For any $R_0 > 0$ and $k \geq 1$, there exists a constant $C_{k,R_0}$ such that
    \begin{equation*}
        \|u_s \|_{C^k(\{R \geq R_0\})} \leq C_{k,R_0}
    \end{equation*}
    for any $s \in (0,s_0]$.
\end{prop}

\begin{proof}
    This follows from the uniform $C^2$-bound of Proposition \ref{prop: C2 bounds for viscosity solutions} and the local estimates of \cite{sherman2013local} for instance, cf. also \cite[Prop. 4.7]{collins2022degeneration}.
\end{proof}

We can in fact use the same arguments as in the proof of \cite[Th. 8.6.11]{joyce2000compact} to show that we even have uniform \emph{weighted} estimates at infinity:

\begin{prop}    \label{prop: uniform weighted estimates}
    For any $\delta \in (0,1)$ and $k \geq 0$, there exists a constant $C_k > 0$ such that for all $s \in (0,s_0]$,
    \begin{equation*}
        \|\nabla_{\omega_s}^k u_s\|_{C^0(\{ R \geq R_1\})} \leq C_{k,\delta} \rho^{-\delta} \mr^{-k} .
    \end{equation*}
\end{prop}

\begin{proof}
    The basic idea is to rewrite the Monge--Amp\`ere equation as the linear elliptic equation $P_s u_s = (e^{f_s}-1)$ where 
    \begin{equation*}
        (P_s u_s) \omega_s^3= \iq \partial\overline{\partial} u_s \wedge (\omega_s^2 + 2 \omega_s \wedge (\omega_s + \iq\partial\overline{\partial}u_s) + (\omega_s + \iq \partial\overline{\partial} u_s)^2) .
    \end{equation*}
    Using the uniform estimates of Proposition \ref{prop: uniform higher-order estimates}, the weighted $L^\infty$-estimate of Proposition \ref{prop: Uniform decay of potential} and the bounds on $(1-e^{f_s})$ from Lemma \ref{lem: functions fs}, we can use the same scaling argument as in the proof of Proposition \ref{prop: covariant derivatives at infinity} to rewrite the equation $P_s u_s = (e^{f_s}-1)$ in the domain $\{\frac{1}{2} \lambda^4 \leq \upsilon, \eta \leq \lambda\}$ as a linear elliptic equation on the fixed domain $\{\frac{1}{2} \leq \upsilon, \eta \leq 1\}$, where all the coefficients (and their derivatives) of the rescaled linear elliptic operator are uniformly under control with respect to the scaling parameter $\lambda \asymp \mr$. The result then follows from the classical interior estimates for linear elliptic operators (see the proof of \cite[Th. 8.6.11]{joyce2000compact} for more details).
\end{proof}

This allows us to prove:

\begin{thm}        \label{thm: convergence to omegaCY0}
    Let $\delta \in (0,1)$. There exists a unique plurisubharmonic function
    \begin{equation*}
        \varphi_{\m{CY},0} \in \mathsf{PSH}(\mZ_+) \cap L^\infty_{\m{loc}}(\mZ_+) \cap C^\infty_{\m{loc}}(\mZ_+ \setminus \mC_+)
    \end{equation*}
    such that $u_0 \coloneqq \varphi_{\m{CY},0} - \phi_0$ satisfies
    \begin{equation*}
        |\nabla^k_{\omega_0} u_0 | = \mo(\rho^{-\delta} \mr^{-k}), ~~~~ \forall k \geq 0,
    \end{equation*}
    and $(\iq\partial \overline{\partial} \varphi_{\m{CY},0})^3 = \vol_+$ in the Bedford--Taylor sense. Moreover, $\varphi_{\m{CY},0}$ is $\Un(2)$-invariant, and $\{\omega_{\m{CY},s}\}_{s \in (0,s_0]}$ converges to the (possibly incomplete) Calabi--Yau metric $\omega_{\m{CY},0} \coloneqq \iq \partial\overline{\partial}\varphi_{\m{CY},0}$ in the $C^\infty_{\m{loc}}(\mZ_+ \setminus \mC_+)$-sense.
\end{thm}

\begin{proof}
    Given Proposition \ref{prop: uniform higher-order estimates}, the Arzel\`a--Ascoli theorem immediately implies that for any sequence $s_j \to 0$, there exists a subsequence $\omega_{\m{CY},s^\prime_{j}}$ converging to a Kähler Ricci-flat metric $\omega_{\m{CY},0}$ in $C^\infty_{\m{loc}}(\mZ_+ \setminus \mC_+)$. Moreover, we can write $\omega_{\m{CY},0} = \iq \partial\overline{\partial} (\phi_0 + u_0)$ where $\phi_0 \in C^0_{\m{loc}}(\mZ_+) \cap C^\infty_{\m{loc}}(\mZ_+ \setminus \mC_+)$ is the potential introduced in Proposition \ref{prop: interpolation metric} and $u_0 \in L^\infty_{\m{loc}}(\mZ_+) \cap C^\infty_{\m{loc}}(\mZ_+ \setminus \mC_+)$ is the limit of the potentials $u_{s^\prime_{j}}$ in the $C^\infty_{\m{loc}}(\mZ_+ \setminus \mC_+)$-sense. Note that by Proposition \ref{prop: uniform weighted estimates}, $u_0$ satisfies the estimates
    \begin{equation*}
        |\nabla^k_{\omega_0} u_0 | = \mo(\rho^{-\delta} \mr^{-k}), ~~~~ \forall k \geq 0.
    \end{equation*}
    Since $\phi_0 + u_0 \in L^\infty_{\m{loc}}(\mZ_+)$ and is continuous (even smooth) on $\mZ_+ \setminus \mC_+$, it admits a unique plurisubharmonic extension $\varphi_{\m{CY},0} \in \mathsf{PSH}(\mZ_+)$ \cite[Prop. 1.7]{demailly1985mesures}. Moreover, since the equation $(\iq \partial\overline{\partial} \varphi_{\m{CY},0})^3 = \vol_+$ holds on $\mZ_+ \setminus \mC_+$, the Chern--Levine--Nirenberg inequalities \cite{chern1970intrinsic} imply that this equation also holds in the Bedford--Taylor sense \cite{bedford1976dirichlet} on $\mZ_+$.
    
    Thus in order to prove the theorem, it only remains to prove the uniqueness of such a plurisubharmonic function, which will automatically imply that all sequences $\omega_{\m{CY},s_j}$ converging in the $C^\infty_{\m{loc}}(\mZ_+ \setminus \mC_+)$-sense must have the same limit, and the $\Un(2)$-invariance of $\varphi_{\m{CY},0}$. Now since we are assuming enough decay on $u_0$ and its derivatives, the proof of uniqueness is the same as in \cite[Prop. 4.17]{collins2022degeneration}, cf. also \cite[Th. 7.4]{conlon2023degenerations}.
\end{proof}

    \subsection{Gromov--Hausdorff convergence and singular behaviour}   \label{subsec: GH convergence}

In this section, we study the convergence of $\{\omega_s\}_{s \in (0,s_0]}$ from the metric point of view. The following local volume lower bound will play a key role:

\begin{lem}
    Let $\Vol(B_{\R^6}(r))$ be the volume of the Euclidean ball with radius $r$. Then for any $x \in \mZ_+$, $r > 0$ and $s \in (0,s_0]$, 
    \begin{equation*}
        \frac{\Vol(B_{\omega_{\m{CY,s}}}(x,r))}{\Vol(B_{\R^6}(r))} \geq \frac{1}{2} \cdot
    \end{equation*}
\end{lem}

\begin{proof}
    This follows from Bishop--Gromov volume monotonicity since the tangent cone at infinity of $\omega_s$ is $\C \times (\C^2 / \Z_2)$, which has volume ratio exactly $\tfrac{1}{2}$.
\end{proof}

\begin{cor}     \label{cor: diameter estimate}
    There exists $D > 0$ such that $\m{diam}_{\omega_s}(\{R \leq R_1\}) \leq D$ for all $s \in (0,s_0]$.
\end{cor}

\begin{proof}
    Let $x \in \{ R < R_1\}$, and for any $s \in (0,s_0]$ define
    \begin{equation*}
        D_{x,s} \coloneqq \sup \{ r > 0 \mid B_{\omega_{\m{CY},s}}(x,r) \subseteq \{R \leq R_1\} \} .
    \end{equation*}
    In particular, for any $s \in (0,s_0]$, there exists a point $y_{x,s} \in \mZ_+$ such that $R(y_{x,s}) = R_1$ and $\dist_{\omega_{\m{CY},s}}(x,y_{x,s}) = D_s$. Since $\omega_{\m{CY},s}$ converges smoothly in the region $\{R_1 \leq R \leq R_1 + 1\}$, it follows that there exists a constant $C > 0$ independent of $s$ such that $\dist_{\omega_{\m{CY},s}}(y,y^\prime) \leq C$ for any $y,y^\prime \in \{R_1 \leq R \leq R_1 + 1\}$. By the triangular inequality, it follows that
    \begin{align*}
        \dist_{\omega_{\m{CY},s}}(o_+,x) & \leq \dist_{\omega_{\m{CY},s}}(o_+,y_{o_+,s}) + \dist_{\omega_{\m{CY},s}}(y_{o_+,s},y_{x,s}) + \dist_{\omega_{\m{CY},s}}(y_{x,s},x) \\
            & \leq D_{o_+,s} + D_{x,s} + C .
    \end{align*}
    On the other hand, by the previous lemma
    \begin{equation*}
        \frac{\Vol_{\omega_s}(\{R \leq R_1\})}{\Vol(B_{\R^6}(1)) D_{x,s}^6} \geq \frac{\m{Vol}(B_{\omega_{\m{CY},s}}(x,D_s))}{\Vol(B_{\R^6}(1)) D_{x,s}^6} \geq \frac{1}{2}
    \end{equation*}
    which yields uniform bounds for $D_{o_+,s}$ and $D_{x,s}$ for any $x \in \{R \leq R_1\}$ since the volume form of $\omega_s$ is independent of $s$ in this region (this is just a constant multiple of $\vol_+$).
\end{proof}

We also record the following estimate for the diameter of the exceptional curve:

\begin{lem}     \label{lem: diameter of exceptional curve}
    There exists a constant $C > 0$ such that $\m{diam}_{\omega_{\m{CY},s}}(\mC_+) = C s^{1/2}$.
\end{lem}

\begin{proof}
    The restriction of $\omega_{\m{CY},s}$ is a $\Un(2)$-invariant Kähler metric on $\mC_+ \cong \mathbb{P}^1$, and hence it is a constant multiple of the Fubini-Study metric. On the other hand, the area of $\mC_+$ proportional to $s$ since it is determined by the Kähler class, thus the result follows.
\end{proof}

Once these estimates are in place, we may deduce the following result from the arguments of \cite{song2015conjecture}:

\begin{thm}     \label{thm: metric completion}
    The metric completion of $(\mZ_+ \setminus \mC_+ \cong \mZ^{\m{reg}},\omega_{\m{CY},0})$ is a length space homeomorphic to $\mZ$, whose distance will be denoted by $d_0$. Moreover, $\{(\mZ_+,\omega_{\m{CY},s},o_+)\}_{s \in (0,s_0]}$ converges to $(\mZ,d_0,o)$ in the pointed Gromov--Hausdorff sense as $s \to 0$.
\end{thm}

\begin{proof}
    We shall indicate how to adapt the arguments of \cite{song2015conjecture} (and references therein) to our specific situation to deduce the theorem:
    \begin{itemize}
        \item Thanks to Corollary \ref{cor: diameter estimate}, Cheeger--Colding theory \cite{cheeger1997structure} already implies that for any sequence $s_k \to 0$, $(\mZ_+,\omega_{\m{CY},s_k},o)$ sub-converges in the Gromov--Hausdorff sense. In fact, by standard arguments, one can embed $(\mZ^{\m{reg}},\omega_{\m{CY},0})$ isometrically as an open subset of any such limit, which must have full measure and therefore be dense. This implies genuine Gromov--Hausdorff convergence to the metric completion of $(\mZ^{\m{reg}},\omega_{\m{CY},0})$, but this argument fails to show that the limit is homeomorphic to $\mZ$ because we cannot prevent the following bad situation from occurring: there could exist a sequence of points $x_k \in \mZ_+$, such that $x_k \to o_+$ in the topology of $\mZ_+$, and such that $\dist_{\omega_{\m{CY},s_k}}(x_k,o_+) \geq \delta > 0$ as $s_k \to 0$.
        
        \item Thus the crucial part is to improve the $C^2$-estimates near the exceptional curve $\mC_+$. Indeed, if we examine the proof of Proposition \ref{prop: C2 bounds for viscosity solutions}, we showed that $\omega_{\m{CY},s} \leq C R^{-2} \omega_{\m{c}}$ near $\mC_+$, where $\omega_{\m{c}}$ coincides with the restriction of the Euclidean metric. From this inequality we can only deduce that $|\partial_R|_{\omega_{\m{CY},s}} \leq C^{1/2} R^{-1}$, which just about fails to prove that the velocity of radial paths is integrable. In \cite[Prop. 4.2]{song2015conjecture}, Song used a meromorphic slicing argument to prove that in a neighbourhood of the exceptional curve the following estimates hold uniformly with respect to $s \in (0,s_0]$:
        \begin{equation*}
            C^{-1} \omega_{\m{c}} \leq \omega_{\m{CY},s} \leq C R^{-2} \omega_{\m{c}} ~~ \text{and} ~~ |\partial_{R}|_{\omega_{\m{CY},s}} \leq C R^{-1/2},
        \end{equation*}
        where $C > 0$ is a constant independent of $s$ and $\partial_{R}$ the radial vector field\footnote{The restriction of the Euclidean metric is denoted by $\omega_{\hat{E}}$ and its potential by $e^{\rho}$ in Song's notations, and therefore our $R$ is proportional to his $e^{\frac{\rho}{2}}$. Moreover, the vector field denoted by $W$ in Song's paper is (up to some scaling conventions) the $(1,0)$-part of $\partial_R$.}. Note that Song's result is stated in the compact case but his estimates are purely local and therefore apply to our situation. Crucially, the estimate on the norm of $\partial_R$ now implies that the $\omega_{\m{CY},s}$-velocity of radial paths coming out of $\mC_+$ is uniformly integrable with respect to $s$, and together with Lemma \ref{lem: diameter of exceptional curve} (which plays a similar role as Corollary 4.3 in Song's paper, but with a stronger statement) this shows that the diameters $\m{diam}_{\omega_{\m{CY},s}}(\{R \leq R_0\})$ converge to zero as $R_0 \to 0$ uniformly with respect to $s$. This rules out the bad situation described above.

        \item Using the arguments of \cite[Sec. 3]{song2013contracting}, one can then show that $\mZ$ can be endowed with a distance $d_0$, such that $d_0(x,y)$ is the infimum of the $\omega_{\m{CY},0}$-length over piecewise smooth paths connecting $x$ and $y$ in $\mZ$ (cf. Sec. 5 in \cite{song2015conjecture}).
    \end{itemize}
    It is then immediate to conclude that $(\mZ_+,\omega_{\m{CY},s},o_+)$ converges to $(\mZ,d_0,o)$ in the pointed Gromov--Hausdorff sense as $s \to 0$, since the blow-down map $\mZ_+ \to \mZ$ provides a family of quasi-isometries.
\end{proof}

While the proof of Theorem \ref{thm: metric completion} above does not give any information as to the behaviour of $\omega_{\m{CY},0}$ near $o$ besides the fact that the completion of $(\mZ^{\m{reg}},\omega_{\m{CY},0})$ is homeomorphic to $\mZ$ itself, we can give a much more precise description using the recent work of Zhang \cite{zhang2024polynomial}, which generalises some aspects of Hein--Sun's construction of compact Calabi--Yau manifolds with conical singularities \cite{hein2017calabi} to certain singular Kähler--Einstein metrics that are not necessarily obtained by \emph{polarised} non-collapsed degenerations of Kähler--Einstein manifolds (when the Kähler form is the curvature of a line bundle, which is not the case of our degeneration since the Kähler class $[\omega_{\m{CY},s}] \in \m{H}^2(\mZ_+)$ goes to $0$ as $s \to 0$).

\begin{prop}        \label{prop: tangent cone at o}
    $(\mZ,d_0)$ has a unique tangent cone at $o$, isometric to the Stenzel cone $(\mZ,\omega_{\m{Stz}})$. Moreover, $\omega_{\m{CY},0}$ is conically singular in the following sense: there exist neighbourhoods $U,V$ of $o$ in $\mZ$, $\nu > 0$ and a biholomorphism $\Phi \colon U \to V$ such that
    \begin{equation*}
        |\nabla^k_{\omega_{\m{Stz}}}(\Phi^* \omega_{\m{CY},0} - \omega_{\m{Stz}})| = \mo(\varrho^{\nu-k})
    \end{equation*}
    as $\varrho \to 0$ for all $k \geq 0$ (where as before $\varrho$ is the radius function of the Stenzel cone).
\end{prop}

\begin{proof}
    This follows from the results of \cite{zhang2024polynomial}. Indeed, by Theorem \ref{thm: convergence to omegaCY0} the singular Calabi--Yau metric $\omega_{\m{CY},0}$ admits a potential $\varphi_{\m{CY},0} \in C^\infty_{\m{loc}}(\mZ^{\m{reg}}) \cap L^\infty_{\m{loc}}(\mZ)$ which remains bounded in a neighbourhood of $o$. On the other hand, $\mZ$ has an isolated three-dimensional ordinary double point singularity at $o$, and in particular $(\mZ,o)$ is normal, Gorenstein and terminal. Hence $o$ is an isolated Kawamata log terminal algebraic singularity \cite[Th. 2.4]{kollar1998real}. For such singularities there is a well-established version of Donaldson--Sun $2$-step degeneration process \cite{donaldson2014gromovi,li2019kahler,li2021algebraicity,li2018stability}, and as explained in the discussion after \cite[Th. 1.4]{zhang2024polynomial}, since the germ $(\mZ,o)$ is biholomorphic to the Calabi--Yau cone $\mZ$ itself the K-semistable and K-polystable Fano cones in this $2$-step degeneration must coincide with $\mZ$ and this implies the polynomial convergence of $\omega_{\m{CY},0}$ to $\omega_{\m{Stz}}$ in a holomorphic gauge.
\end{proof}

%%%%% End of Section 5 %%%%%
%%%%%%%%%%%%%%%%%%%%%%%%%%%%

%% file: section6_discussion.tex
%%%%%%%%%%%%%%%%%%%%%%%%%%%%%%%%%%
%%%%% Beginning of Section 6 %%%%%

    \section{Discussion and open questions}    \label{sec: Discussion and open questions}

    \subsection{Motivating adiabatic problem}       \label{subsec: motivating problem}

As mentioned in the introduction, the adiabatic problem motivating the present paper arises from a dimensional-reduction of the construction of compact $\m{G}_2$-manifolds due to Joyce and Karigiannis \cite{joyce2017new}, which has been recently extended by the second-named author \cite{majewski2025spin7orbifoldresolutions}. In its simplest form, this construction resolves a $\m{G}_2$-orbifold $M^7$ whose singular set $S^3$ is associative and has normal bundle $\H / \Z_2$ (where $\H$ is the space of quaternions), by gluing in a family of Eguchi--Hanson spaces depending on the data of a \emph{nowhere vanishing} harmonic $1$-from $\eta \in \Omega^1(S)$. Under local McKay duality, the direction of the $1$-form $\eta$ selects the hyperkähler complex structure on the Eguchi--Hanson space used to resolve the normal fibres, while its norm determines the size of the exceptional sphere. This produces a family of $\m{G}_2$-manifolds $N^7_t$ degenerating towards the original orbifold $M$ as the adiabatic parameter $t$ goes to $0$.

If the original $\m{G}_2$-orbifold is isometric to $\mathbb{S}^1 \times Y^6$, where $Y$ is a Calabi--Yau orbifold of complex dimension $3$ singular along a complex curve $\Sigma$, and the whole resolution data is invariant under a circle action, then the resulting $\m{G}_2$-manifolds will split a circle factor and be of the form $N_t = \mathbb{S}^1 \times Z_t$, where $Z_t$ is a family of Calabi--Yau manifolds degenerating towards the orbifold $Y$.

In this $\mathbb{S}^1$-invariant setting, a natural choice of nowhere-vanishing harmonic $1$-form $\eta$ on the singular set $S = \mathbb{S}^1 \times \Sigma$ is simply to take $\eta = \diff \theta$. This choice recovers a holomorphic resolution of $Y$ with respect to its given complex structure. For a general choice of resolution data, however, the degeneration of $Z_t$ towards $Y$ will usually involve deforming both the Kähler class and the complex structure jointly. In particular, when some connected components of the curve $\Sigma$ have genus greater than $1$, we can use a $1$-parameter family of $1$-forms $\eta_s = s \diff \theta + \Re(\alpha)$ where $\alpha \in H^0(K_\Sigma)$ is a non-trivial holomorphic $1$-form on $\Sigma$, and obtain a corresponding $2$-parameter family of Calabi--Yau manifolds $(Z_{s,t},I_{s,t},\omega_{s,t})$.  Note that $\alpha$ always admits zeroes, which are simple if $\alpha$ is sufficiently generic, and therefore $\eta_s$ is nowhere-vanishing if and only if $s \neq 0$. It is therefore natural to expect that conically singular Calabi--Yau manifolds should arise as directional limits of the two-parameter family of Calabi--Yau manifolds $(Z_{s,t},I_{s,t},\omega_{s,t})$.

The question motivating the construction of the present paper was to find a local model for such a singular limit near a simple zero of the holomorphic $1$-form $\alpha$, and to understand the geometry of the transition between $Z_{s,t}$ and $Z_{-s,t}$. In particular, we expect that the singular metric $\omega_{\m{CY},0}$ can be used in order to carry out the singular gluing construction, which should produce families of conically singular Calabi--Yau manifolds degenerating towards orbifolds, in analogy with Li's construction of collapsing Calabi--Yau metrics \cite{li2018gluing}.

    \subsection{Related questions}

\paragraph{Relation to the Conlon--Rochon metrics.} It seems likely that the singular Calabi--Yau metric $\omega_{\m{CY},0}$ coincides (at least up to biholomorphisms) with the singular metric constructed by Conlon--Rochon \cite{conlon2023degenerations} on $\mZ$. The latter is constructed by considering the smoothings of $\mZ$ instead of degenerating the Kähler class of the small resolutions $\mZ_\pm$. Their construction works for a wide class of affine varieties and uses rather different analytical machinery (namely, the notions of QAC manifolds with fibred corners and warped-QAC metrics). The downside is that the asymptotic description of the metrics that can be obtained using warped-QAC analysis is much less explicit than what we can achieve, making it difficult to compare the metrics asymptotically.

For the analogous question on $\C^3$, Sz\'ekelyhidi \cite{szekelyhidi2020uniqueness} proved the uniqueness (up to biholomorphisms) of a Calabi--Yau metric on $\C^3$ with tangent cone $\C \times (\C^2 / \Z_2)$ at infinity (or more generally, $\C \times A_1$ on $\C^n$), so one could possibly adapt these techniques to show that our singular metric on $\mZ$ coincides with Conlon--Rochon's. This would be very interesting, for it would yield a continuous `conifold transition' for Calabi--Yau manifolds with maximal volume growth and singular tangent cones at infinity, a phenomenon which has been extensively studied in the compact setting \cite{rong2011continuity,ruan2011convergence,song2015conjecture}. On the other hand, remark that the uniqueness property is highly non-trivial, and it already fails on $\C^3$ when the tangent cone at infinity is $\C \times A_2$, as was explicitly shown by Chiu \cite{chiu2022nonuniqueness}.

Similarly, when $s \neq 0$ it would be interesting to know if there is a unique complete Calabi--Yau metric on $\mZ_+$ with tangent cone $\C \times (\C^2/\Z_2)$ at infinity in a given Kähler class.

\paragraph{Generalisations to other quaternionic representations.} It is possible that the construction of this paper could be extended to other hyperkähler reductions of quaternionic representations. One natural class of representations would be the ones considered by Kronheimer \cite{kronheimer1989construction} in his general construction of the ALE gravitational instantons by hyperkähler reduction, which have tangent cone at infinity $\C^2 / \Gamma$ for some finite subgroup $\Gamma \subset \SU(2)$.
It is natural to ask whether the methods of this paper can be extended to construct complete Calabi--Yau threefolds with tangent cone at infinity $\mathbb C \times \left(\mathbb C^2/\Gamma\right)$. In particular for $\Gamma = \mathbb Z_{k+1}$, this would produce families of non-affine Calabi--Yau threefolds with tangent cone $\mathbb C\times A_k$ at infinity.

In this more general situation, we can still split the hyperkähler moment map into a real and a complex component upon fixing a preferred quaternionic complex structure, and the threefolds on which those metrics would live will depend on the choice of a base-point for the real moment map and the preimage of an affine complex line for the complex moment map. Note that the base space of the fibration induced by the moment map has a chambered structure, with chambers corresponding to the regular values, and therefore the threefolds obtained in this way would only be smooth for a suitably generic choice of parameters. When the base point of the real moment map approaches a wall, one expects the metrics to become singular, the limit $\omega_{\m{CY},0}$ constructed in this article being the simplest instance of this phenomenon for $\Gamma = \Z_2$. For more general groups $\Gamma$, the main new analytical difficulties are the absence of the symmetry available in the $A_1$-case and the possibility that several exceptional cycles might shrink simultaneously.

It is also possible to consider higher-dimensional generalisations of this picture beyond threefolds, with for instance Nakajima quiver varieties providing a large and particularly structured class of examples \cite{nakajima1994instantons}.

\paragraph{Resolutions of compact orbifolds.} The construction fits naturally into the universal resolution scheme for special-holonomy orbifolds developed in \cite{majewski2025spin7orbifoldresolutions}. Given a singular stratum, the scheme begins with its intrinsic normal cone bundle. Local McKay duality assembles the ALE resolutions of the fibres $\mathbb R^4/\Gamma$, together with their hyperkähler structures, into an equivariant universal resolution of the normal cone bundle. A particular resolution is obtained only afterwards, by pulling back this universal family along a section of the bundle of McKay parameters. Because the universal family is equivariant under the normaliser of $\Gamma$, this construction is natural and globalises over the possibly twisted normal cone bundle. In this sense, it is not merely a collection of fibrewise resolutions, but a universal resolution of the special-holonomy normal model.

Crucially, the scheme does not fix either a complex structure or a symplectic structure in advance. The fundamental resolution datum is the full hyperkähler structure on the resolved normal fibres. The same principle governs the Calabi--Yau construction in this paper. In the $A_1$-case, a transverse section of the McKay parameter through the discriminant produces the family $\mathcal Z_s$. Its natural geometric enhancement is therefore not a resolution in a fixed complex or symplectic category, but the family of complete Calabi--Yau structures. The two smooth sides give the two small resolutions related by the Atiyah flop, while the central fibre is the singular conifold. Thus the metrics constructed here provide the natural Calabi--Yau local model for extending the universal resolution of the normal cone bundle across the discriminant.

\paragraph{Exotic $\m{G}_2$-metrics on the universal Eguchi--Hanson family.} From the perspective of the universal resolution picture, it is natural to expect that there should be exotic metrics with exceptional holonomy analogues to the Calabi--Yau metrics constructed here.

In the $\m{G}_2$ setting, the local model near an isolated zero of the harmonic $1$-form encoding the resolution data \cite{joyce2017new,majewski2025spin7orbifoldresolutions} would presumably be an exotic torsion-free $\m{G}_2$-structure with an isolated conical singularity on the universal Eguchi--Hanson family, that is, the manifold $\M \cong (0,\infty) \times \C\mathbb{P}^3$ in our notations. One expects the fibration $\mu \colon \M \to \R^3$ to become coassociative asymptotically, with the induced fibrewise geometry approaching the hyperkähler Eguchi--Hanson metrics. However, neither coassociativity nor hyperkählerity of the fibres are necessarily expected in the interior. Although we do not pursue this question here, the desired local model in the $\m{G}_2$-case is closely analogous in spirit and purpose to the Calabi--Yau model constructed in this paper. To close this discussion, let us mention that there is ongoing work by Jason Lotay and Spiro Karigiannis investigating the $\m{G}_2$-case.

%%%%% End of Section 6 %%%%%
%%%%%%%%%%%%%%%%%%%%%%%%%%%%

%% file: appendixA_moseriteration.tex
    \section{Quantitative Moser iteration for \texorpdfstring{$\mathsf{SOB}(2m)$ $m$-folds}{SOB(2m) m-folds}}     \label{app: quantitative moser iteration}

In this appendix, we let $(M,\omega)$ be a non-compact K\"ahler manifold of complex dimension $m \geq 2$ satisfying the $\mathsf{SOB}(2m)$ property. We will prove two quantitative $L^\infty$-bounds underlying Proposition \ref{prop: L infinity bounds for us} and Proposition \ref{prop: Uniform decay of potential}, using Moser iteration as outlined in \cite{hein2010gravitational}.

    \subsection{Quantitative $L^\infty$-bounds}

By \cite[Prop. 3.2]{hein2010gravitational}, there exists a constant $A_1 > 0$ such that
\begin{equation}        \label{eq: Sobolev constant A1}
    \left( \int_M |u|^{\frac{2m}{m-1}} \omega^m \right)^{\frac{m-1}{m}} \leq A_1 \int_M |\diff u|^2 \omega^m , ~~~~ \forall u \in C^\infty_c(M) .
\end{equation}
Taking a cutoff function, it is not difficult to extend this inequality to $u \in L^2_1(M)$. Let us also fix in this section a function $\rho \geq 1$ on $M$ such that (outside some compact subset $K \subset M$) we have $\rho \asymp \mathrm{dist}(x_0,\cdot)$ for some base-point $x_0 \in M$, and $|\diff \rho| \lesssim 1$. By a result of Schoen--Yau, such a function always exists.

Then Hein proved that if $f \in C^\infty(M)$ is a function on $M$ whose derivatives of all orders are bounded\footnote{H\"older bounds for two derivatives would be sufficient, but in our situation we do have bounds on all derivatives, so we do not need to consider this more general case.} and such that $\sup \{\rho^\mu |f|\} < \infty$ for some $\mu > 2$, then there exists $u \in C^\infty(M)$, with derivatives of all orders bounded, solution of
\begin{equation}    \label{eq: MA for SOB(2m)}
    (\omega + \iq \partial \overline{\partial} u)^m = e^f \omega^m .
\end{equation}
His method is to show that one can solve the viscosity equation
\begin{equation}    
\label{eq: viscosity for SOB(2m)}
    (\omega + \iq \partial \overline{\partial} u_\varepsilon)^m = e^{f + \varepsilon u_\varepsilon} \omega^m
\end{equation}
where $u_\varepsilon \in C^\infty(M)$ with uniform bounds in $\varepsilon \in (0,1)$, and the solution $u$ of \eqref{eq: MA for SOB(2m)} is obtained as a limit of $u_\varepsilon$ using the Arzel\`a--Ascoli theorem. We claim the following result, which can be deduced from Hein's arguments and immediately implies Proposition \ref{prop: L infinity bounds for us}:

\begin{prop}        \label{prop: Quantitative L infinity bound}
    There exists a universal exponent $q_* > 2m$, depending only on $m$ and $\mu$, and a constant $Q_1 > 0$, depending only on $m$, $\mu$ and the constants $A_1$, $A_2$, $A_3$, where $A_1$ is as in \eqref{eq: Sobolev constant A1}, 
    \begin{equation*}
        A_2 = \sup \{ \rho^\mu |1-e^f|\}, ~~  A_3 = \int_M \rho^{-q_*} \omega^m,
    \end{equation*}
    such that for all $\varepsilon \in (0,1)$,  
    \begin{equation*}
        \|u_\varepsilon\|_{L^\infty} \leq Q_1, ~~ \text{and} ~~ \|u\|_{L^\infty} \leq Q_1 .
    \end{equation*}
\end{prop}

To prove this, let us recall a few facts about the viscosity solutions:
\begin{itemize}
    \item Each $u_\varepsilon \in C^\infty(M)$ is bounded, and has derivatives of any order bounded.
    \item The real $(1,1)$-forms $\omega_{u_\varepsilon} = \omega + \iq \partial \overline{\partial} u_\varepsilon$ are K\"ahler, and by the above they are comparable to $\omega$. 
\end{itemize}
As usual, all-important is the following identity, which can be obtained by multiplying \eqref{eq: viscosity for SOB(2m)} by $u_{\varepsilon} |u_\varepsilon|^{p-2}$:
\begin{equation*}
    u_\varepsilon |u_\varepsilon|^{p-2} (e^{f+\varepsilon u_\varepsilon} - 1) \omega^m = \frac{1}{2}u_\varepsilon |u_\varepsilon|^{p-2} \diff\diff^c u_\varepsilon \wedge T_{u_{\varepsilon}}
\end{equation*}
where $T_{u_\varepsilon} = \omega^{m-1} + \omega^{m-2} \wedge \omega_{u_\varepsilon} + \cdots + \omega_{u_\varepsilon}^{m-1}$. Multiplying the previous identity by a compactly supported function $v \in C^\infty_c(M)$ which is non-negative and integrating by parts one obtains the inequality
\begin{multline}        \label{eq: master inequality}
    \int_M v |\diff |u_\varepsilon|^{\frac{p}{2}}|^2 \omega^m + \frac{mp^2}{2(p-1)} \int_M v u_\varepsilon |u_\varepsilon|^{p-2} (e^{\varepsilon u_\varepsilon}-1) e^f \omega^m \\ 
        \leq - \frac{mp^2}{2(p-1)} \left( \int_M v u_\epsilon |u_\varepsilon|^{p-2} (e^f-1) \omega^m + \frac{1}{2} \int_M u_\varepsilon |u_\varepsilon|^{p-2} \diff v \wedge \diff^c u_\varepsilon \wedge T_{u_\varepsilon} \right)
\end{multline}
for any $p > 1$. Crucially, the second term on the left-hand side of the inequality is non-negative, and moreover there exists a constant $C_\varepsilon$ depending on $\varepsilon$ and $\|u_\varepsilon\|_{L^\infty}$ (which we know is finite, but at this point we do not have a uniform bound), such that $u_\varepsilon(e^{\varepsilon u_\varepsilon}-1) \geq C_\varepsilon |u_\varepsilon|^2$. Taking this into account and applying this inequality to the functions $\chi(\frac{\rho}{R})\rho^k$ for any $k \in \Z$, and letting $R \to \infty$ one obtains
\begin{equation}
    \int_M \rho^k |u_\varepsilon|^p \omega^m \leq C_{k,\varepsilon} \left( \int_M \rho^k |u_\varepsilon|^{p-1}|e^f-1| \omega^m + \int_M \rho^{k-1} |u_\varepsilon|^{p-1} \omega^m  \right) 
\end{equation}
for any $k \in \R$ and $p > 1$, where the constant $C_{k,\varepsilon}$ may depend on $k$ and $\varepsilon$, and we have used that $\|\diff \rho\|_{L^\infty}$, $\|\diff u_\varepsilon\|_{L^\infty}$ and $\|T_{u_\varepsilon}\|_{L^\infty}$ are all finite. From this one may first deduce that the functions $u_\varepsilon$ are very integrable in the following sense\footnote{In \cite[Sec. 4.3]{hein2010gravitational}, it is temporarily assumed that $f$ is compactly supported to prove much more integrability, but this is not needed for our purpose.}:

\begin{lem}     \label{lem: integrability of viscosity solutions}
    For any $\varepsilon \in (0,1)$, $p > 4m$ and $k \leq \frac{1}{2}$, $\int |u_\varepsilon|^p \rho^k \omega^m < \infty$.
\end{lem}

\begin{proof}
    By assumption $|e^f - 1| \lesssim \rho^{-\mu}$ and thus (since $\mu > 1$) using the H\"older's inequality we have
    \begin{equation*}
        \int_M \rho^k |u_\varepsilon|^p \omega^m  \lesssim \int_M \rho^{k-1} |u_\varepsilon|^{p-1} \omega^m \leq \left( \int_M \rho^{\frac{p(2k-1)}{2(p-1)}} |u_\varepsilon|^p \omega^m\right)^{\frac{p-1}{p}} \left( \int_M \rho^{- \frac{p}{2}} \omega^m \right)^{\frac{1}{p}}
    \end{equation*}
    and since $\rho^{-\frac{p}{2}}$ is integrable (as $p > 4m$) and $k \leq \frac{1}{2}$, we deduce that
    \begin{equation*}
        \int_M \rho^k |u_\varepsilon|^p \omega^m \lesssim \left( \int_M \rho^{k-\frac{1}{2}} |u_\varepsilon|^p \omega^m \right)^{1-\frac{1}{p}} .
    \end{equation*}
    On the other hand since $\|u_\varepsilon\|_{L^\infty} < \infty$ there exists $k_0 \in \frac{1}{2}\Z$ negative enough that $\rho^k |u_\varepsilon|^p$ is integrable, and the result follows easily by ascending induction from $k_0$ up to $k = \frac{1}{2}$.
\end{proof}

\begin{cor}
    For any $\varepsilon \in (0,1)$ and $p > 4m+1$, 
    \begin{equation*}
        \int_M |\diff |u_\varepsilon|^{\frac{p}{2}}|^2 \omega^m \leq \frac{mp^2}{2(p-1)} \int_M u_\varepsilon |u_\varepsilon|^{p-2} (1-e^f) \omega^m .
    \end{equation*}
    In particular, $|u_\varepsilon|^\frac{p}{2} \in L^2_1(M,\omega)$ for any $p > 4m+1$.
\end{cor}

\begin{proof}
    Apply \eqref{eq: master inequality} with the function $v = \chi(\frac{\rho}{R})$ and keeping only the first term on the left-hand side (which we can do since the second term is non-negative) one obtains, for any $R > 1$,
    \begin{equation*}
        \int_M \chi(\tfrac{\rho}{R}) |\diff |u_\varepsilon|^{\frac{p}{2}}|^2 \omega^m \leq \frac{mp^2}{2(p-1)} \int_M \chi(\tfrac{\rho}{R})u_\varepsilon |u_\varepsilon|^{p-2} (1-e^f) \omega^m + a_R
    \end{equation*}
    where 
    \begin{equation*}
        |a_R| \leq  \frac{C}{R} \int_{\rho \leq R} |u_\varepsilon|^{p-1} \omega^m,
    \end{equation*}
    $C$ being a constant depending on $p$, $\|\nabla u_\varepsilon\|_{L^\infty}$, $\|T_{u_\varepsilon}\|_{L^\infty}$ and $\sup |\chi^\prime|$ but not on $R$. But since $p > 4m+1$, the previous lemma implies that $|u_\varepsilon|^p$ is integrable and thus $a_R \to 0$ as $R \to 0$. Given the boundedness of $f$ this proves the lemma.
\end{proof}

Thanks to the previous lemma, we can use \eqref{eq: Sobolev constant A1} and the inequality $|e^f-1| \leq A_2 \rho^{-\mu}$ to obtain, for any $p > 4m +1$,
\begin{equation}    \label{eq: Ineq for iteration}
    \begin{aligned}
        \left( \int_M |u_\varepsilon|^{\frac{mp}{m-1}} \omega^m \right)^{\frac{m-1}{m}} \leq A_1 \int_M |\diff |u_\varepsilon|^{\frac{p}{2}}|^2 \omega^m & \leq \frac{mp^2A_1}{2(p-1)} \int_M |u_\varepsilon|^{p-1} |e^f-1| \omega^m \\
        & \leq \frac{mp^2 A_1A_2}{2(p-1)} \int_M |u_\varepsilon|^{p-1} \rho^{-\mu} \omega^m .
    \end{aligned}
\end{equation}

\begin{lem}     \label{lem: constant Q2}
    There exist exponents $p_* > 4m+1$, $q_* > 2m$, depending only on $m$ and $\mu$, and a constant $Q^\prime_1$ depending only on $m$, $\mu$, $A_1$, $A_2$, and $A_3$ (defined as in Proposition \ref{prop: Quantitative L infinity bound} for this choice of $q_*$) such that for any $\varepsilon \in (0,1)$,
    \begin{equation*}
        \|u_\varepsilon\|_{L^{p_*}} \leq Q^\prime_1 .
    \end{equation*}
\end{lem}

\begin{proof}
    Because we assumed $\mu > 2$, there exists $p_0 > 4m+1$ such that the conjugate exponent $q_0$ of $\frac{mp_0}{(m-1)(p_0-1)}$ satisfies $q_0 > \frac{2m}{\mu}$. Thus if we let $q_* = \mu q_0 > 2m$, then the inequality \eqref{eq: Ineq for iteration} together with H\"older's imply
    \begin{align*}
        \left( \int_M |u_\varepsilon|^{\frac{mp_0}{m-1}} \omega^m \right)^{\frac{m-1}{m}} & \leq \frac{mp^2_0 A_1A_2}{2(p_0-1)} \int_M |u_\varepsilon|^{p_0-1} \rho^{-\mu} \omega^m \\
            &\leq \frac{mp_0^2A_1A_2A_3^{\frac{1}{q_0}}}{2(p_0-1)} \left( \int_M |u_\varepsilon|^{\frac{mp_0}{m-1}} \omega^m \right)^{\frac{m-1}{m}\cdot \frac{p_0-1}{p_0}} .
    \end{align*}
    Choosing $p_* = \frac{mp_0}{m-1} > p_0 > 4m+1$ and $Q^\prime_1$ as the constant in front of the integral in the second line above, we obtained the desired result.
\end{proof}

We can now prove the quantitative $L^\infty$-bounds claimed earlier:

\begin{proof}[Proof of Proposition \ref{prop: Quantitative L infinity bound}]
    By the previous lemma, $\|u_\varepsilon\|_{L^{p_*}} \leq Q^\prime_1$. On the other hand, since $\rho \geq 1$ we deduce by applying \eqref{eq: Ineq for iteration} (and swapping $p$ for $p+1$) that for any $p \geq p_*$,
    \begin{equation*}
        \|u_\varepsilon\|_{L^{\frac{m(p+1)}{m-1}}} \leq \left( \frac{m(p+1)^2A_1A_2}{2p}\right)^{\frac{1}{p+1}} \|u_\varepsilon\|_{L^p}^\frac{p}{p+1} .
    \end{equation*}
    So the result follows by a classical induction starting at $p = p_*$.
\end{proof}

    \subsection{Weighted $L^\infty$-bounds}

Let us now introduce the following additional data and assumptions:
\begin{itemize}
    \item $K \subset M$ is a compact subset with non-empty interior and smooth boundary in $M$, such that the weight function $\rho$ satisfies $\left. \rho \right|_K \equiv 1$.
    \item There exists a constant $A_4 > 0$ such that $\|\diff \diff^c u_\varepsilon\|_{C^0(M \backslash K)} \leq A_4$ for the viscosity solutions on $M \backslash K$, for any $\varepsilon \in (0,1)$.
    \item There exists a constant $A_5 > 0$ such that $|\diff \rho|_\omega +\rho |\diff \diff^c \rho |_\omega \leq A_5$.
\end{itemize}
Throughout this part, we shall also fix $\nu_0 \in (0, \min\{1,\mu-2\})$. Then in this general setup, we prove the following proposition, which implies Proposition \ref{prop: Uniform decay of potential}:

\begin{prop}    \label{prop: weighted L infinity bound}
    In the situation described above, there exists a universal exponent $\tilde{q}_* > 2m$, depending only on $m$, $\mu$ and $\nu_0$, and a constant $Q_{\nu_0} > 0$, depending only on $m$, $\mu$, $\nu_0$ and the constants $A_1$, $A_2$, $A_3$ (defined in Section \ref{subsec: L infinity bounds}), $A_4$, $A_5$ (defined above), and 
    \begin{equation*}
        A_6 \coloneqq \int \rho^{-\tilde{q}_*} \omega^m,
    \end{equation*}
    such that the solution $u$ of \eqref{eq: MA for SOB(2m)} satisfies the weighted $L^\infty$ bound 
    \begin{equation*}
        \|\rho^{\nu_0} u\|_{L^\infty} \leq Q_{\nu_0} .
    \end{equation*}
\end{prop}

We first need to adapt Lemma \ref{lem: integrability of viscosity solutions} to a statement on the weighted integrability of the viscosity solutions.

\begin{lem}     \label{integrability of the viscosity solutions, II}
    There exists $p_0 > 4m$, depending only on $m$ and $\nu_0$, such that $\|\rho^\nu u_\varepsilon\|_{L^p} < \infty$ for all $\varepsilon > 0$, $\nu \leq \nu_0$ and $p \geq p_0$.
\end{lem}

\begin{proof}
    Let $\delta \coloneqq \frac{1-\nu_0}{2}$ and let $p_{\nu_0} \coloneqq \frac{2m}{\delta} = \frac{4m}{1-\nu_0} > 4m$. Let us fix $p > p_{\nu_0}$; in particular, $\rho^{-\delta p}$ is integrable. Now as in the proof of Lemma \ref{lem: integrability of viscosity solutions} we have, for any $\nu \leq \nu_0$,
    \begin{align*}
        \int_M \rho^{p \nu} |u_\varepsilon|^p & \lesssim \int_M \rho^{p\nu-1} |u_\varepsilon|^{p-1} \omega^m \\
            & \lesssim \int_M \rho^{(p-1)(\nu-\tfrac{\delta}{p})} \rho^{\nu-1+\delta} |u_\varepsilon|^{p-1} \omega^m  \\
            & \lesssim \int_M \rho^{(p-1)(\nu-\tfrac{\delta}{p})} \rho^{-\delta} |u_\varepsilon|^{p-1} \omega^m \\
            & \lesssim \left( \int_M \rho^{p(\nu-\tfrac{\delta}{p})}|u_\varepsilon|^p \omega^m \right)^{1-\frac{1}{p}} \left( \int_M \rho^{-p\delta} \omega^m \right)^{\frac{1}{p}}
    \end{align*}
    and since for $\nu$ negative enough $\rho^{\nu}u_\varepsilon$ is $L^p$, the result follows by ascending induction (and we choose for instance $p_0 \coloneqq p_{\nu_0} + 1$).
\end{proof}

Before stating the next lemma, let us remark that if $p \geq 2$ and $v \in C^\infty_c(M)$ is a smooth compactly supported function, we have (using the fact that $T_{u_\varepsilon}$ is a $(m-1,m-1)$-form)
\begin{align*}
    p \int u_\varepsilon |u_\varepsilon|^{p-2} \diff v \wedge \diff^c u_\varepsilon \wedge T_{u_\varepsilon} = \int \diff v \wedge \diff^c (|u_\varepsilon|^p) \wedge T_{u_\varepsilon} & = \int \diff(|u_\varepsilon|^p) \wedge \diff^c v \wedge T_{u_\varepsilon} \\
        & = - \int |u_\varepsilon|^p \diff\diff^c v \wedge T_{u_\varepsilon}
\end{align*}
where in the last line we integrate by parts using $\diff T_{u_\varepsilon} = 0$.
Thus we can rewrite the inequality \eqref{eq: master inequality} as
\begin{multline}        \label{eq: second master inequality}
    \int_M v |\diff |u_\varepsilon|^{\frac{p}{2}}|^2 \omega^m + \frac{mp^2}{2(p-1)} \int_M v u_\varepsilon |u_\varepsilon|^{p-2} (e^{\varepsilon u_\varepsilon}-1) e^f \omega^m \\ 
        \leq - \frac{mp^2}{2(p-1)} \int_M v u_\epsilon |u_\varepsilon|^{p-2} (e^f-1) \omega^m - \frac{mp}{4(p-1)} \int_M |u_\varepsilon|^p \diff \diff^c v \wedge T_{u_\varepsilon} .
\end{multline}
Using the previous inequality we shall deduce:

\begin{lem} \label{lem: constant A_8}
    There exists a constant $A_7 > 0$, depending only on $m$, $A_4$ and $A_5$, such that for any $\nu \leq \nu_0$ and $p \geq p_{0}$,
    \begin{multline*}
        \int_M |\diff (\rho^{\frac{p\nu}{2}} |u_\varepsilon|^{\frac{p}{2}})|^2 \omega^m \leq \frac{mp^2 A_1}{2(p-1)} \int |u_\varepsilon|^{p-1} \rho^{p\nu-\mu} \omega^m \\ + p^2 \nu^2 A_7 \left( 1 + \frac{mp}{4(p-1)} \right) \int \rho^{p\nu-2} |u_\varepsilon|^p \omega^m .
    \end{multline*}
    In particular, $\rho^{\frac{p\nu}{2}} |u_\varepsilon|^\frac{p}{2} \in L^2_1(M,\omega)$.
\end{lem}

\begin{proof}
    Let us consider the compactly supported function $v_R = \chi(\frac{\rho}{R})\rho^{\nu}$. Then we have
    \begin{equation*}
        \int_M |\diff ((v_R|u_\varepsilon|)^{\frac{p}{2}})|^2 \omega^m \leq 2 \int |u_\varepsilon|^p |\diff v_R^{\frac{p}{2}}|^2 \omega^m +2 \int v_R^p |\diff |u_\varepsilon|^{\frac{p}{2}}|^2 \omega^m.
    \end{equation*}
    Let us estimate each term on the right-hand side separately.

    For the first term, we have
    \begin{align*}
        |\diff v_R^{p/2}|^2 & \leq 2 \rho^{p\nu} |\diff (\chi(\tfrac{\rho}{R})^{\frac{p}{2}})|^2 + \frac{p^2\nu^2}{2} \chi(\tfrac{\rho}{R})^p |\diff \rho|^2 \\
            & \leq \frac{p^2\nu^2}{2}(R^{-2}|\chi^\prime(\tfrac{\rho}{R})|^2\rho^{p\nu} + \rho^{p\nu-2}) |\diff \rho|^2
    \end{align*}
    and thus we obtain
    \begin{equation*}
        2 \int |u_\varepsilon|^p |\diff v_R^{\frac{p}{2}}|^2 \omega^m \leq  p^2\nu^2 \int \rho^{p\nu-2} |u_\varepsilon|^p |\diff \rho|^2 \omega^m + a_R
    \end{equation*}
    where
    \begin{equation*}
        a_R \coloneqq \frac{p^2\nu^2}{R^2} \int \rho^{p\nu} |u_\varepsilon|^p |\chi^\prime(\tfrac{\rho}{R})|^2 \omega^m \to 0
    \end{equation*}
    as $R \to \infty$ since $\rho^{p\nu} |u_\varepsilon|^p$ is integrable.

    Now for the second term we use \eqref{eq: second master inequality} to obtain
    \begin{equation*}
        \int v_R^p |\diff |u_\varepsilon|^\frac{p}{2}| \leq \frac{mp^2A_1}{2(p-1)} \int \rho^{p\nu-\mu} |u_\varepsilon|^{p-1} \omega^m  + \frac{mp}{4(p-1)} \int |u_\varepsilon|^p |\diff \diff^c v_R^p|_\omega |T_{u_\varepsilon}|_\omega \omega^m .
    \end{equation*}
    Now (at least for $R$ large enough) $|\diff\diff^c(v^p_R)|$ is supported in $M \backslash K$, where we have a uniform estimate $\|\diff\diff^c u_\varepsilon\|_{C^0(M\backslash K)} \leq A_4$, so that there exists a constant $A^\prime_4 > 0$, depending only on $m$ and $A_4$, such that $|T_{u_\varepsilon}|_\omega \leq A^\prime_4$ on $M \backslash K$. Using the integrability of $\rho^{p\nu}|u_\varepsilon|^p$ and the bound $|\diff \rho| + \rho |\diff \diff^c \rho| \leq A_5$, it is not difficult to conclude that 
    \begin{equation*}
        \int |u_\varepsilon|^p |\diff \diff^c v_R^p|_\omega |T_{u_\varepsilon}|_\omega \omega^m \leq p^2\nu^2 A^\prime_4A_5  \int |u_\varepsilon|^p \rho^{p\nu-2} \omega^m + b_R
    \end{equation*}
    where $b_R$ is a boundary term such that $b_R \to 0$ as $R \to \infty$. Gathering the previous estimates and letting $R \to \infty$ this proves the lemma.
\end{proof}

\begin{lem}
    There exists $\nu_* \in (0,\nu_0)$ and exponents $\tilde{p}_*  \geq p_0$, $\tilde{q}_* > 2m$, depending only on $m$, $\mu$ and $\nu_0$, and a constant $Q_3$, depending on only $m$, $\nu_0$, $A_1,\ldots,A_7$ (where $A_6$ is defined for this choice of $\tilde{q}_*$) such that for all $\varepsilon > 0$,
    \begin{equation*}
        \|\rho^{\nu_*} u_\varepsilon\|_{L^{\tilde{p}_*}} \leq Q_3 .
    \end{equation*}
\end{lem}

\begin{proof}
    Let $\nu \in (0,\nu_0)$ to be chosen later and $p \geq p_0$. Then the previous lemma, the assumption $\nu \leq \nu_0 < \mu-2$, the Sobolev inequality \eqref{eq: Sobolev constant A1} and H\"older's inequality imply: 
    \begin{align*}
        \left(\int (\rho^\nu |u_\varepsilon|)^{\frac{m p}{m-1}} \omega^m \right)^{\frac{m-1}{m}} & \leq C_1 \int (\rho^\nu |u_\varepsilon|)^{p-1} \rho^{-(\mu-\nu_0)} \omega^m + C_2 \int \rho^{p\nu-2} |u_\varepsilon|^p \omega^m \\
            & \leq C_1 \left( \int \rho^{q(\mu-\nu_0)} \omega^m \right)^{\frac{1}{q}} \left( \int (\rho^\nu |u_\varepsilon|)^{\frac{mp}{m-1}} \omega^m\right)^{\frac{m-1}{m}\cdot \frac{p-1}{p}} \\
            & ~~~~ + C_2 \int (\rho^{\nu- \frac{2}{p}} |u_\varepsilon|)^p \omega^m
    \end{align*}
    where $q$ is the conjugate exponent of $\frac{mp}{(m-1)(p-1)}$ and $C_1, C_2 >0$ constants which can depend on $m$, $p$, $\nu$, any of the constants $A_1,\ldots,A_5$ and $A_7$, but not $\varepsilon$. As in the proof of Lemma \ref{lem: constant Q2}, because $\mu-\nu_0 > 2$ we may choose $p \geq p_0$ large enough that $\tilde{q}_* \coloneqq q(\mu-\nu_0) > 2m$, and thus
    \begin{equation*}
        \left(\int (\rho^\nu |u_\varepsilon|)^{\frac{m p}{m-1}} \omega^m \right)^{\frac{m-1}{m}} \leq C^\prime_1 \left( \int (\rho^\nu |u_\varepsilon|)^{\frac{m p}{m-1}} \omega^m\right)^{\frac{m-1}{m} \cdot \frac{p-1}{p}} + C_2 \int (\rho^{\nu- \frac{2}{p}} |u_\varepsilon|)^p \omega^m
    \end{equation*}
    where $C^\prime_1$ depends on $m$, $\mu$, $\nu_0$, $p$, $A_1,\ldots,A_6$ but not on $\varepsilon$. Now choosing $\tilde{p}_* \coloneqq \frac{mp}{m-1}$ and $\nu_* = \frac{2}{p}$ for this value of $p$ and using the fact that $\|u_\varepsilon\|_{L^\infty} \leq Q_1$ for all $\varepsilon > 0$ (at least if we chose $p$ large enough, which we can always assume), this proves the lemma.
\end{proof}

We can now prove the uniform weighted $L^\infty$-bound:

\begin{proof}[Proof of Proposition \ref{prop: weighted L infinity bound}]
    We first prove a sub-optimal uniform bound $\|\rho^{\nu_*}u_\varepsilon\|_{L^\infty} < Q_{\nu_*}$. Using the uniform $L^\infty$-bound $\|u_\varepsilon\|_{L^\infty} \leq Q_1$ and the fact that $\rho \geq 1$, it is easy to deduce from Lemma \ref{lem: constant A_8} that for any $p \geq \tilde{p}_*$,
    \begin{equation*}
        \|\rho^{\nu_*} u_\varepsilon\|_{L^{\frac{m(p+1)}{m-1}}} \leq \left( \frac{m(p+1)^2A_1}{2p} + (p+1)^2\nu^2_* A_7 Q_1 \left(1+ \frac{m(p+1)}{4p}\right)\right)^{\frac{1}{p+1}} \|\rho^{\nu_*} u_\varepsilon\|_{L^p}^{\frac{p}{p+1}} .
    \end{equation*}
    Using the previous lemma, the uniform weighted bound $\|\rho^{\nu_*} u_\varepsilon\|_{L^\infty} < Q_{\nu_*}$ easily follows by induction starting from $\tilde{p}_*$ as in the proof of Proposition \ref{prop: Quantitative L infinity bound}.

    We now improve this to find a uniform bound on $\|\rho^{\nu_0} u_\varepsilon\|_{L^\infty}$. Using the previous bound, we may deduce from Lemma \ref{lem: constant A_8} and the uniform Sobolev bound \eqref{eq: Sobolev constant A1} that for any $p \geq p_0$, 
    \begin{equation*}
        \left( \int (\rho^{\nu_0} |u_\varepsilon|)^{\frac{mp}{m-1}} \omega^m \right)^{\frac{m-1}{m}} \leq C \int \rho^{-2-\delta} (\rho^\nu |u_\varepsilon|)^{p-1} \omega^m
    \end{equation*}
    where $\delta \coloneqq \min \{\nu_*, \mu-2-\nu_0\}$ and $C$ is a constant which may depend on $m$, $\mu$, $\nu_0$, $\nu_*$, $p$, the constants $A_i$, and $Q_{\nu_*}$, but not on $\varepsilon$. Since $\delta > 0$, we may chose $p$ large enough such that the conjugate exponent $q$ of $\frac{mp}{(m-1)(p-1)}$ satisfies $(2+\delta)q > 2m$, and by Holder's inequality we deduce that $\|\rho^{\nu_0} u_\varepsilon\|_{L^{\frac{mp}{m-1}}}$ has a uniform bound. Then Proposition \ref{prop: weighted L infinity bound} follows by induction starting at this value of $p$, just as we did above.
\end{proof}